\documentclass{article}
\pdfoutput=1

\usepackage[utf8]{inputenc}
\usepackage[english]{babel}

\usepackage{amsmath,amssymb,amsfonts,amsthm}
\usepackage{fullpage}

\usepackage{graphicx,caption,subcaption}
\usepackage{mathtools}
\mathtoolsset{showonlyrefs}

\usepackage{gensymb}
\usepackage[misc]{ifsym}

\newtheorem{prop}{Proposition}

\author{Iveta~Hn\v{e}tynkov\'{a}%
  \thanks{Faculty of Mathematics and Physics, Charles University, Prague, Czech Republic. Electronic address: \texttt{hnetynko@karlin.mff.cuni.cz}}
  \and Marie~Kub\'{i}nov\'{a}%
  \thanks{Faculty of Mathematics and Physics, Charles University, Prague, Czech Republic, and Institute of~Computer Science, The Czech Academy of~Sciences, Prague, Czech~Republic. Electronic address: \texttt{kubinova@karlin.mff.cuni.cz} \Letter}
\and Martin~Ple\v{s}inger%
\thanks{Faculty of Education, Technical University of Liberec, Liberec, Czech~Republic. Electronic address: \texttt{martin.plesinger@tul.cz} }}
\date{}

\title{Noise representation in residuals of LSQR, LSMR, and CRAIG regularization}

\begin{document}
\maketitle

\begin{abstract}
Golub-Kahan iterative bidiagonalization represents the core algorithm in several regularization methods for solving large linear noise-polluted ill-posed problems. We consider a general noise setting and derive explicit relations between (noise contaminated) bidiagonalization vectors and the residuals of bidiagonalization-based regularization methods LSQR, LSMR, and CRAIG. For LSQR and LSMR residuals we prove that the coefficients of the linear
combination of the computed bidiagonalization vectors reflect the amount of propagated noise in each of these vectors. For CRAIG the residual is only a multiple of a particular bidiagonalization vector. We show how its size indicates the regularization effect in each iteration by expressing the CRAIG solution as the exact solution to a modified compatible problem. Validity of the results for larger two-dimensional problems and influence of the loss of orthogonality is also discussed.
\end{abstract}



\section{Introduction}\label{sec:introduction}

In this paper we consider ill-posed linear algebraic problems of the form
\begin{equation}
  b =Ax + \eta,\qquad A\in\mathbb{R}^{m\times n},
  \qquad b\in\mathbb{R}^{m},\qquad
  {\|\eta\|\ll\|Ax\|,}
  \label{eq:theproblem}
\end{equation}
where the matrix $A$ represents a~discretized smoothing operator with the singular values decaying gradually to zero without a~noticeable gap. We assume that multiplication of a~vector $v$ by $A$ or $A^{T}$ results in smoothing which reduces the relative size of the high-frequency components of $v$. The operator $A$ and the vector $b$ are supposed to be known. The vector $\eta$ represents errors, such as \emph{noise}, that affect the exact data. Problems of this kind are commonly referred to as linear discrete ill-posed problems or linear inverse problems and arise in a~variety of applications \cite{Hansen2010Discrete,Hansen1998Rank}. Since $A$ is ill-conditioned, the presence of noise makes the naive solution
\begin{equation}
x^{\text{naive}}\equiv A^{\dagger}b,
\end{equation}
where $A^\dagger$ denotes the Moore-Penrose pseudoinverse, meaningless. Therefore, to find an acceptable numerical approximation to $x$, it is necessary to use regularization methods.

Various techniques to regularize the linear inverse problem \eqref{eq:theproblem} have been developed. For large-scale problems, iterative regularization is a~good alternative to
direct regularization methods. When an iterative method is used, regularization is achieved by early termination of the process, before noise $\eta$
starts to dominate the approximate solution \cite{Hansen2010Discrete}. Many iterative regularization methods such as LSQR \cite{Paige1982LSQR,Bjoerck1988bidiagonalization,Saunders1997Computing,Jensen2007Iterative},  CRA\-IG \cite{Craig1955N,Saunders1995Solution}, LSMR \cite{Fong2011LSMR}, and CRAIG-MR/MRNE \cite{Arioli2013Iterative,Morikuni2013Inner} involve the Golub-Kahan iterative bidiagonalization \cite{Golub1965Calculating}. Combination with an additional inner regularization (typically with a~spectral filtering method) gives so-called hybrid regularization; see, for example, \cite{Bjoerck1988bidiagonalization,Hanke2001Lanczos,Kilmer2001Choosing,Chung2015hybrid}.
Various approaches for choosing the stopping criterion, playing here the role of the regularization parameter, are based on comparing the properties of the actual residual to an a priori known property of noise, such as the noise level in the Morozov's discrepancy principle \cite{Morozov1966solution}, or the noise distribution in the cumulative residual periodogram method \cite{Rust2000Parameter,Rust2008Residual,Hansen2006Exploiting}. Thus understanding how noise translates to the residuals during the iterative process is of great interest.

The aim of this paper is, using the analysis of the propagation of noise in the left bidiagonalization vectors provided in \cite{Hnetynkova2009regularizing}, to study the relation between residuals of bidiagonali\-za\-tion-based methods and the noise vector $\eta$. Whereas in \cite{Hnetynkova2009regularizing}, white noise was assumed, here we have no particular assumptions on the distribution of noise. 
We only assume the amount of noise is large enough to make the noise propagation visible through the smoothing by $A$ in construction
of the  bidiagonalization vectors. This is often the case in ill-posed problems, as we illustrate on one-dimensional (1D) as well as significantly noise contaminated two-dimensional (2D) benchmarks. 
We prove that LSQR and LSMR residuals are given by a linear
combination of the bidiagonalization vectors with the coefficients 
related to the amount of propagated noise in the corresponding vector.
For CRAIG, the residual is only a multiple of a particular bidiagonalization vector.
This allows us to prove that an approximate solution obtained  in a given iteration by
CRAIG applied to \eqref{eq:theproblem} coincides with an exact solution of
the (compatible) modified problem
\begin{equation}
Ax = b-\tilde{\eta}\label{eq:transformed},
\end{equation}
where $\tilde{\eta}$ is a noise vector estimate constructed from the currently computed bidiagonalization vectors.
These results contribute to understanding of regularization properties of the considered methods and should be considered when devising reliable 
stopping criteria.

Note that since LSQR is mathematically equivalent to CGLS and CGNR, CRAIG is mathematically equivalent to CGNE and CGME \cite{Saad2003Iterative}, and LSMR is mathematically equivalent to CRLS \cite{Fong2011LSMR}, then in exact arithmetic, the analysis applies also to these methods.

The paper is organized as follows. In Section~\ref{sec:preliminaries}, after a recollection of the previous results, we study the propagation of various types of noise and the influence of the loss of orthogonality on this phenomenon. Section~\ref{sec:residuals_noise} investigates the residuals of selected methods with respect to the noise contamination in the left bidiagonalization vectors and compares their properties. Section~\ref{sec:2Dproblems} discusses validity of obtained results for larger 2D problems.
Section~\ref{sec:conclusions} concludes the paper.

Unless specified otherwise, we assume exact arithmetic and the presented experiments are performed with full double reorthogonalization in the bidiagonalization process.
Throughout the paper, $\|v\|$ denotes the standard Euclidean norm of the vector $v$, vector $e_k$ denotes the $k$-th column of the identity matrix. By $\mathcal{P}_k$, we denote the set of polynomials of degree less or equal to $k$. The noise level is denoted by
$\delta_\text{noise}\equiv \|\eta\|/\|Ax\|$. By Poisson noise, we understand $b_i\sim \text{Pois}([Ax]_i)$, i.e., the right-hand side $b$ is a Poisson random with the Poisson parameter $Ax$. The test problems were adopted from the Regularization tools \cite{Hansen2007Regularization}. For simplicity of exposition, we assume the initial approximation $x_0\equiv 0$ throughout the paper. Generalization to $x_0\neq0$ is straightforward.

\section{Properties of the Golub-Kahan iterative bidiagonalization} \label{sec:preliminaries}

\subsection{Basic relations}
Given the initial vectors
$w_{0}\equiv0$, $s_{1}\equiv b/\beta_{1}$, where $\beta_{1}\equiv\|b\|$,
the Golub-Kahan iterative bidiagonalization \cite{Golub1965Calculating} computes, for $k = 1,2,\ldots$,
\begin{align}
\alpha_{k}w_{k} & =A^{T}s_{k}-\beta_{k}w_{k-1}\,, & \|w_{k}\| & =1,\\
\beta_{k+1}s_{k+1} & =Aw_{k}-\alpha_{k}s_{k}\,, & \|s_{k+1}\| & =1,\label{eq:bid_alg}
\end{align}
until $\alpha_{k}=0$ or $\beta_{k+1}=0$, or until $k=\min(m,n)$.  Vectors $s_1,\ldots,s_k$, and $w_1,\ldots,w_k$, form orthonormal bases of the Krylov subspaces $\mathcal{K}_k(AA^{T},b)$ and $\mathcal{K}_k(A^{T}A,A^{T}b)$, respectively. In the rest of the paper, we assume that the bidiagonalization process does not terminate before the iteration $k + 1$, i.e., $\alpha_l,\beta_{l+1} >0$, $l = 1,\ldots k.$

Denoting $S_k\equiv[s_{1},\ldots,s_k]\in\mathbb{R}^{m\times k}, \, W_k\equiv[w_{1},\ldots,w_k]\in\mathbb{\mathbb{R}}^{n\times k}$  and
\begin{equation}
  L_k\equiv\left[\begin{array}{cccc}
    \alpha_1\\\beta_2&\alpha_2\\&\ddots&\ddots\\&&\beta_k&\alpha_k
  \end{array}\right]\in\mathbb{R}^{k\times k},
  \qquad
  L_{k+}\equiv\left[\begin{array}{c}L_k\\e_k^T\beta_{k+1}\end{array}\right]\in\mathbb{R}^{(k+1)\times k},
  \label{eq:Lk}
\end{equation}
we can write the matrix version of the bidiagonalization as
\begin{equation}
  A^TS_k = W_kL_k^T,
  \qquad
  AW_k = S_{k+1}L_{k+}. \label{eq:GKmat}
\end{equation}
The two corresponding Lanczos three-term recurrences
\begin{equation}
  (AA^T)S_k = S_{k+1}(L_{k+}L_k^T),
  \qquad
  (A^TA)W_k =  W_{k+1}(L_{k+1}^TL_{k+}),\label{eq:lanczos}
\end{equation}
allow us to describe the bidiagonalization vectors $s_{k+1}$ and $w_{k+1}$  in terms of the Lanczos polynomials as
\begin{equation}
s_{k+1} = \varphi_k(AA^T)b, \qquad w_{k+1} = \psi_k(A^TA)A^Tb \qquad \varphi_k,\psi_k \in \mathcal{P}_k; \label{eq:lanczos_poly}
\end{equation}
see \cite{Paige1982LSQR,Bjoerck1988bidiagonalization,Meurant2006Lanczos,Meurant2006Lanczosa,Golub2009Matrices}.
From \eqref{eq:lanczos_poly} we have that
\begin{equation}
s_{k+1} = \varphi_k(AA^T)b = \varphi_k(AA^T)(Ax+\eta), \label{eq:poly_form}
\end{equation}
giving
\begin{equation}
s_{k+1} = \left[\varphi_k(AA^T)Ax + (\varphi_k(AA^T) - \varphi_k(0))\eta\right] + \varphi_k(0)\eta.  \label{eq:poly_split}
\end{equation}	
The first component on the right-hand side of \eqref{eq:poly_split} can be rewritten as
\begin{equation}
s_{k+1}^\text{LF} \equiv  \left[\varphi_k(AA^T)Ax + (\varphi_k(AA^T) - \varphi_k(0)) \eta\right]
= A q_{k-1}(AA^T)\left[x+A^T \eta\right] ,
\end{equation}
for some $q_{k-1} \in \mathcal{P}_k$. Since $A$ has the smoothing property, then $s_{k+1}^\text{LF}$ is smooth for $k\ll\min(m,n)$. Thus $s_{k+1}$ is a~sum of a~low-frequency vector and the scaled noise vector $\eta$,
\begin{equation}
s_{k+1} = s_{k+1}^\text{LF} + \varphi_k(0)\eta.  \label{eq:poly_split_simple}
\end{equation}
Note that this splitting corresponds to the low-frequency part and propagated 
(non-smoothed) noise part only when $\|s_{k+1}^\text{LF}\|^{2}+\|\varphi_k(0)\eta\|^{2}\approx 1$. For large $k$s, there is a~considerable cancellation between $s_{k+1}^\text{LF}$ and $\varphi_k(0)\eta$, the splitting \eqref{eq:poly_split_simple} still holds but it does not correspond to our intuition of an underlying smooth vector and some added scaled noise. Thus we restrict ourselves to smaller values of $k$.

It has been shown in \cite{Hnetynkova2009regularizing} that whereas for $s_1$ (the scaled right-hand side) the noise part in \eqref{eq:poly_split_simple} is small compared to the true data, for larger $k$, due to the smoothing property of the matrix $A$ and the orthogonality between the vectors $s_k$, the noise part becomes more significant.
The noise scaling factor determining the relative amplification of the non-smoothed part of noise corresponds to the constant term of the Lanczos polynomial
\begin{equation}
\varphi_k(0) = (-1)^{k}\frac{1}{\beta_{k+1}}\prod_{j=1}^{k}\frac{\alpha_j}{\beta_j}
\label{eq:ampl_factor}
\end{equation}
called the \emph{amplification factor}.\footnote{Note that  in \cite{Hnetynkova2009regularizing} a different notation was used. The Lanczos polynomial $\varphi_k$ was scaled by $\|b\|$ so that $s_{k+1} = \tilde\varphi_k(AA^T)s_{1}$. The vector $s_{k+1}$ was split into $s_{k+1} = s_{k+1}^\text{exact}+s_{k+1}^\text{noise}$. In our notation,  $s_{k+1}^\text{exact}=s_{k+1}^\text{LF}$, and $s_{k+1}^\text{noise}=\varphi_k(0)\eta$.}
Its behavior for problems with white noise was studied in \cite{Hnetynkova2009regularizing} and the analysis concludes that its size increases with $k$ until the \emph{noise revealing iteration} $k_{\text{rev}}$, where the vector $s_{k+1}$ is dominated by the non-smoothed part of noise. Then the amplification factor increases at least for one iteration. 
The noise revealing iteration $k_{\text{rev}}$ can be defined as 
Note that there is no analogy for the right bidiagonalization vectors, since all vectors $w_k$ are smoothed and the factor $\psi_k(0)$ on average grows till late iterations. A recursive relation for $\psi_k(0)$, obtained directly from \eqref{eq:bid_alg} has the form
\begin{align}
\psi_0(0) &= \frac{1}{\alpha_1\beta_1},\\
\psi_k(0) &= \frac{1}{\alpha_{k+1}}(\varphi_k(0) - \beta_{k+1}\psi_{k-1}(0)), \quad k = 1,2,\ldots. \label{eq:abspsi}
\end{align}

\subsection{Behavior of the noise amplification factor}\label{sec:behavior}
\paragraph{Influence of the noise frequency characteristics} The phenomenon of noise amplification is demonstrated on the problems from \cite{Hansen1994Regularization, Hansen2007Regularization}. Figures~ \ref{fig:phi0} and \ref{fig:psi0} show the absolute terms of the Lanczos polynomials $\varphi_k$ and $\psi_k$ for the problem \texttt{shaw} polluted with white noise of various noise levels. 
For example, for the noise level $10^{-3}$, the maximum of $\varphi_k(0)$ is achieved for $k = 6$, which corresponds to the observation that the vector $s_{7}$ in Figure~\ref{fig:S_k} is the  most dominated by propagated noise. Obviously, the noise revealing iteration increases with decreasing noise level. The amplification factors exhibit similar behavior before the first decrease. 
However, the behavior of $\varphi_k(0)$ can be more complicated. In Figure~\ref{fig:phillips_multi} for \texttt{phillips}, the sizes of the amplification factors oscillate as a consequence of the oscillations in the sizes of the spectral components of $b$ in the left singular subspaces of $A$. Thus there is a partial reduction of the noise component, which influences the subsequent iterations, even before the noise revealing iteration. 
\begin{figure}
        \centering
        \begin{subfigure}[b]{0.95\textwidth}
                \includegraphics[width=\textwidth]{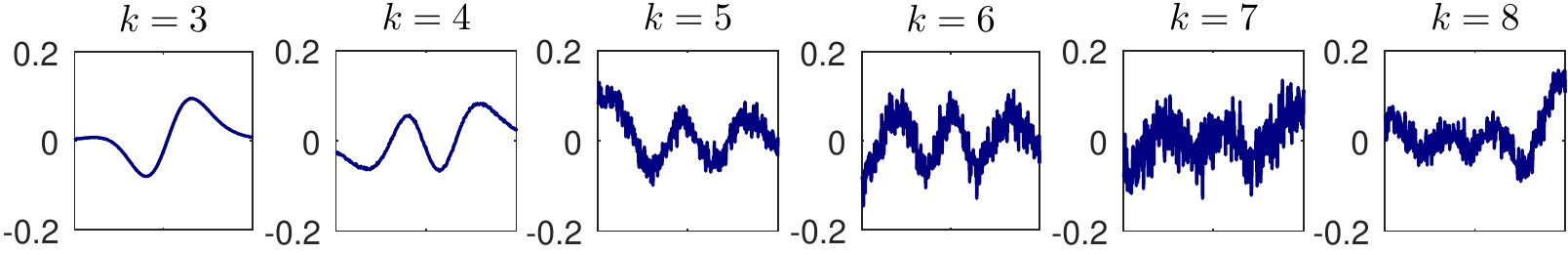}
                \caption{left bidiagonalization vectors $s_{k+1}$}
                \label{fig:S_k}
        \end{subfigure}

        \vspace*{.3cm}

        \begin{subfigure}[b]{0.45\textwidth}
                \includegraphics[height = 4cm]{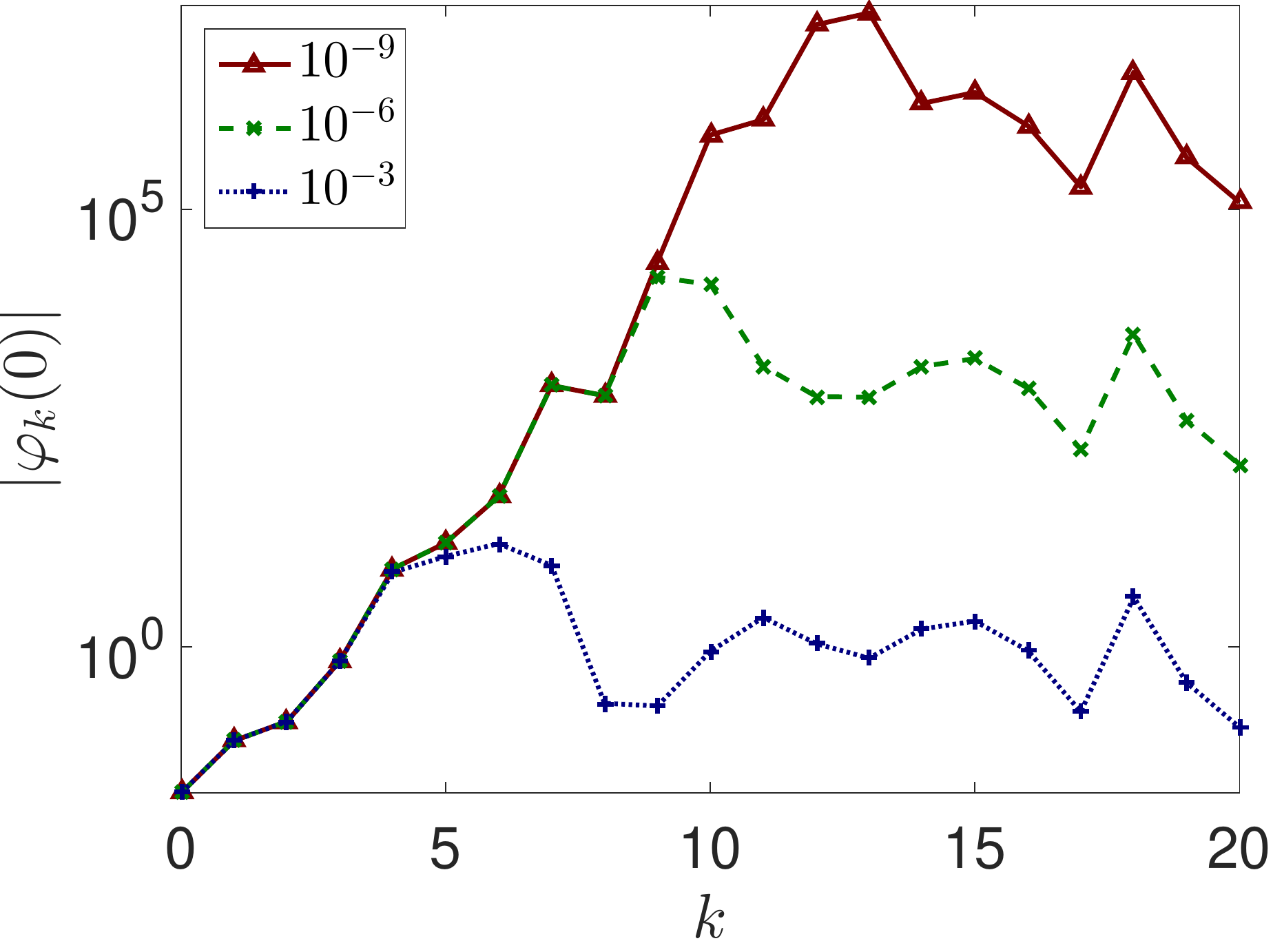}
                \caption{the size of $\varphi_{k}(0)$}
                \label{fig:phi0}
        \end{subfigure}
        \begin{subfigure}[b]{0.45\textwidth}
                \includegraphics[height = 4cm]{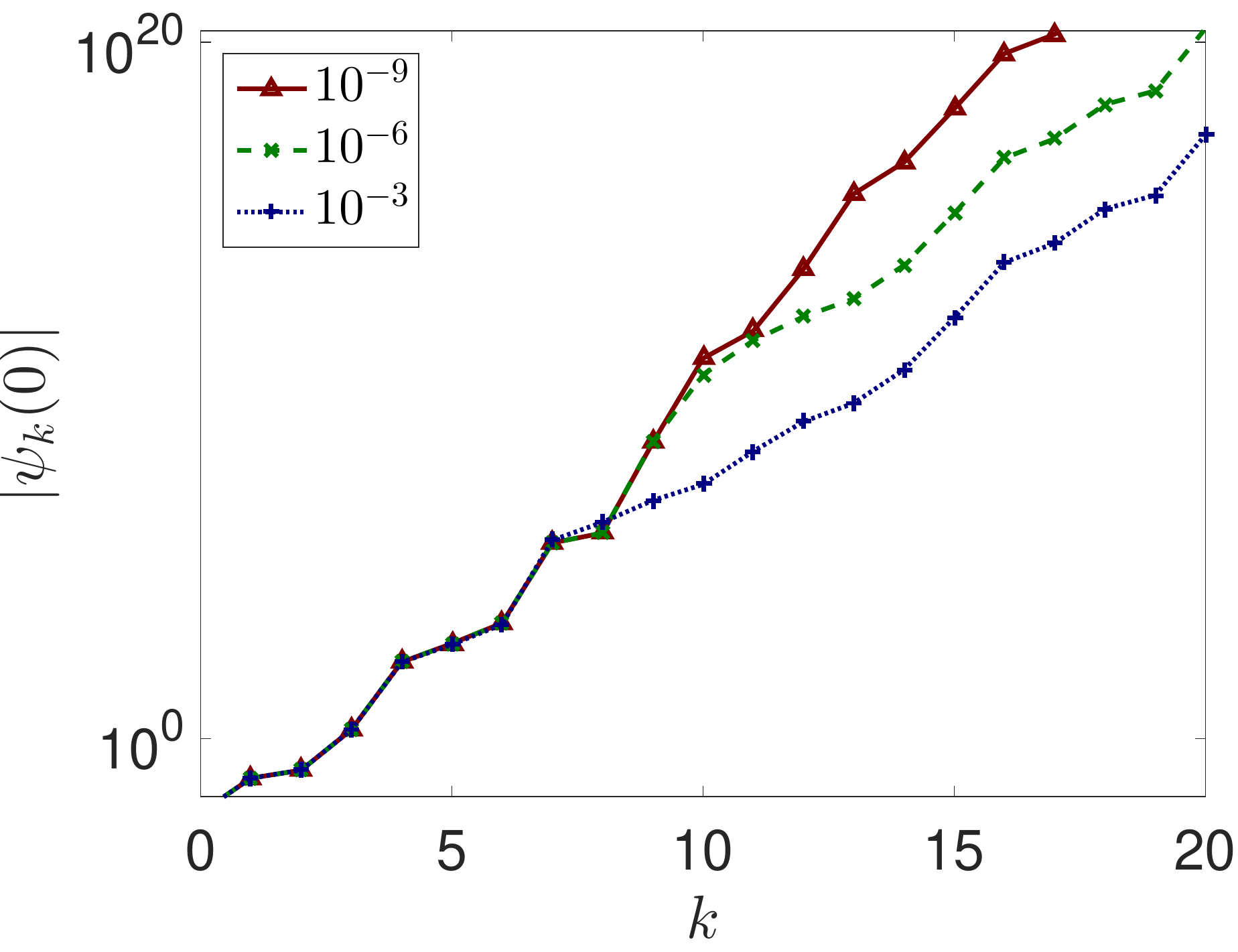}
                \caption{the size of $\psi_{k}(0)$}
                \label{fig:psi0}
        \end{subfigure}
        \caption{The problem \texttt{shaw(400)} polluted by white noise: (a) the left bidiagonalization vectors $s_{k+1}$ for the noise level $10^{-3}$;
        (b) the size of the absolute term of the Lanczos polynomial $\varphi_k$ for various noise levels;
        (c) the size of the absolute term of the Lanczos polynomial $\psi_k$ for various noise levels.}\label{fig:S_k_separate}
\end{figure}

\begin{figure}
        \centering     
        \begin{subfigure}[b]{0.32\textwidth}
                \includegraphics[width=\textwidth]{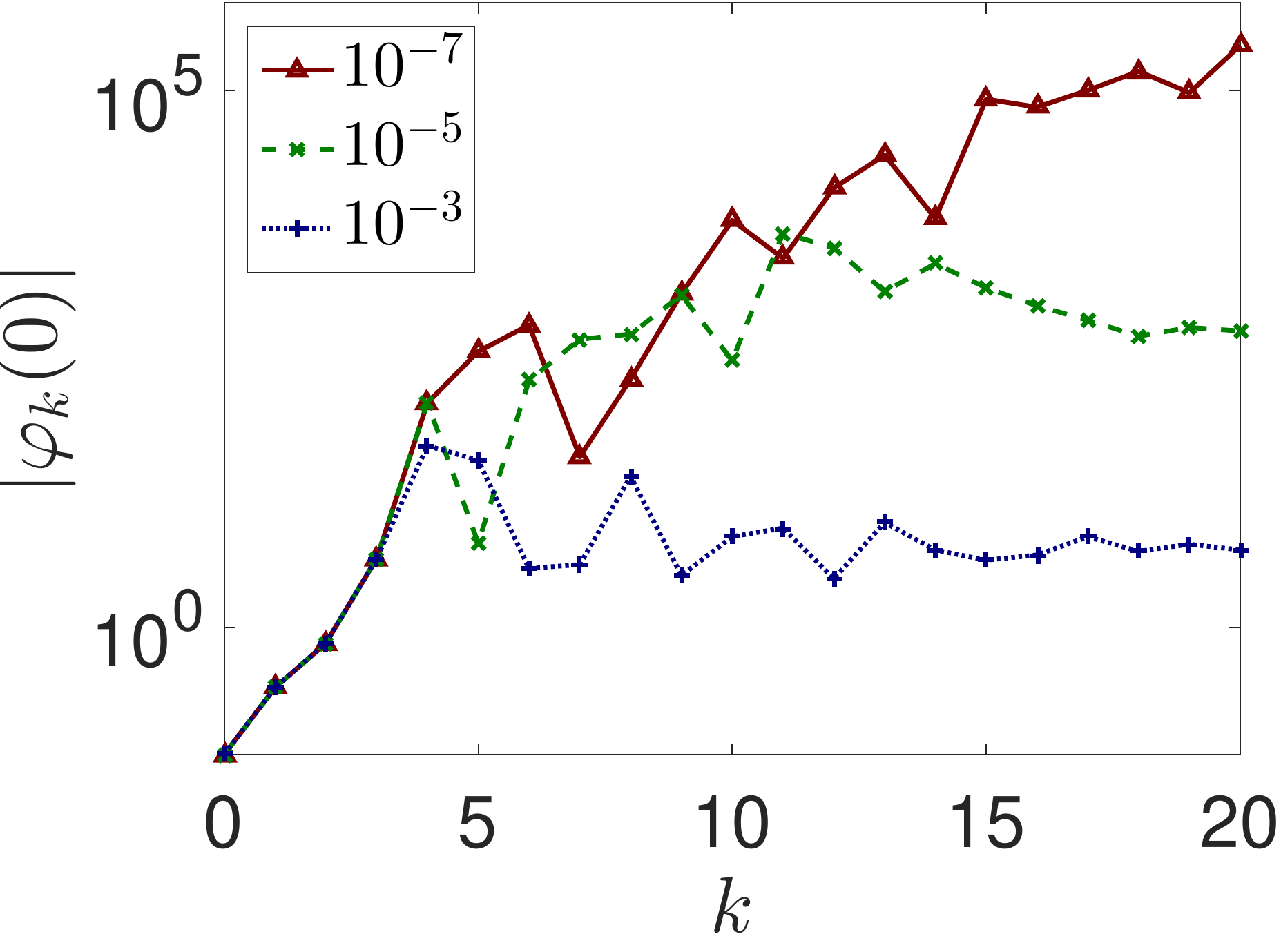}
                \caption{\texttt{phillips(400)}}\label{fig:phillips_multi}
        \end{subfigure}
        \begin{subfigure}[b]{0.32\textwidth}
                \includegraphics[width=\textwidth]{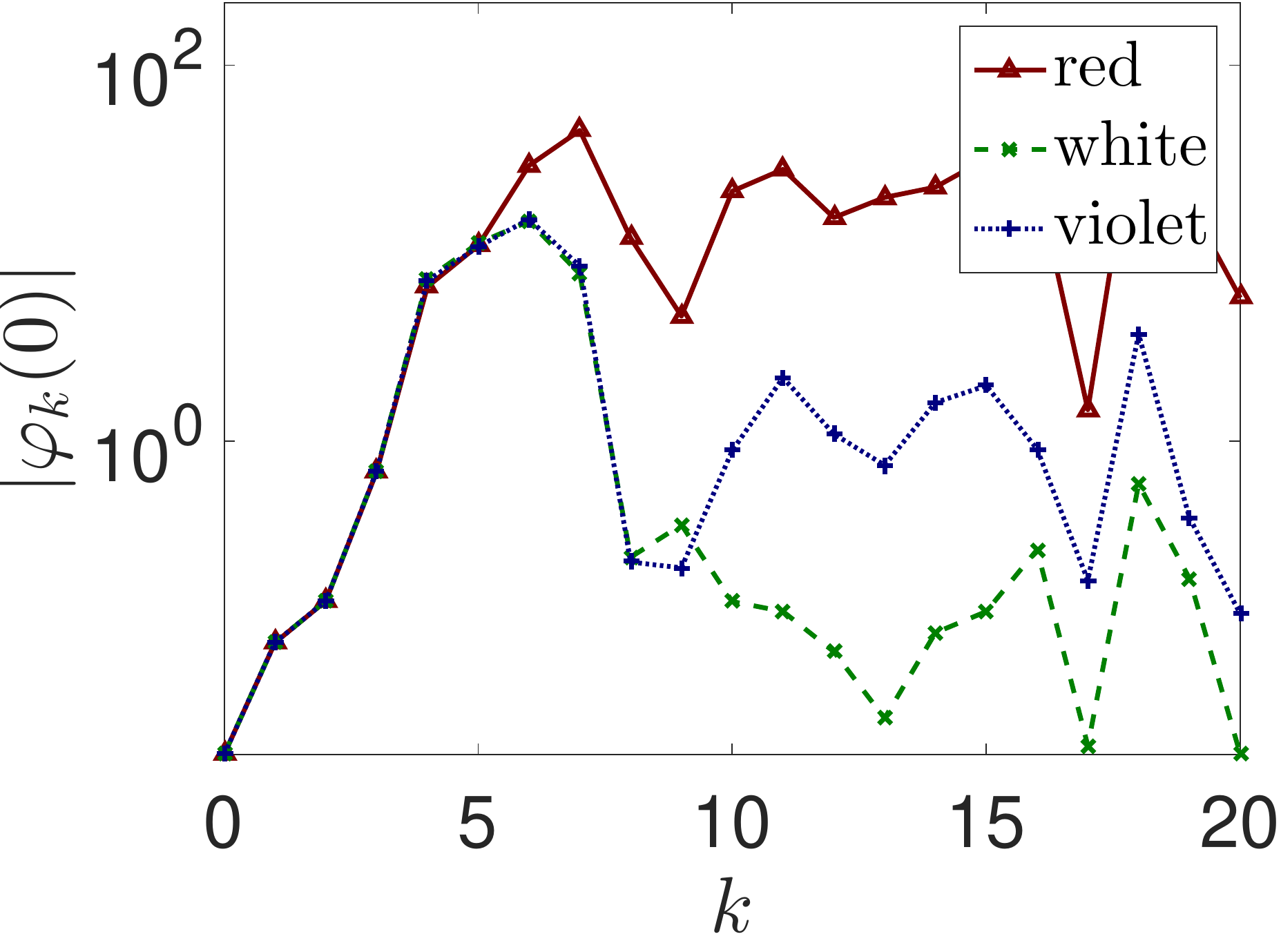}
                \caption{\texttt{shaw(400)}}\label{fig:shaw_multi_color}
        \end{subfigure}
        \begin{subfigure}[b]{0.32\textwidth}
                \includegraphics[width=\textwidth]{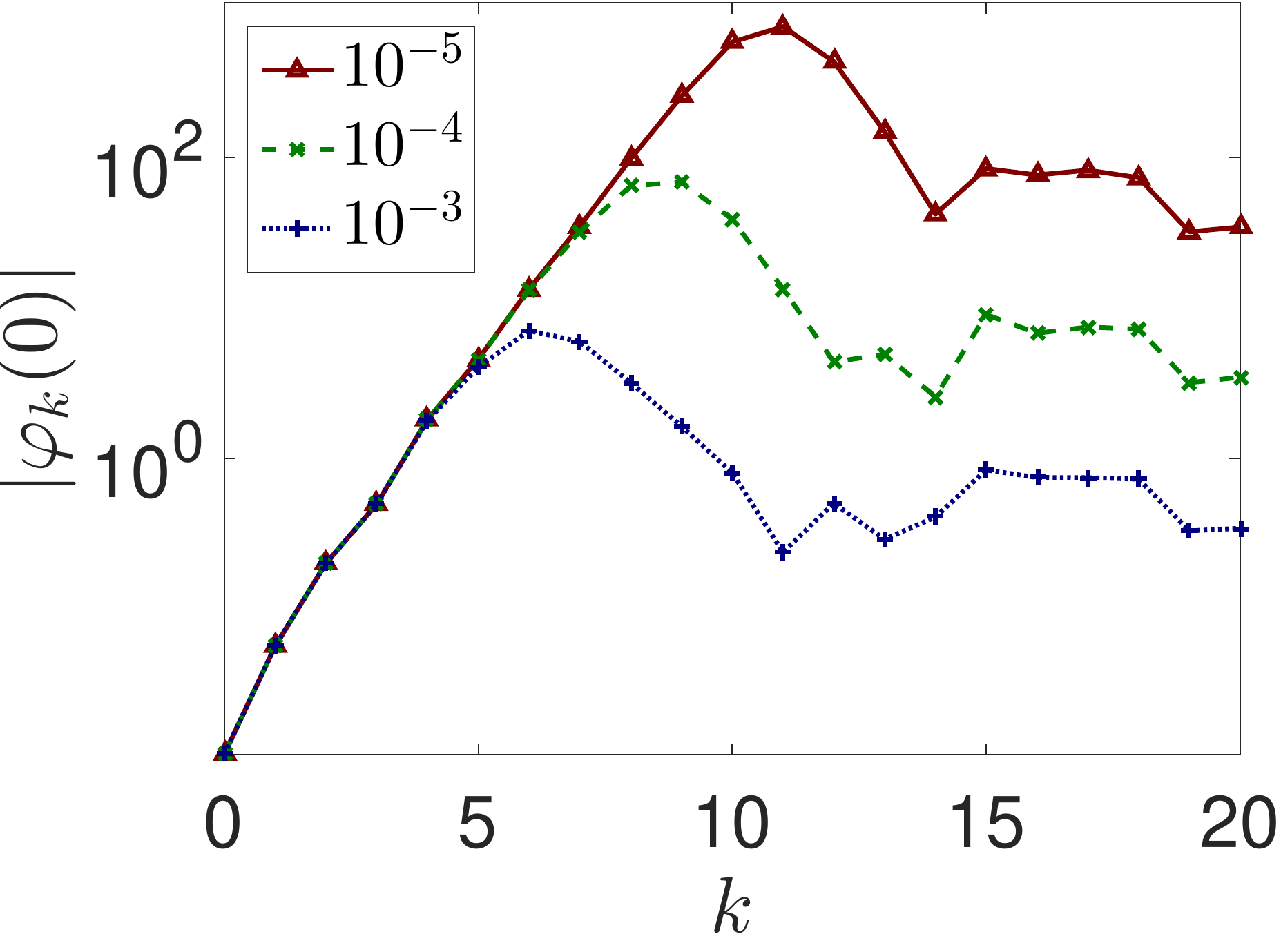}
                \caption{\texttt{gravity(400)}}\label{fig:gravity_multi_poisson}
        \end{subfigure}
        \caption{Influence of the amount of noise and its frequency characteristics on the amplification factor \eqref{eq:ampl_factor}: (a) the problem \texttt{phillips} with various noise levels of white noise;  (b) the problem \texttt{shaw} with noise of different frequency characteristics; (c) the problem \texttt{gravity} with Poisson noise with different noise levels achieved by scaling.}\label{fig:gamma_shaw_all}
\end{figure}

Even though \cite{Hnetynkova2009regularizing} assumed white noise, noise amplification can be observed also for other noise settings and the formulas \eqref{eq:lanczos_poly}-\eqref{eq:abspsi} still hold. However, for high-frequency noise, there is smaller cancellation between the low-frequency component $s_{k+1}^\text{LF}$ and the noise part $\varphi_k(0)\eta$ in \eqref{eq:poly_split_simple}. Therefore, in the orthogonalization steps succeeding the noise revealing iteration $k_\text{rev}$, the noise part is projected out more significantly. For low-frequency noise, on the other side, this smoothing is less significant, which results in smaller drop of \eqref{eq:ampl_factor} after $k_\text{rev}$.
This is illustrated in Figure~\ref{fig:shaw_multi_color} on the problem \texttt{shaw}
polluted by red (low-frequency), white, and violet (high-frequency) noise of the same noise level. For spectral characteristics of these types of noise see Figure~\ref{fig:power_spectra}.
\begin{figure}
        \centering     
        \begin{subfigure}[b]{0.31\textwidth}
                \includegraphics[width=\textwidth]{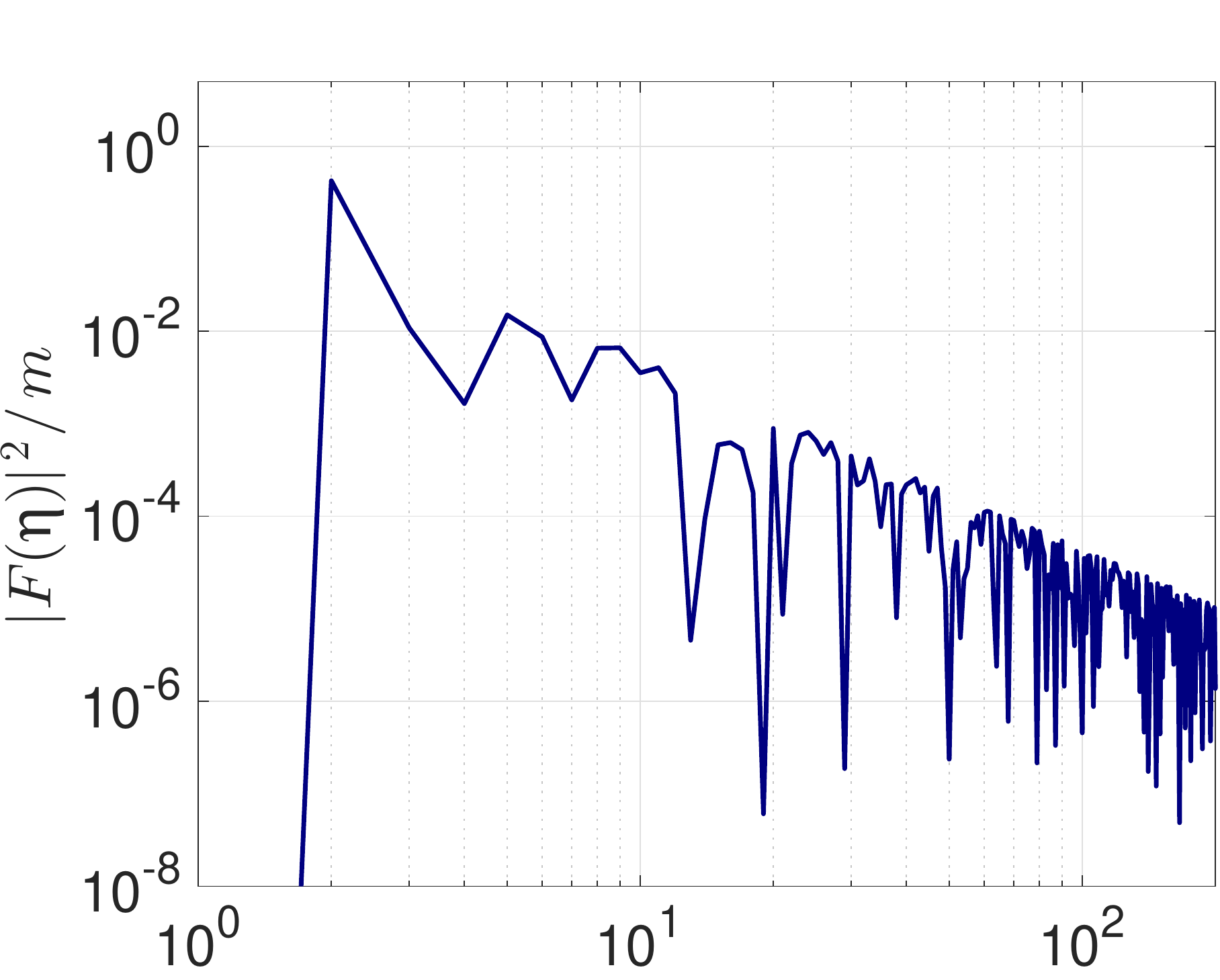}
                \caption{red noise}
        \end{subfigure}
        \hspace*{.2cm}
        \begin{subfigure}[b]{0.31\textwidth}
                \includegraphics[width=\textwidth]{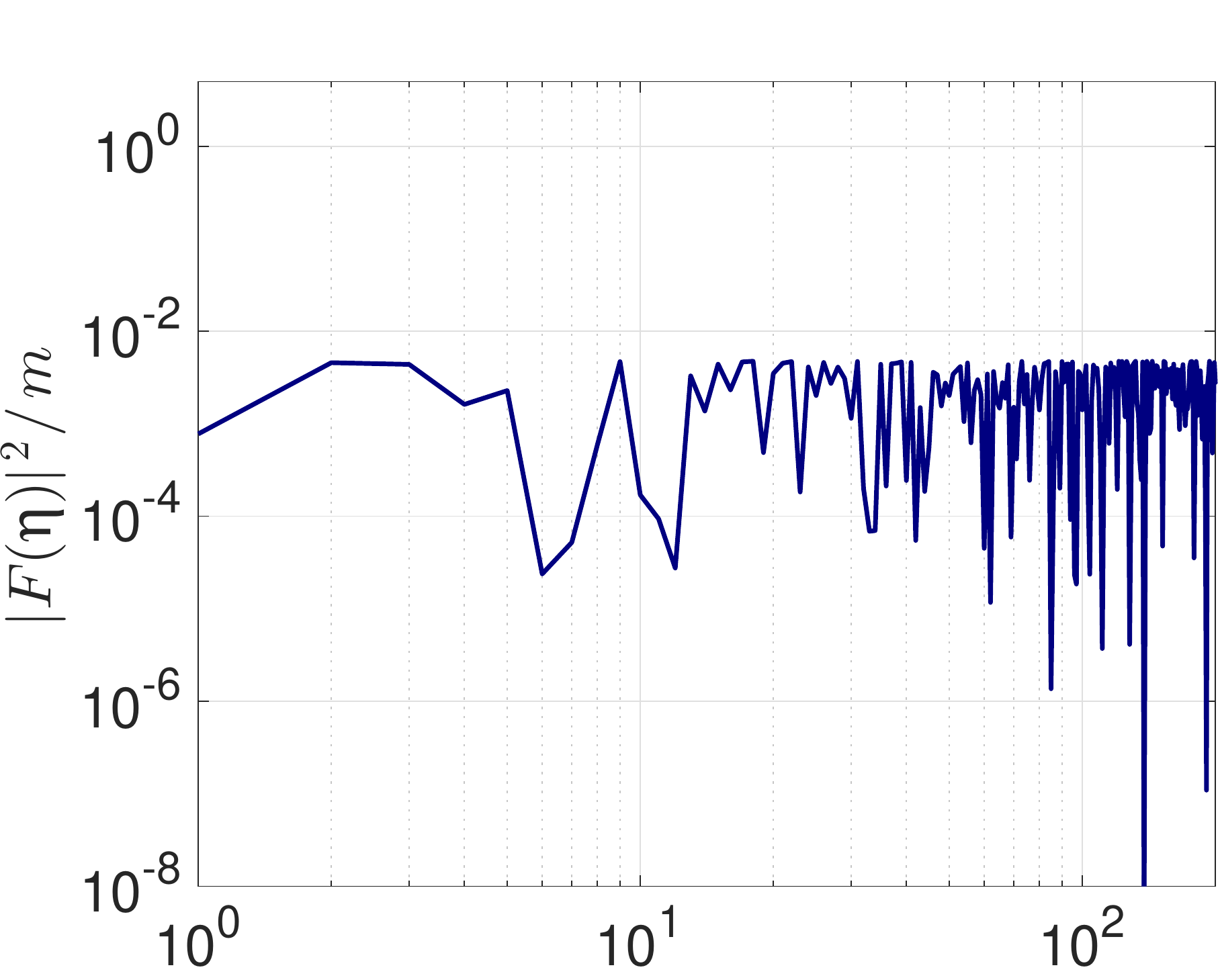}
                \caption{white noise}
        \end{subfigure}
        \hspace*{.2cm}
        \begin{subfigure}[b]{0.31\textwidth}
                \includegraphics[width=\textwidth]{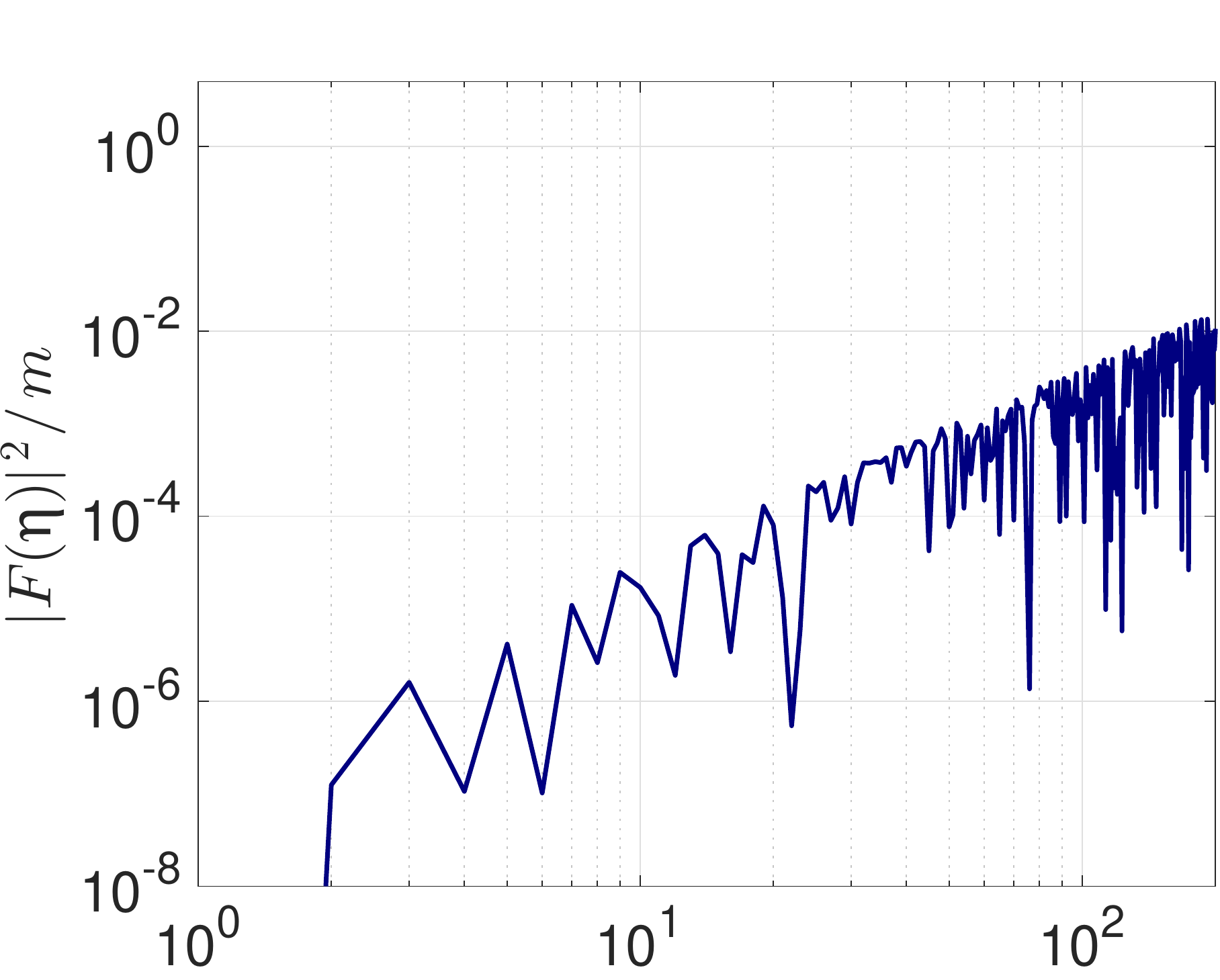}
                \caption{violet noise}
        \end{subfigure}
        \caption{Power spectral densities (or simply power spectra) for red (low-frequency dominated),
        white (or Gaussian), and violet (high-frequency dominated) noise $\eta$, $\|\eta\|=1$. Power spectrum is given by squared
        magnitudes of Fourier coefficients $F(\eta)$ of $\eta$ (see, e.g., \cite[chap. 2.7]{Brown1997Introduction}),
        here computed by the discrete Fourier transform. Power spectra are normalized by the length
        of the vector.}\label{fig:power_spectra}
\end{figure}
Figure~\ref{fig:gravity_multi_poisson} shows the amplification factor for various levels of Poisson noise.

\paragraph{Influence of the loss of orthogonality}
First note that the splitting \eqref{eq:poly_split_simple} remains valid even if $\varphi_k$ are not exactly orthonormal Lanczos polynomials, since the propagated noise can be still tracked using the absolute term of the corresponding (computed) polynomial. Nevertheless, it is clear that the loss of orthogonality among the left bidiagonalization vectors in finite precision arithmetic influences the behavior of the amplification factor $\varphi_k$, i.e. the propagation of noise. In the following, we denote all quantities computed without reorthogonalization by hat. 
Loss of orthogonality can be detected, e.g., by tracking the size of the smallest singular value $\bar{\sigma}_\text{min}$ of the matrix $\hat{S}_k$ of the computed left bidiagonalization vectors. In Figure~\ref{fig:gamma_shift} (left) for the problem \texttt{shaw} and \texttt{gravity} we see that when $\bar{\sigma}_\text{min}$ drops below one detecting the loss of orthogonality among its columns, the size of the amplification factor $\hat{\varphi}_k(0)$ starts to oscillate.
However, except of the delay, the larger values of $|\hat{\varphi}_k(0)|$ still match those of $|{\varphi}_k(0)|$. If we  plot $|\hat\varphi_k(0)|$ against the rank of $\hat{S}_k$ instead of $k$, the sizes of the two amplification factors become very similar. In our experiments, the rank of $\hat{S}_k$ was computed as \texttt{rank(S(:,1:k),1e-1)} in MATLAB, i.e., singular values of $\hat{S}_k$ at least ten times smaller than they would be for orthonormal columns were considered zero. A similar shifting strategy was proposed in  \cite[chap. 3]{Gergelits2013Analysis} for the convergence curves of the conjugate gradient method. Note that the choice of the tolerance can be problem dependent. Further study of this phenomenon is beyond the scope of this paper, but we can conclude that except of the delay the noise revealing phenomenon is in finite precision computations present.  
\begin{figure}
\centering
        \begin{subfigure}[b]{0.49\textwidth}
        		\centering
                \includegraphics[height=4cm]{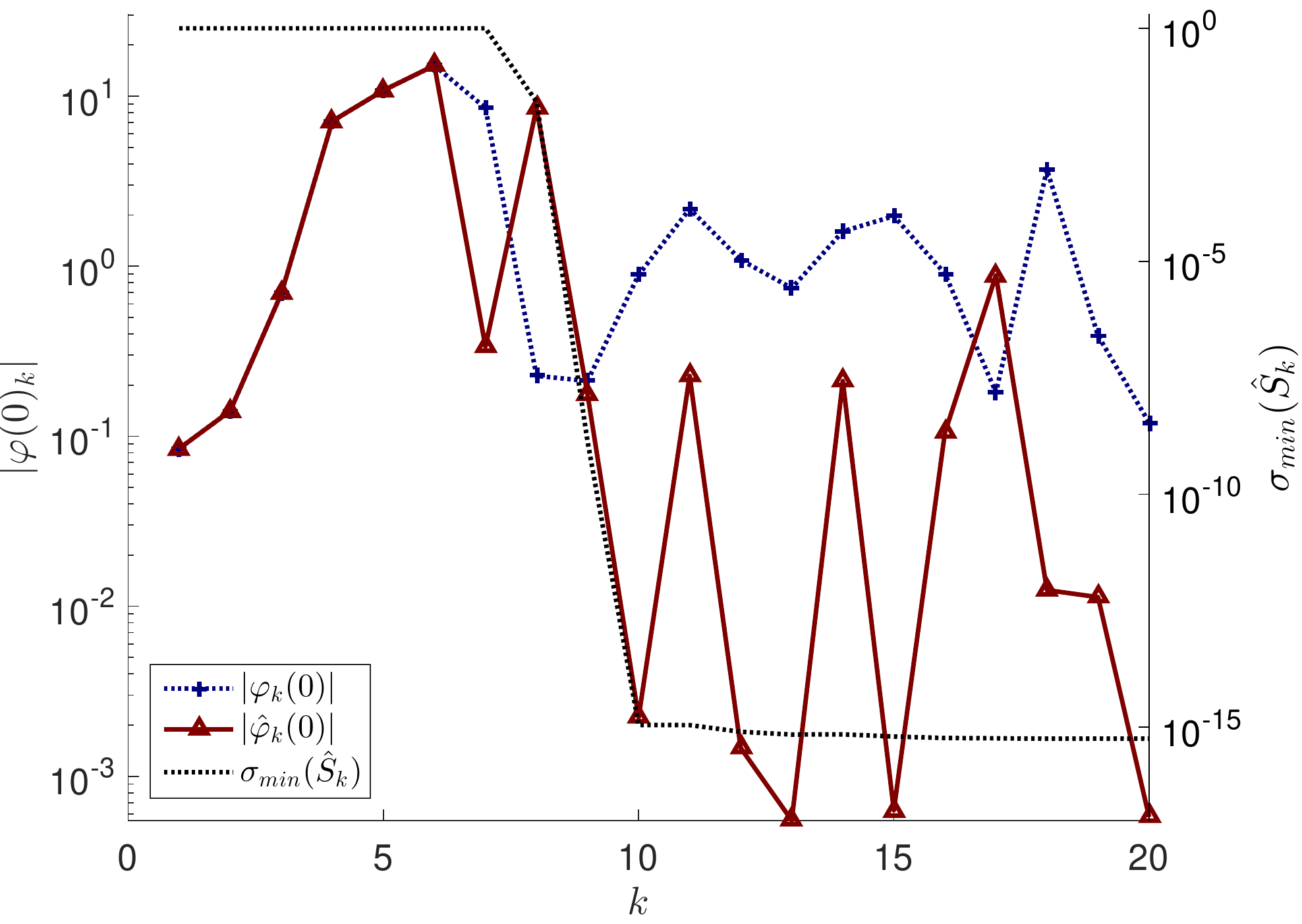}
                \caption{ \texttt{shaw(400)}, $\delta_\text{noise} = 10^{-3}$}
        \end{subfigure}
        \begin{subfigure}[b]{0.49\textwidth}
        		\centering
                \includegraphics[height=4cm]{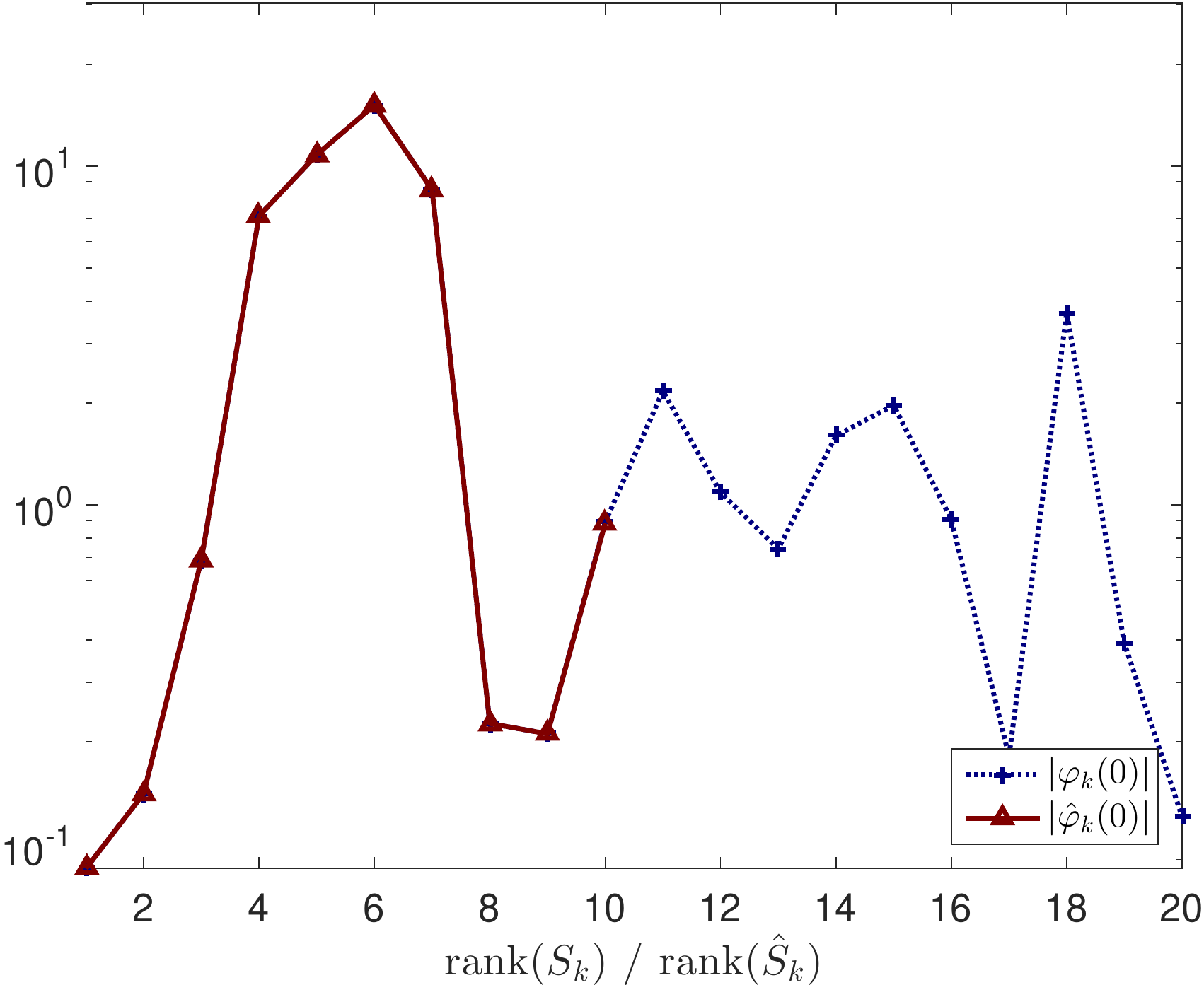}
                \caption{curves from the left after shift}
        \end{subfigure}

        \vspace*{.3cm}

                \begin{subfigure}[b]{0.49\textwidth}
        		\centering
                \includegraphics[height=4cm]{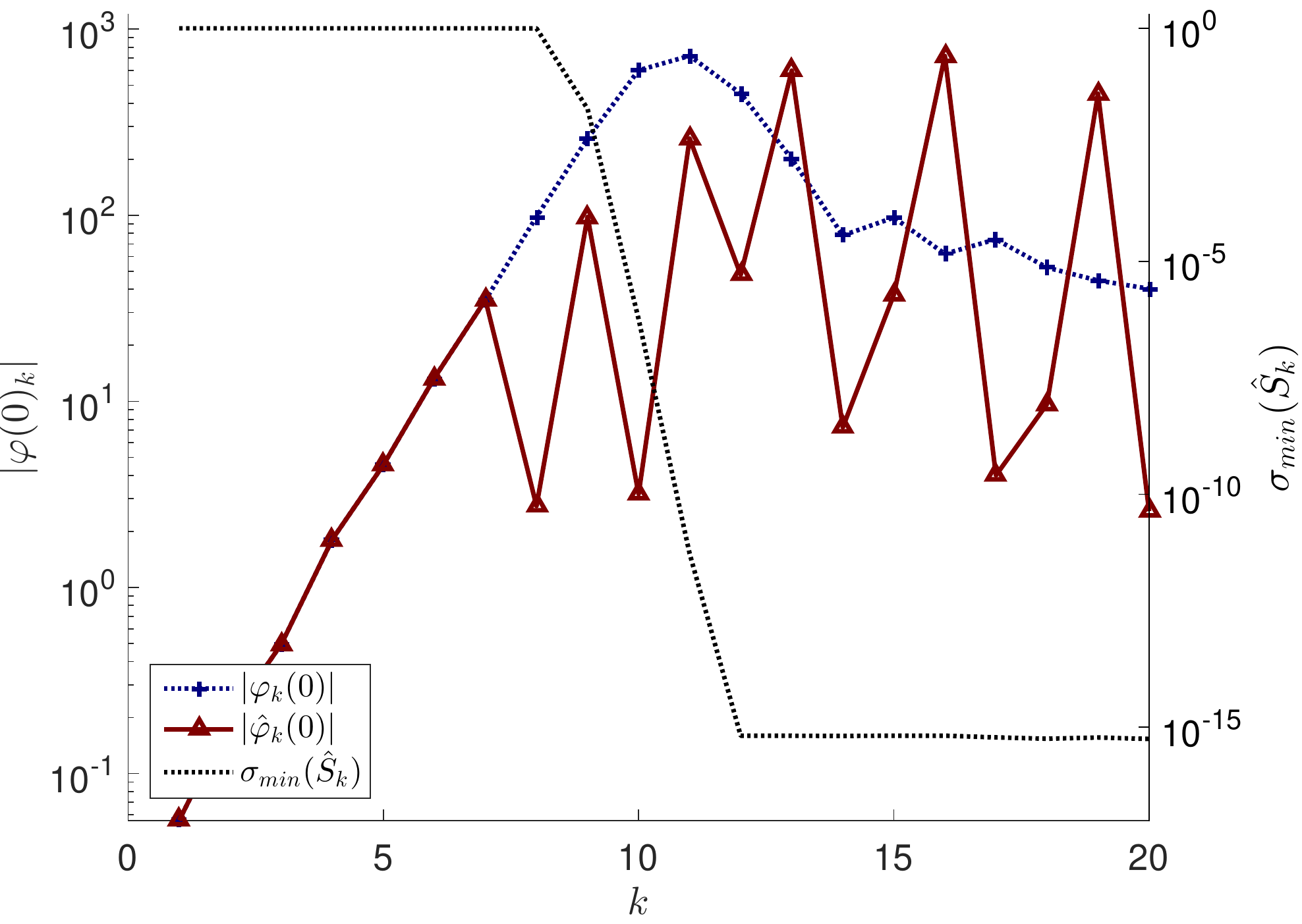}
                \caption{ \texttt{gravity(400)}, $\delta_\text{noise} = 10^{-5}$}
        \end{subfigure}
        \begin{subfigure}[b]{0.49\textwidth}
        		\centering
                \includegraphics[height=4cm]{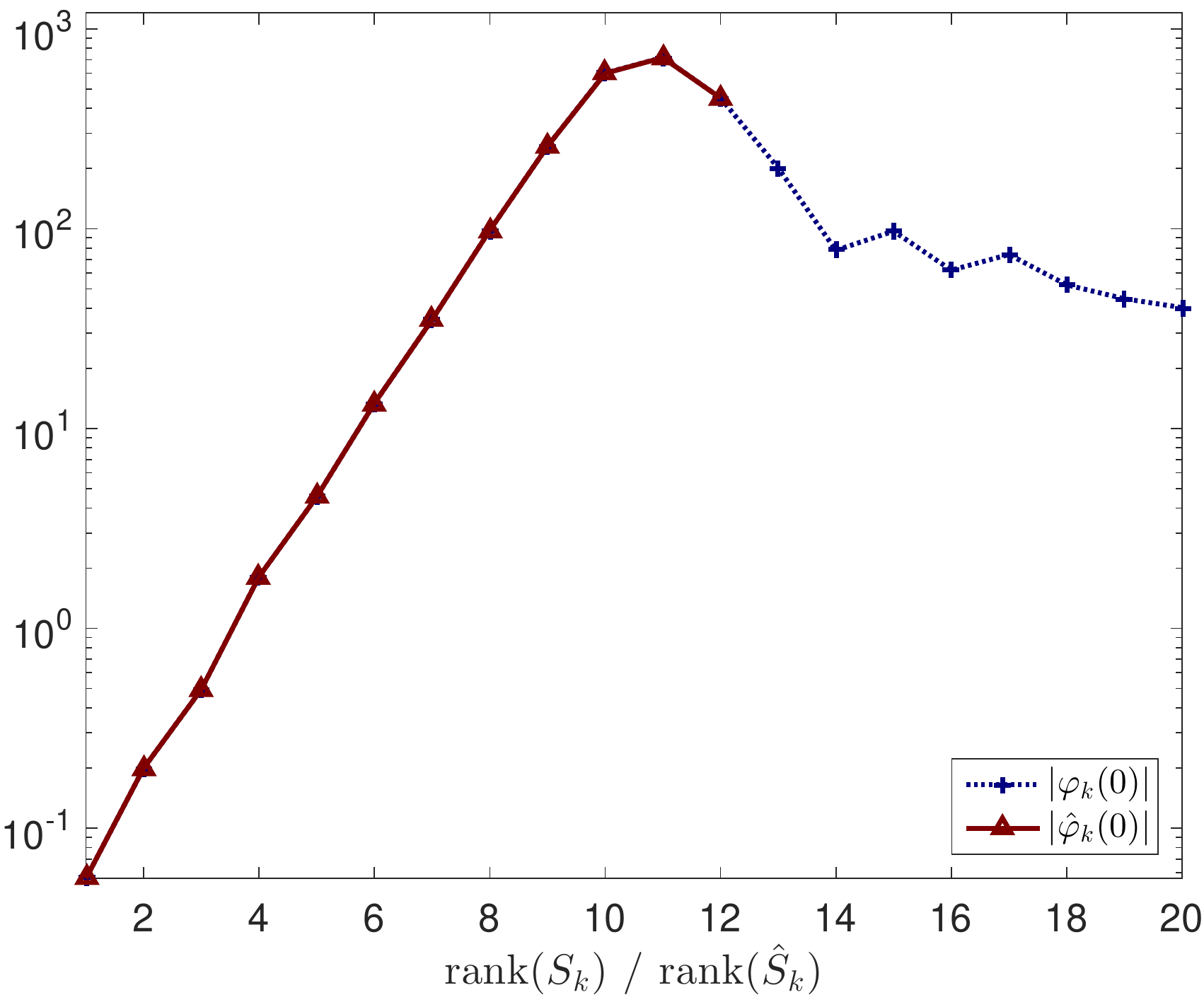}
                \caption{curves from the left after shift}
        \end{subfigure}
        \caption{Illustration of the noise amplification for the problem \texttt{shaw} and \texttt{gravity} in finite precision computations. Left: The sizes of the amplification factor \eqref{eq:ampl_factor}  computed with full double reorthogonalization ($\varphi_k(0)$) and without reorthogonalization ($\hat\varphi_k(0)$). Right: $|\hat\varphi_k(0)|$ plotted against  $\text{rank}(\hat{S}_k)$ computed as \texttt{rank(S(:,1:k),1e-1)} in MATLAB, together with $|\varphi_k(0)|$ plotted against $\text{rank}({S}_k) = k$.}\label{fig:gamma_shift}
\end{figure}

\section{Noise in the residuals of iterative methods}\label{sec:residuals_noise}

CRAIG \cite{Craig1955N}, LSQR \cite{Paige1982LSQR}, and LSMR \cite{Fong2011LSMR} represent three methods based on the Golub-Kahan iterative bidiagonalization. At the $k$-th step, they search for the approximation of the solution in the subspace generated by vectors $w_1,\ldots,w_k$, i.e.,
\begin{equation}
 x_k = W_ky_k, \qquad y_k \in \mathbb{R}^k. \label{eq:projected}
\end{equation}
The corresponding residual has the form
\begin{align}
r_k \equiv b-Ax_k =  b - AW_ky_k &= S_{k+1}(\beta_1e_1 - L_{k+}y_k). \label{eq:residual}
\end{align}
CRAIG minimizes the distance of $x_k$ from the naive solution yielding 
\begin{equation}
L_ky_k^\text{CRAIG} = \beta_1e_1.\label{eq:CRAIG_projected}
\end{equation}
LSQR minimizes the norm of the residual $r_k$ yielding
\begin{equation}
y_k^\text{LSQR} = \underset{y\in\mathbb{R}^k}{\operatorname{argmin}}\|\beta_1e_1 - L_{k+}y\|. \label{eq:LSQR_projected}
\end{equation}
LSMR minimizes the norm of $A^Tr_k$ giving
\begin{equation}
y_k^\text{LSMR} = \underset{y\in\mathbb{R}^k}{\operatorname{argmin}}\|\beta_1\alpha_1e_1 - L_{k+1}^TL_{k+}y\|. \label{eq:LSMR_projected}
\end{equation}
These methods are mathematically equivalent to Krylov subspace methods based on the Lanczos tridiagonalization 
(particularly Lanczos for linear systems and MINRES) applied to particular normal equations. 
The relations useful in the following derivations are summarized
in Table~\ref{tab:methods_ne}.
\begin{table}
\caption{Interpretation of bidiagonalization-based methods (CRAIG, LSQR, LSMR) as tridiagonalization-based methods 
(Lanczos for linear systems, MINRES) applied to the corresponding normal equations.
In last two columns, the solution $x$ of the bidiagonalization-based methods is obtained from their tridiagonalization counterparts as $x = A^Ty$ and $x = A^TAz$, 
respectively. See also \cite{Fong2011LSMR}.}\label{tab:methods_ne}
\centering
\begin{tabular}{c|cccc}
\hline
method/equation & $(A^TA)x = A^Tb$ & $(AA^T)y = b$ & $(A^TA)z = A^\dagger b$  \\
\hline
Lanczos method  & LSQR$(A,b)$ & CRAIG$(A,b)$ & --- \\
MINRES          & LSMR$(A,b)$ & LSQR$(A,b)$ & CRAIG$(A,b)$
\end{tabular}
\end{table}

Since Lanczos method is a Galerkin (residual orthogonalization) method, we immediately see that
\begin{align}
r_k^\text{CRAIG} &= (-1)^k\|r_k^\text{CRAIG}\|s_{k+1},\\ \label{eq:residuals_Galerkin}
A^Tr_k^\text{LSQR} &= (-1)^k\|A^Tr_k^\text{LSQR}\|w_{k+1}.
\end{align}
Using the relation between the Galerkin an the residual minimization method, see\cite[sec. 6.5.7]{Saad2003Iterative}, we obtain, 
\begin{align}
\|r_k^\text{LSQR}\|  & = \frac{1}{\sqrt{\sum_{l=0}^k{1/\|r_l^\text{CRAIG}\|^2}}},\label{eq:res_norm_LSQR}\\
\|A^Tr_k^\text{LSMR}\|  & = \frac{1}{\sqrt{\sum_{l=0}^k{1/\|A^Tr_l^\text{LSQR}\|^2}}}. \label{eq:res_nomr_LSMR}
\end{align}
Note that these equations hold, up to a small perturbation, also in finite precision computations. See \cite{Cullum1996Relations} for more details.

In the rest of this section, we investigate the residuals of each particular method. We focus on in which sense the residuals approximate the noise vector.
We discuss particularly the case when noise contaminates the bidiagonalization vectors fast and thus the noise revealing
iteration is well defined. More general discussion follows in Section~\ref{sec:2Dproblems}.

\subsection{CRAIG residuals}\label{sec:relation_CRAIG}
The following result relates approximate solution obtained by CRAIG  for \eqref{eq:theproblem} to the solution of the problem with the same matrix and a modified right-hand side.  

\begin{prop}\label{th:1}
Consider the first $k$ steps of the Golub-Kahan iterative bidiagonalization. Then the approximation $x_{k}^\text{CRAIG}$ defined in \eqref{eq:projected} and \eqref{eq:CRAIG_projected}, 
is an exact solution to a consistent problem
\begin{equation}
Ax = b - \varphi_k(0)^{-1}s_{k+1}. \label{eq:mod}
\end{equation}
Consequently, 
\begin{equation}
\|r_k^\text{CRAIG}\|  = |\varphi_k(0)|^{-1}.\label{eq:residual_CRAIG}
\end{equation}

\end{prop}
\begin{proof}First note that we only need to show that $r_k^\text{CRAIG} = \varphi_k(0)^{-1}s_{k+1}$, $k=1,2,\ldots$.
From \eqref{eq:residuals_Galerkin} and \eqref{eq:lanczos_poly} it follows that there exist $c\in \mathbb{R}$, such that 
\begin{equation}r_k^\text{CRAIG} = c \cdot s_{k+1} = c \cdot \varphi_k(AA^T)b.
\end{equation} Let us now determine the constant $c$. From \eqref{eq:residual} and \eqref{eq:lanczos_poly}, we have that
\begin{equation}
r_k^\text{CRAIG} = \Pi_k (AA^T)b, \quad \text{where} \quad \Pi_k \in \mathcal{P}_k \ \ \text{and} \ \ \Pi_k(0) = 1. \label{eq:Pi}
\end{equation} 
Combining these two equations, we obtain
\begin{equation}
r_k^\text{CRAIG}  = \varphi_k(0)^{-1}\varphi_k(AA^T)b.\label{eq:craig_res}
\end{equation}
Substituting to \eqref{eq:craig_res} back from \eqref{eq:lanczos_poly}, we immediately have \eqref{eq:mod}. Since $\|s_{k+1}\| = 1$, \eqref{eq:residual_CRAIG} is a direct consequence of \eqref{eq:mod}. \qed
\end{proof}
  
Although the relation \eqref{eq:mod} is valid for any problem of the form \eqref{eq:theproblem}, it has a particularly interesting interpretation for inverse problems with a smoothing operator $A$. Suppose we neglect the low-frequency part $s_{k+1}^\text{LF}$ in \eqref{eq:poly_split_simple} and estimate the unknown noise $\eta$ from the left bidiagonalization vector $s_{k+1}$ as
\begin{equation}
\eta \approx \tilde{\eta} \equiv \varphi_k(0)^{-1} s_{k+1}.\label{eq:noise_est}
\end{equation}
Subtracting $\tilde\eta$ from $b$ in \eqref{eq:theproblem}, we obtain exactly the modified  problem \eqref{eq:mod}.
Thus Proposition~\ref{th:1} in fact states that in each iteration $k$, $x_k^\text{CRAIG}$ represents the exact solution of the problem 
\eqref{eq:transformed} with noise being approximated by a particular re-scaled left bidiagonalization vector. 

The norm of the CRAIG residual $r_k^\text{CRAIG}$ is inversely proportional to the amount of noise propagated to the currently computed left bidiagonalization vector. It reaches its minimum exactly in the noise revealing iteration $k = k_\text{rev}$, which corresponds to 
the iteration with \eqref{eq:noise_est} being the best approximation of the unknown noise vector. 
The actual noise vector $\eta$ and the difference $\eta - \tilde{\eta}$ for $\tilde\eta$ obtained from $s_{k_\text{rev}+1}$ are compared in Figure~\ref{fig:noise_red}; see also \cite{Michenkova2013Regularization}. We see that in iteration $k_\text{rev}$, the troublesome high-frequency part of noise is perfectly removed. The remaining perturbation only contains smoothed, i.e., low-frequency part of the original noise vector. 
The match in \eqref{eq:residual_CRAIG} remains valid, up to a small perturbation, also in finite precision computations, since the noise propagation is preserved, see Section~\ref{sec:behavior}

\begin{figure}
\centering
	\begin{subfigure}[b]{.19\textwidth}
	\captionsetup{justification=centering, size = scriptsize}
	\includegraphics[width=.95\textwidth]{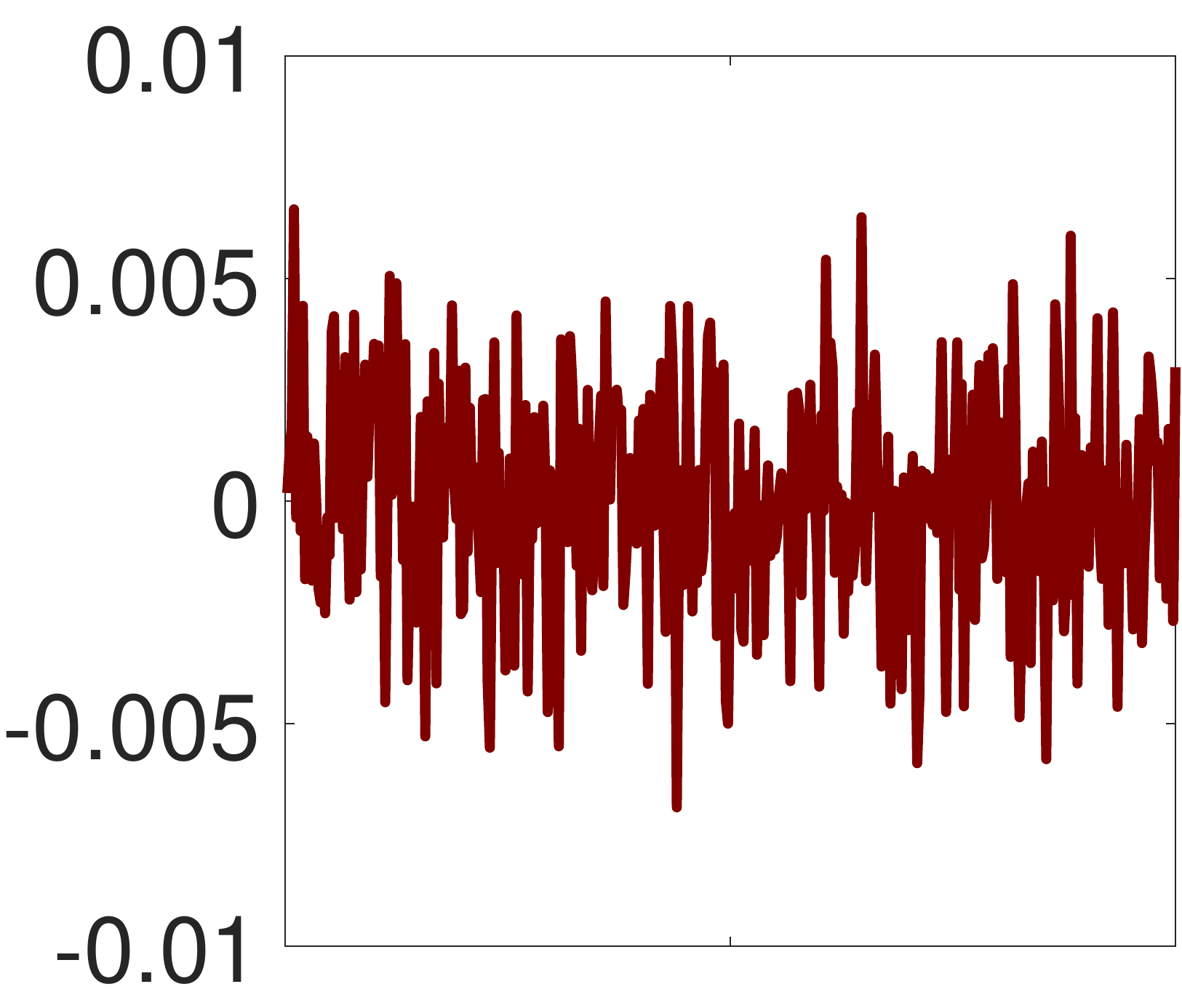}
	
	\vspace*{.2cm}
	
	\includegraphics[width=.95\textwidth]{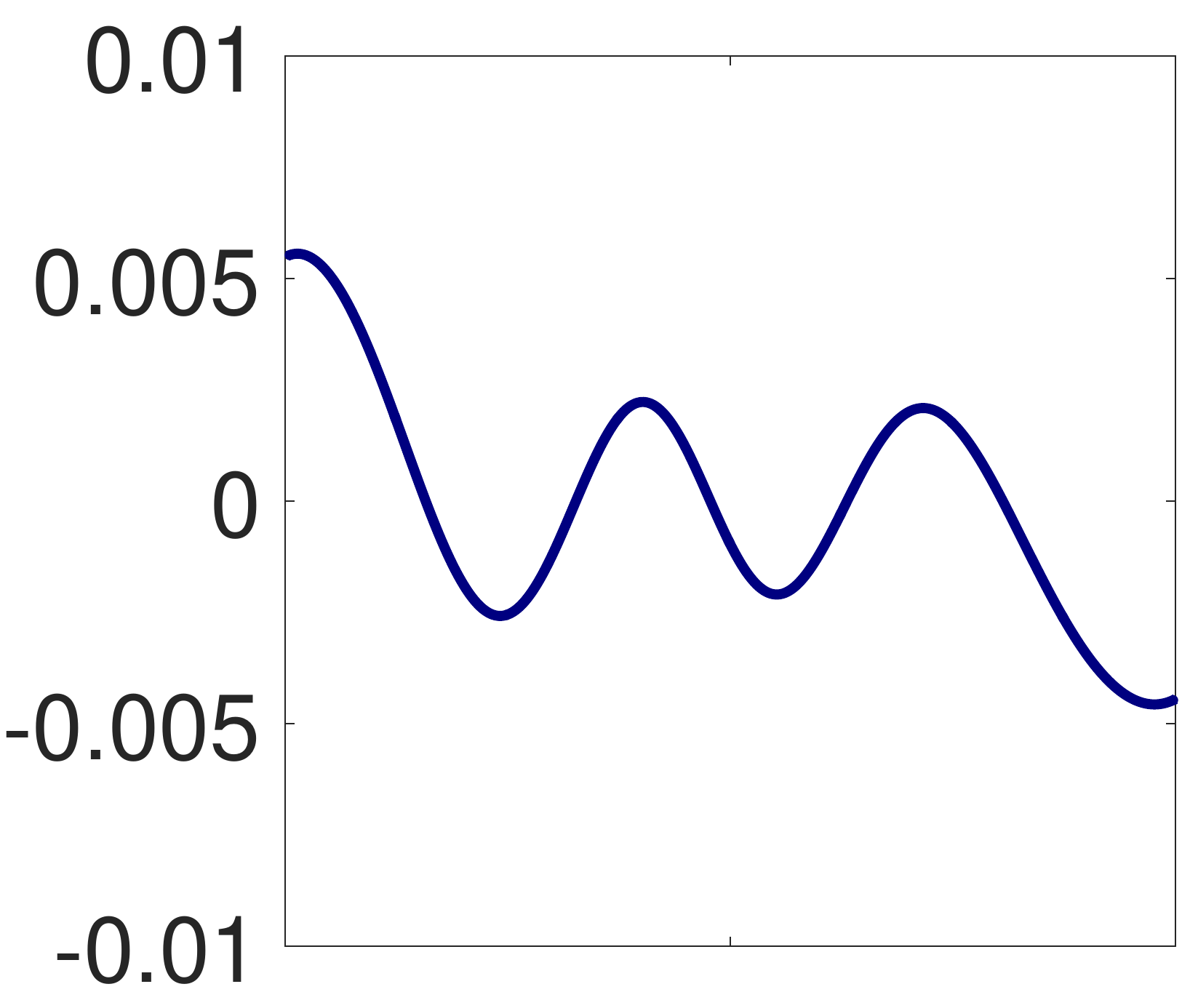}
	\caption{\texttt{shaw(400)},\\ $\delta_\text{noise} = 10^{-3}$,\\ white noise}
	\end{subfigure}
	\begin{subfigure}[b]{.19\textwidth}
	\captionsetup{justification=centering,size = scriptsize}
	\includegraphics[width=.95\textwidth]{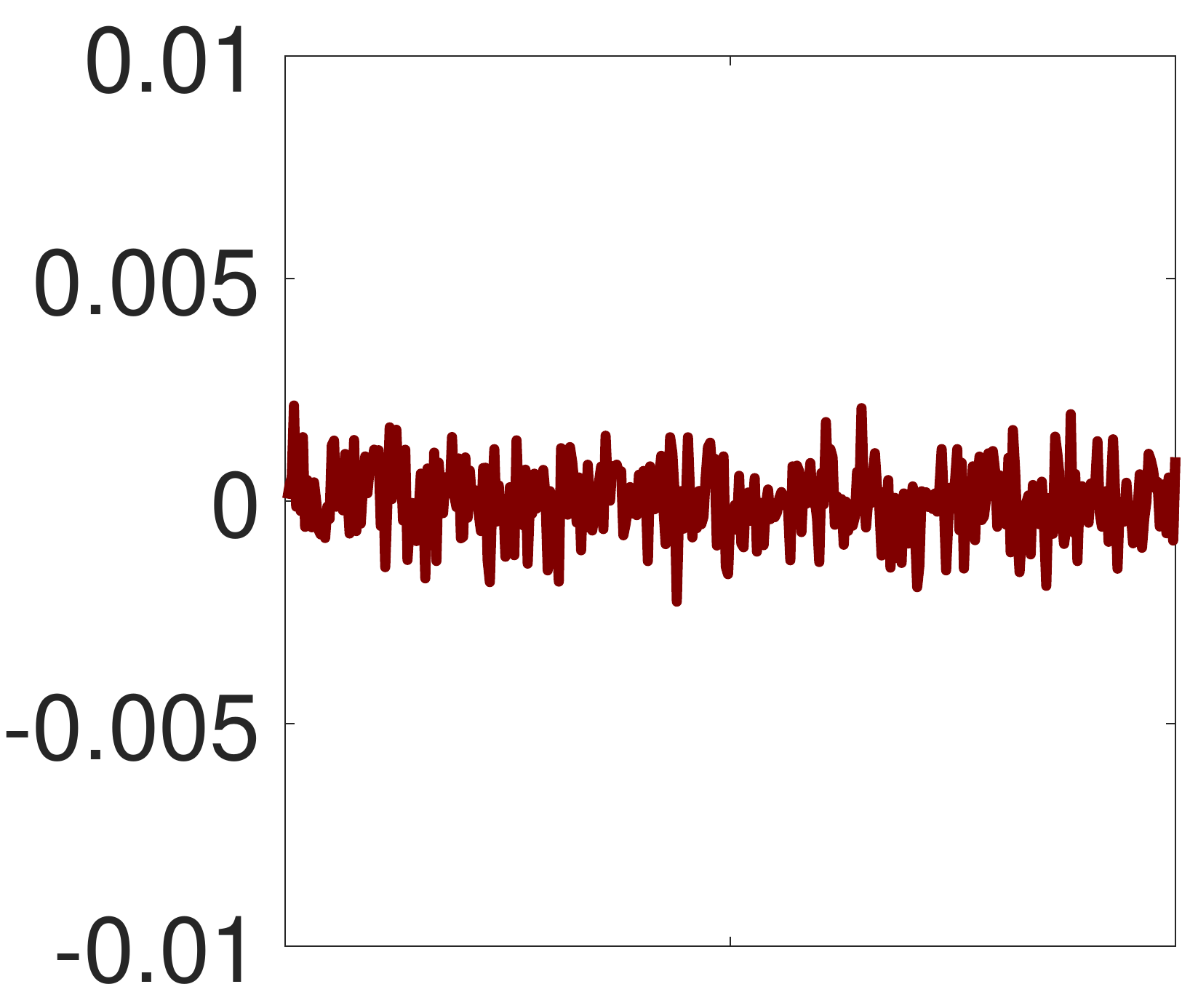}
	
	\vspace*{.2cm}
	
	\includegraphics[width=.95\textwidth]{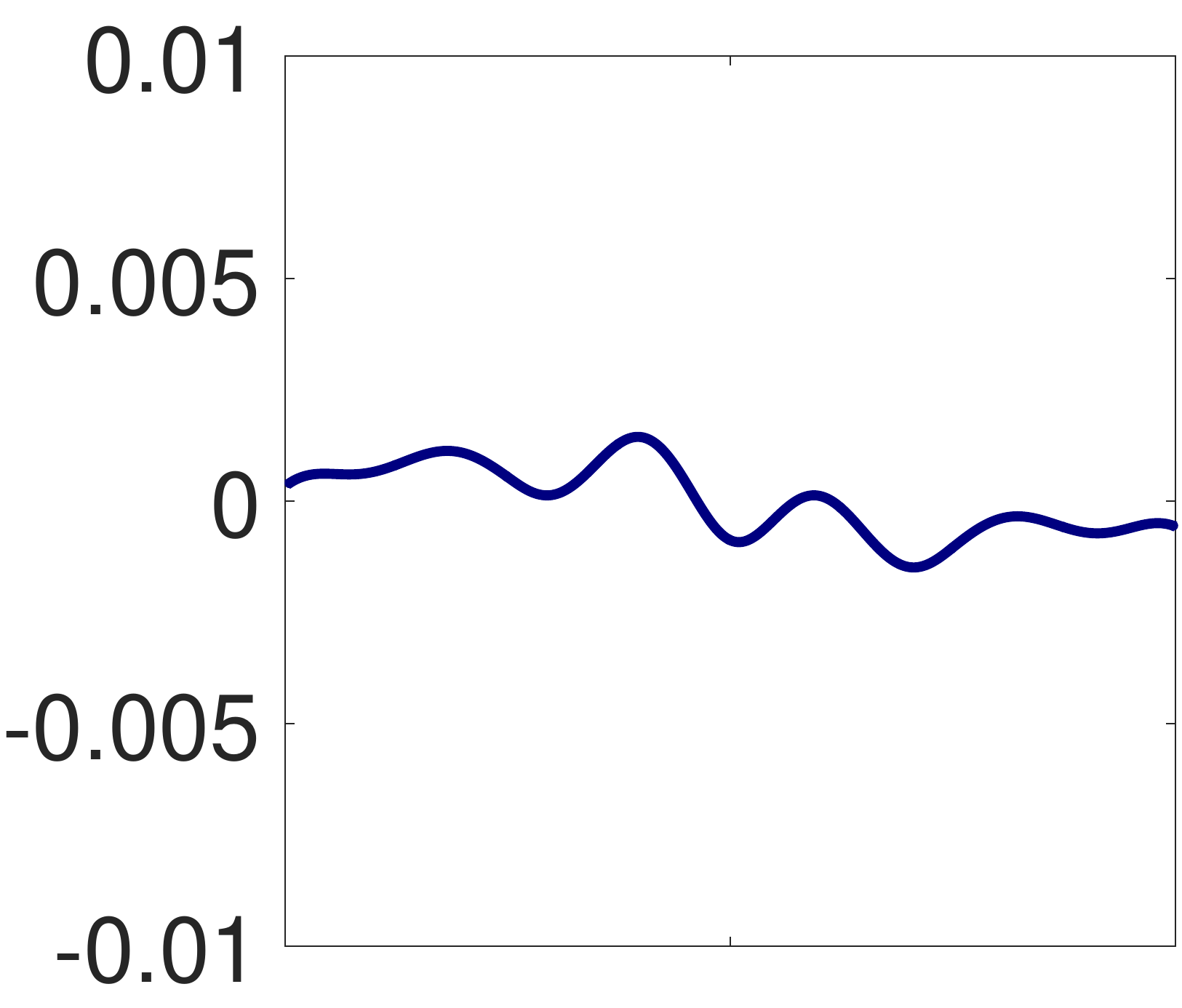}
	\caption{\texttt{phillips(400)},\\ $\delta_\text{noise} = 10^{-3}$,\\ white noise}
	\end{subfigure}
	\begin{subfigure}[b]{.19\textwidth}
	\captionsetup{justification=centering,size = scriptsize}
	\includegraphics[width=.95\textwidth]{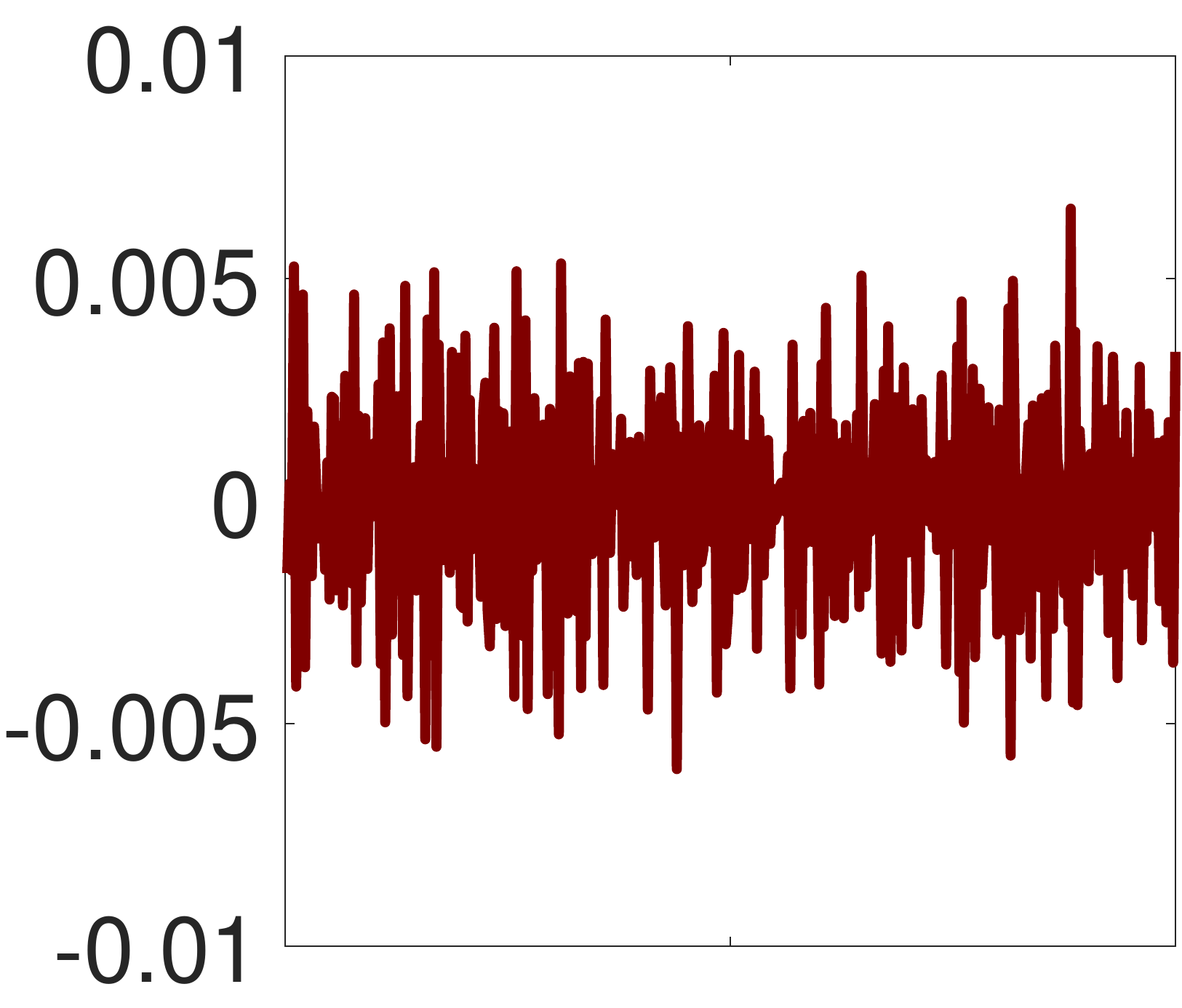}
	
	\vspace*{.2cm}
	
	\includegraphics[width=.95\textwidth]{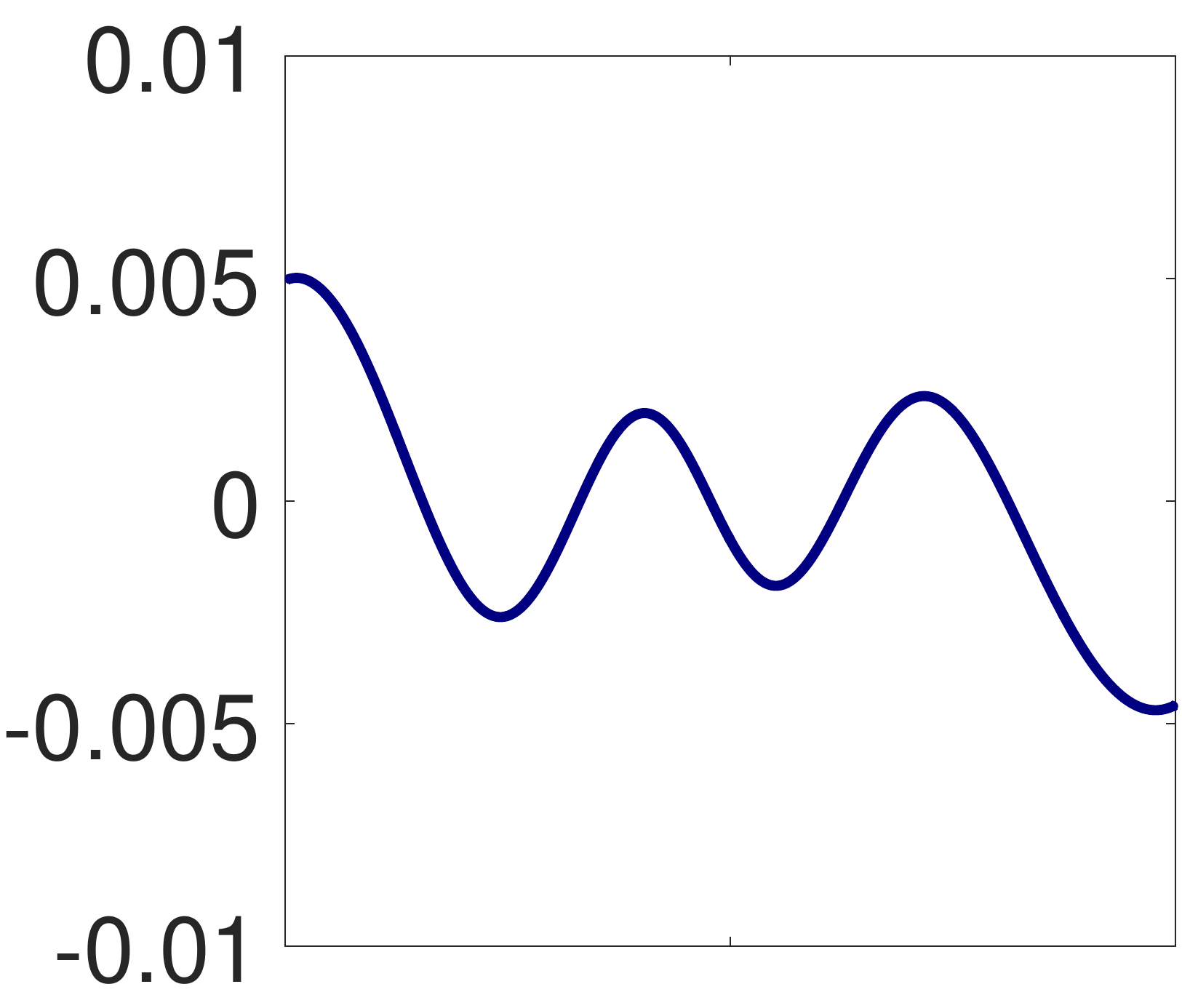}
	\caption{\texttt{shaw(400)},\\ $\delta_\text{noise} = 10^{-3}$,\\ violet noise}
	\end{subfigure}
	\begin{subfigure}[b]{.19\textwidth}
	\captionsetup{justification=centering,size = scriptsize}
	\includegraphics[width=.95\textwidth]{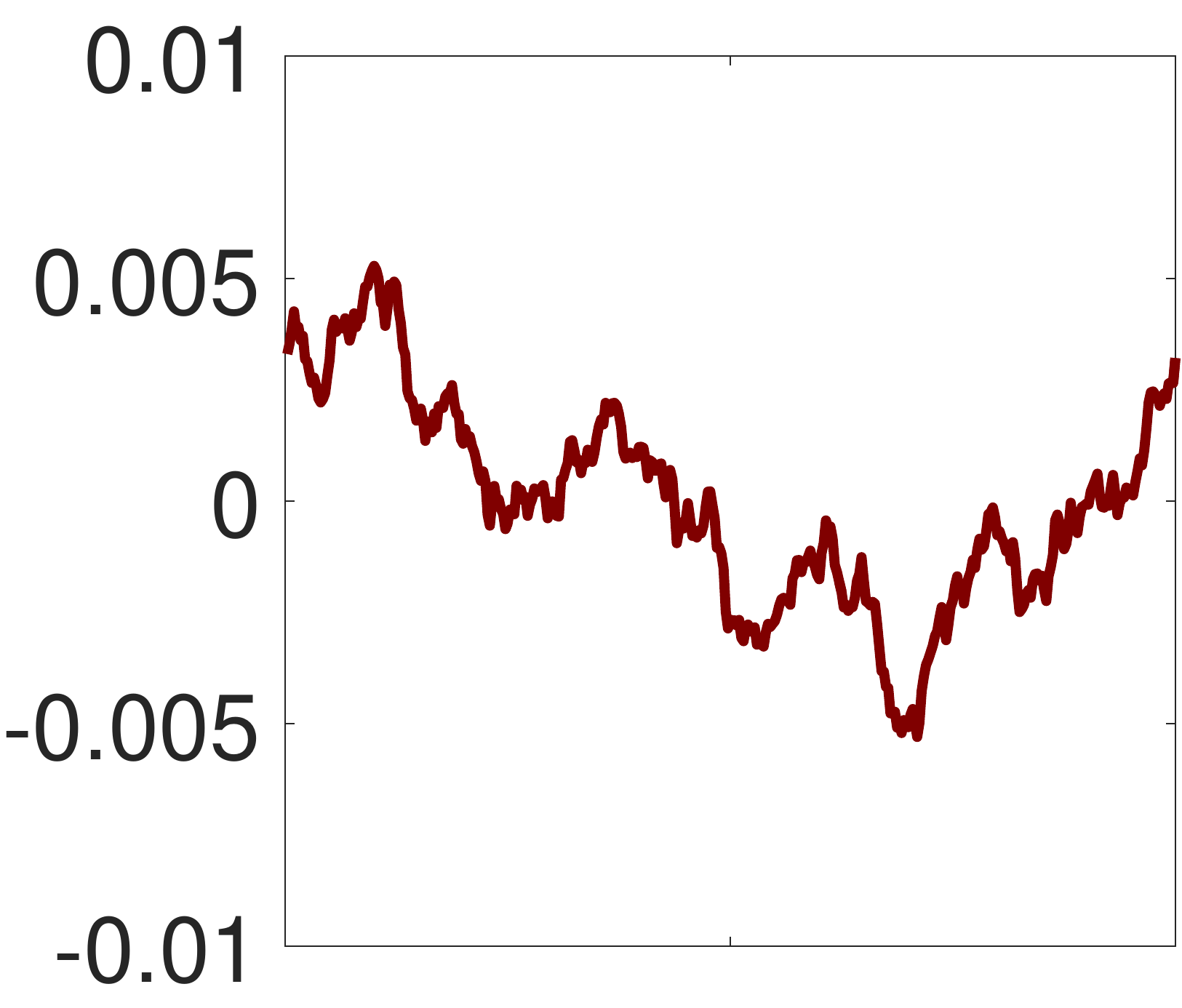}
	
	\vspace*{.2cm}
	
	\includegraphics[width=.95\textwidth]{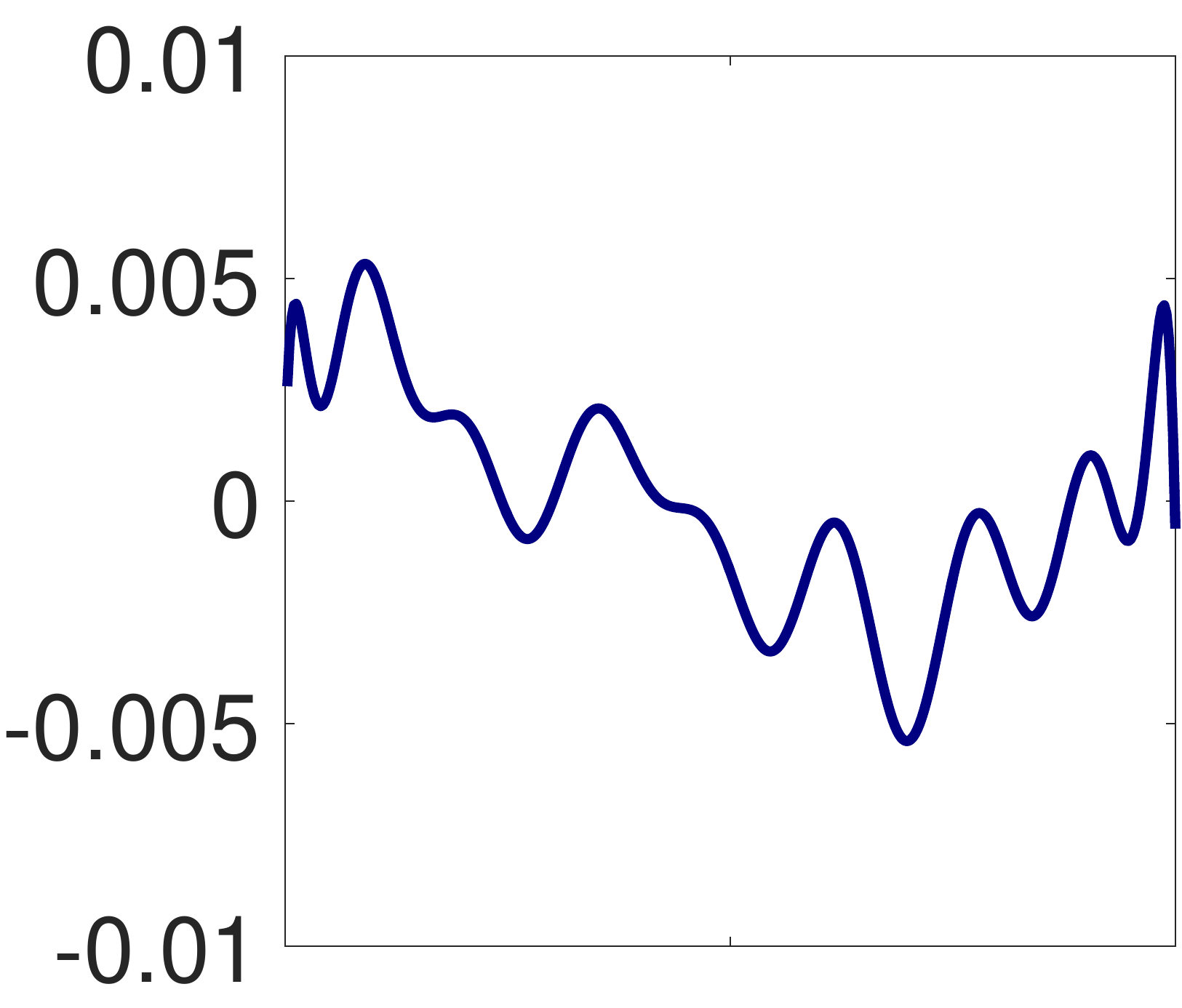}
	\caption{\texttt{shaw(400)},\\ $\delta_\text{noise} = 10^{-3}$,\\ red noise}
	\end{subfigure}
	\begin{subfigure}[b]{.19\textwidth}
	\captionsetup{justification=centering,size = scriptsize}
	\includegraphics[width=.91\textwidth]{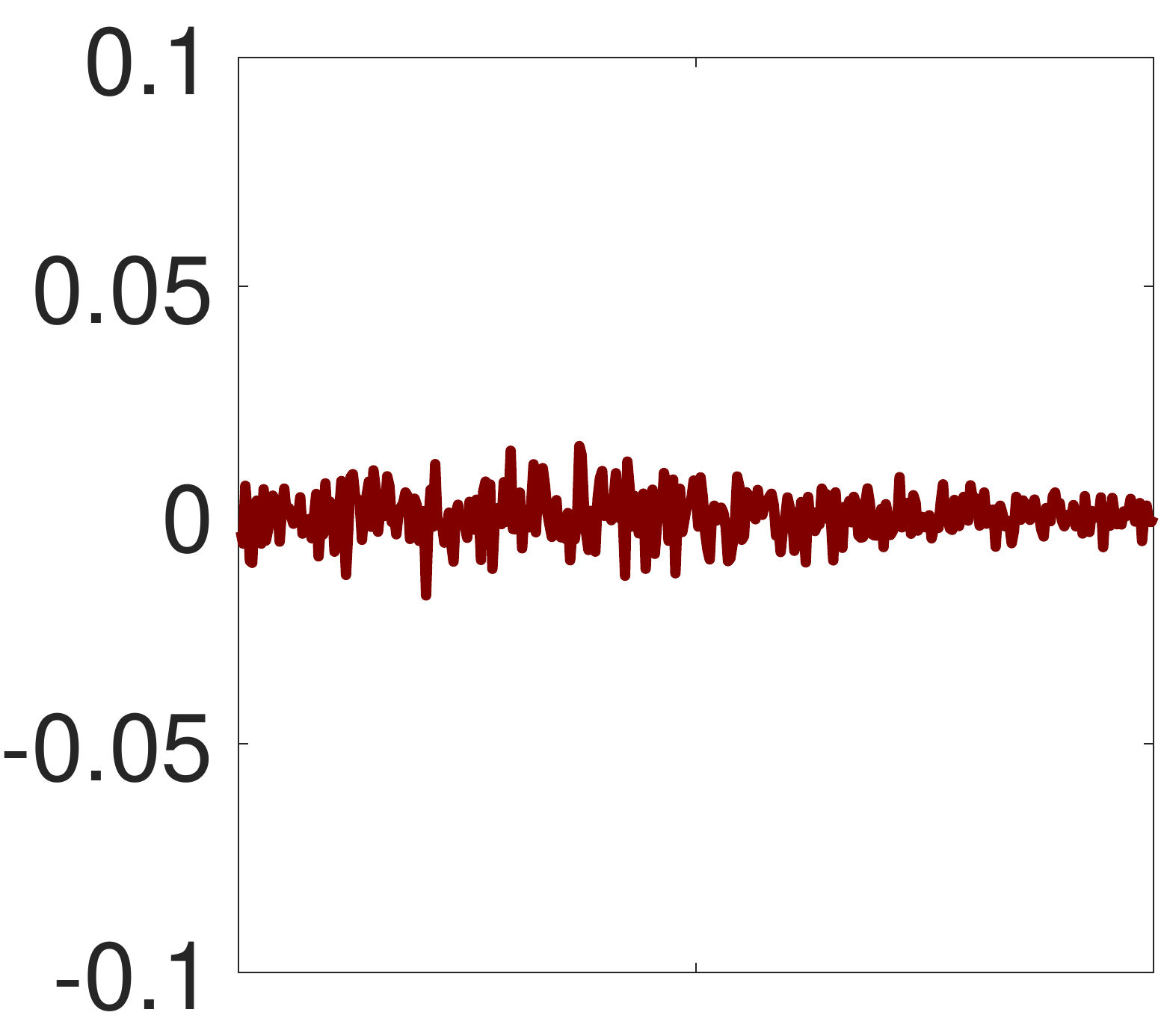}
	
	\vspace*{.2cm}
	
	\includegraphics[width=.91\textwidth]{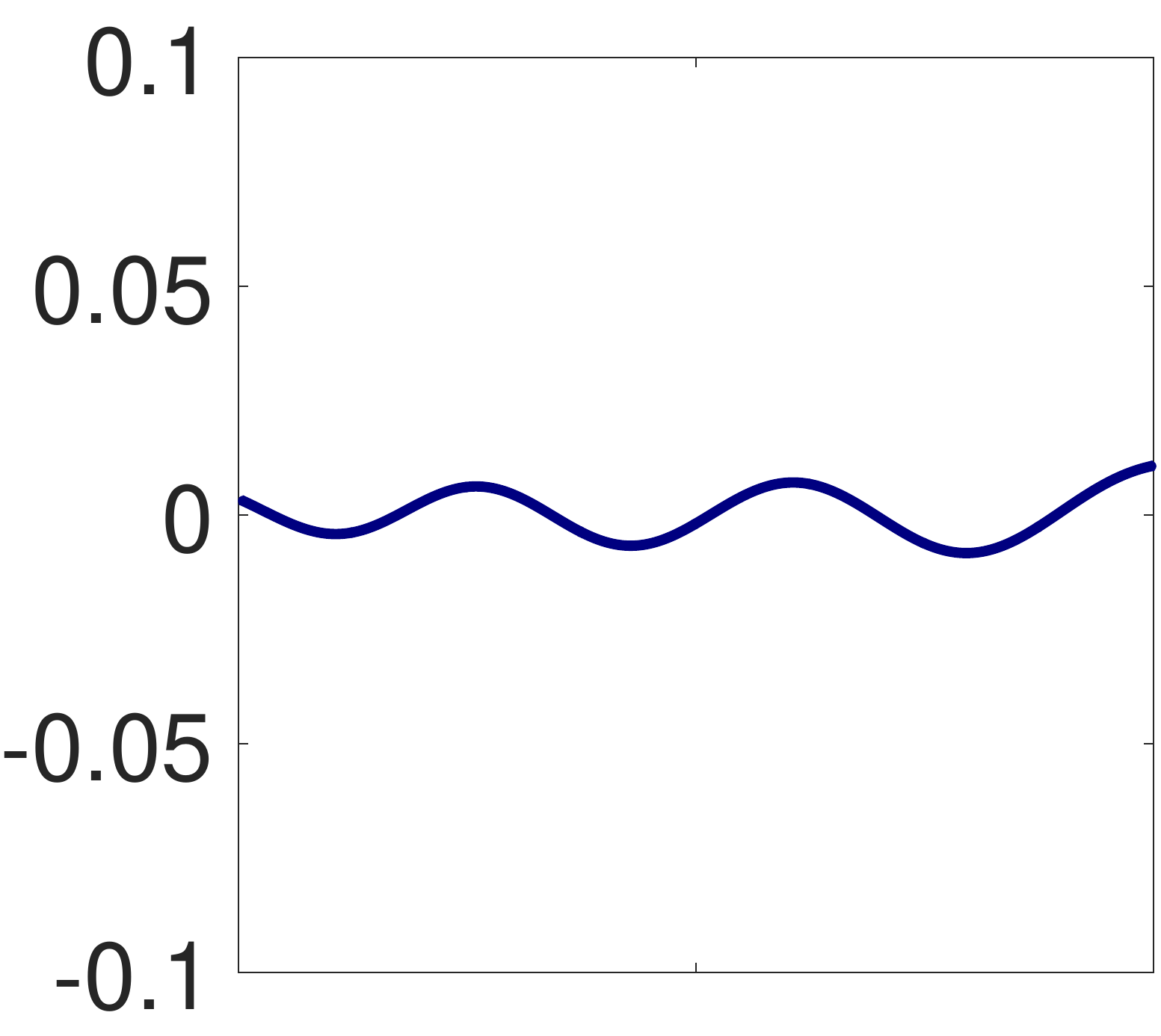}
	\caption{\texttt{gravity(400)},\\ $\delta_\text{noise} = 10^{-3}$,\\ Poisson noise}
	\end{subfigure}
\caption{Illustration of the quality of the noise vector approximation $\tilde{\eta}$ obtained by \eqref{eq:noise_est} for $k={k_\text{rev}+1}$ on various test problems and various characteristics of noise. Upper: The original noise vector $\eta$. Lower: The difference $\eta - \tilde{\eta}$.}
\label{fig:noise_red}
\end{figure}

Note that due to different frequency characteristic of $\eta$ and $s_{k+1}^\text{LF}$ for small $k$, there is a~relatively small cancellation between them and
\begin{equation}
\|s_{k+1}^\text{LF}\|^{2}+\|\varphi_k(0)\eta\|^{2}\approx 1.
\end{equation}
This gives 
\begin{equation}
\|(b - \tilde\eta) - Ax\|  = \|\varphi_k(0)^{-1}s_{k+1}^\text{LF}\|\approx|\varphi_k(0)|^{-1}\sqrt{1-\|\varphi_k(0)\eta\|^{2}} =\sqrt{|\varphi_k(0)|^{-2} -\|\eta\|^{2}}\label{eq:rem_noise_size_est}
\end{equation}
supporting our expectation that the size of the remaining perturbation depends on how closely the inverse amplification factor $|\varphi_k(0)|^{-1}$ approaches $\|\eta\|$.

We may also conclude that for ill-posed problems with a smoothing operator $A$, the minimal error $\|x_k^\text{CRAIG}-x\|$ is reached approximately at the iteration with the maximal noise revealing, i.e.,  with the minimal residual. This is confirmed by numerical experiments in Figure~\ref{fig:error_all} comparing $\|x_k^\text{CRAIG}-x\|$ 
with $\|r_k^\text{CRAIG}\|$ for various test problems and noise characteristics, both with and without reorthogonalization.
\begin{figure}
        \centering
        \begin{subfigure}[b]{0.32\textwidth}
        \captionsetup{justification=centering}
                \includegraphics[width=\textwidth]{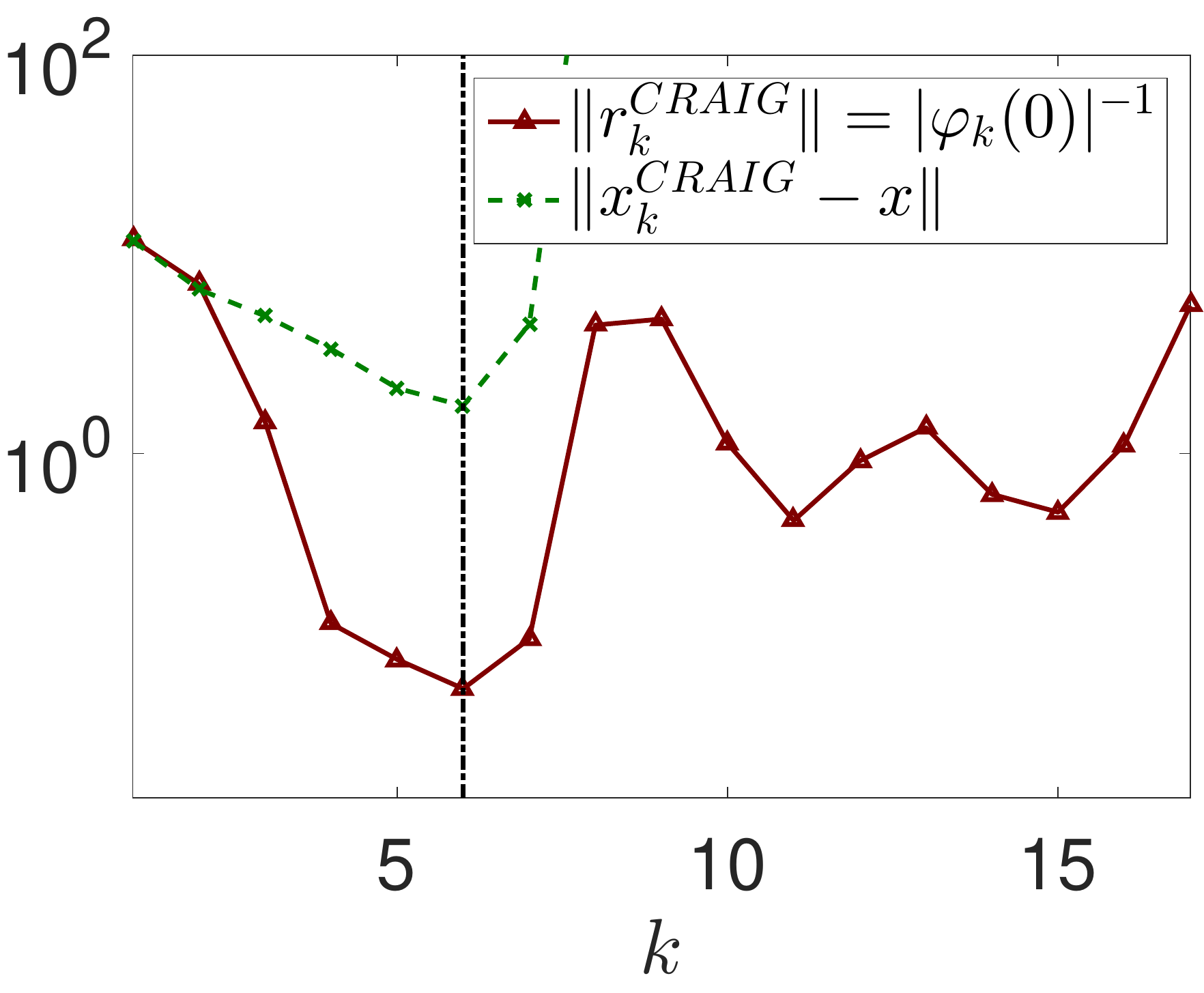}
                \caption{\texttt{shaw(400)},\\ $\delta_\text{noise} = 10^{-3}$, white noise}\label{fig:error}
        \end{subfigure}
        \begin{subfigure}[b]{0.32\textwidth}
        \captionsetup{justification=centering}
                \includegraphics[width=\textwidth]{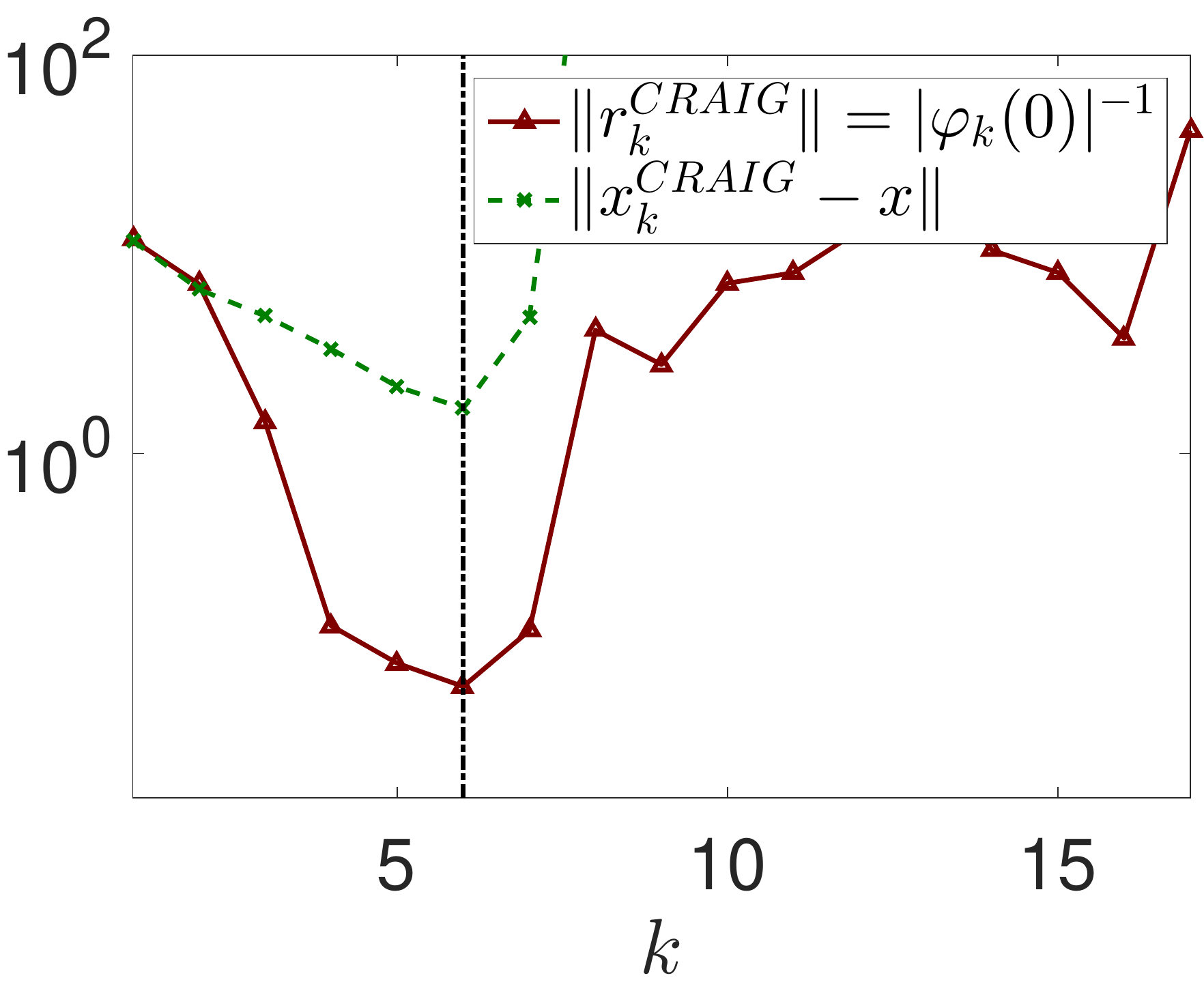}
                \caption{\texttt{shaw(400)}, $\delta_\text{noise} = 10^{-3}$, violet noise}
        \end{subfigure}
        \begin{subfigure}[b]{0.32\textwidth}
        \captionsetup{justification=centering}
                \includegraphics[width=\textwidth]{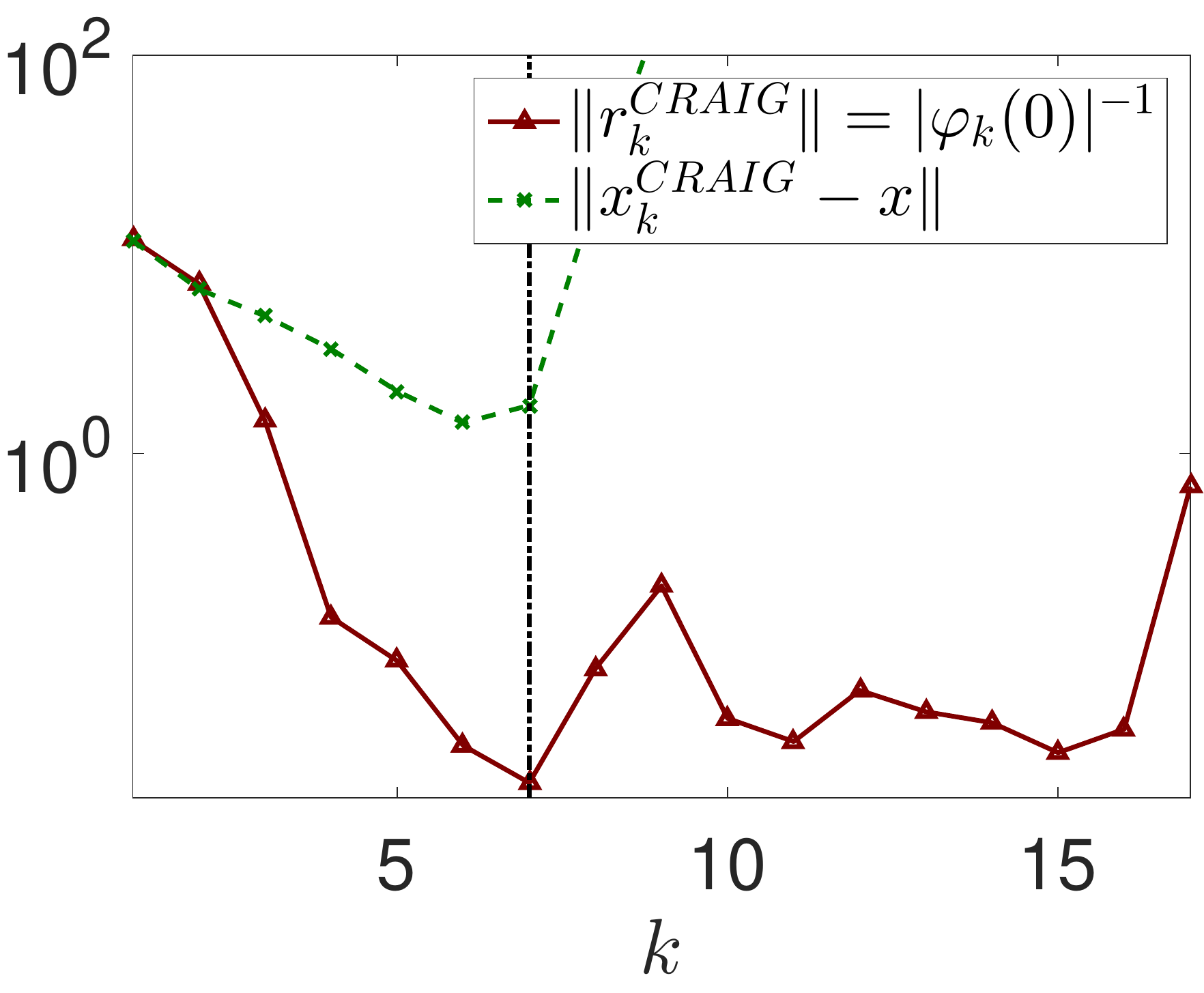}
                \caption{\texttt{shaw(400)}, $\delta_\text{noise} = 10^{-3}$, red noise}
        \end{subfigure}

        \vspace*{.3cm}

        \begin{subfigure}[b]{0.32\textwidth}
        \captionsetup{justification=centering}
                \includegraphics[width=\textwidth]{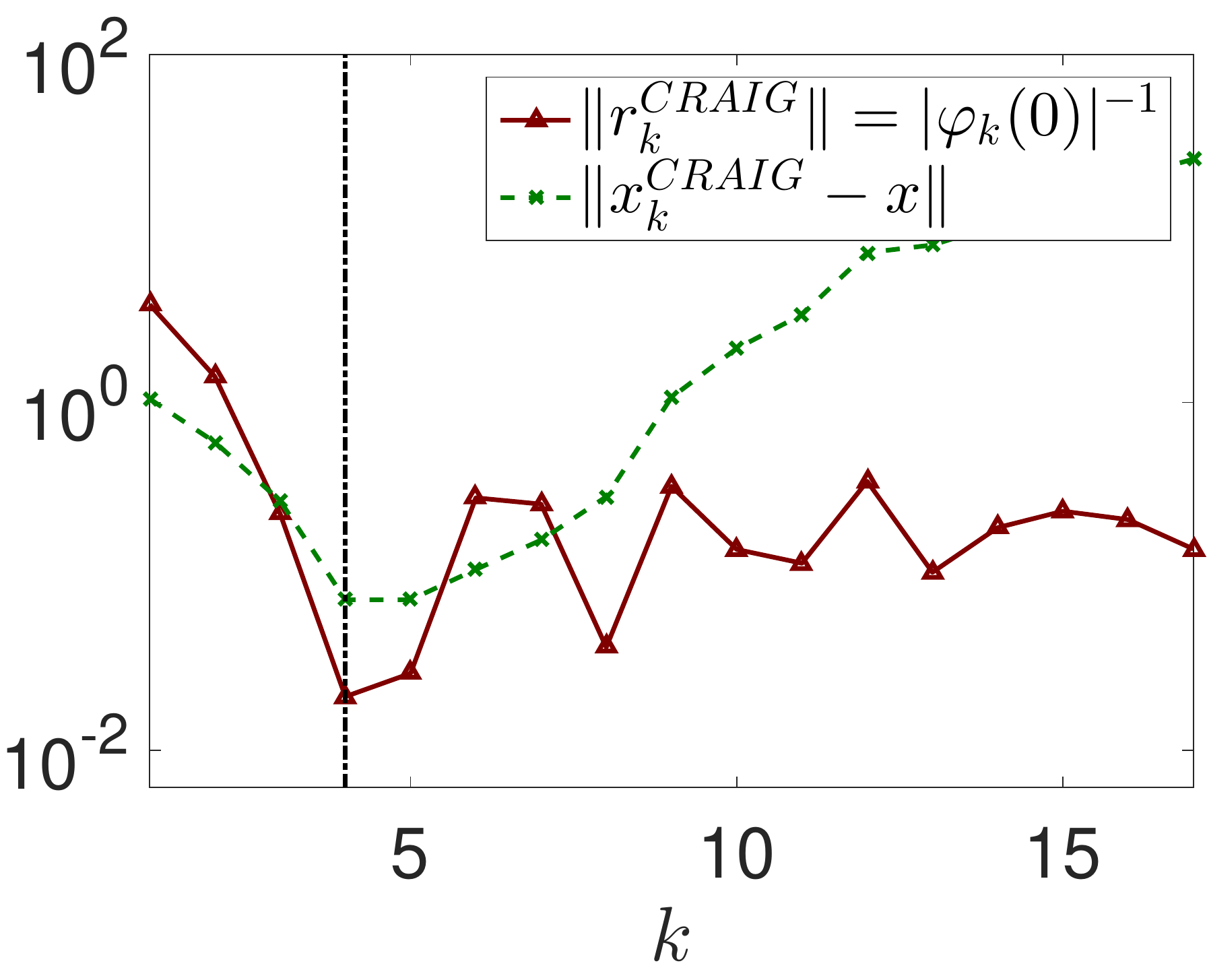}
                \caption{\texttt{phillips(400)}, $\delta_\text{noise} = 10^{-3}$, white noise}
        \end{subfigure}
        \begin{subfigure}[b]{0.32\textwidth}
        \captionsetup{justification=centering}
                \includegraphics[width=\textwidth]{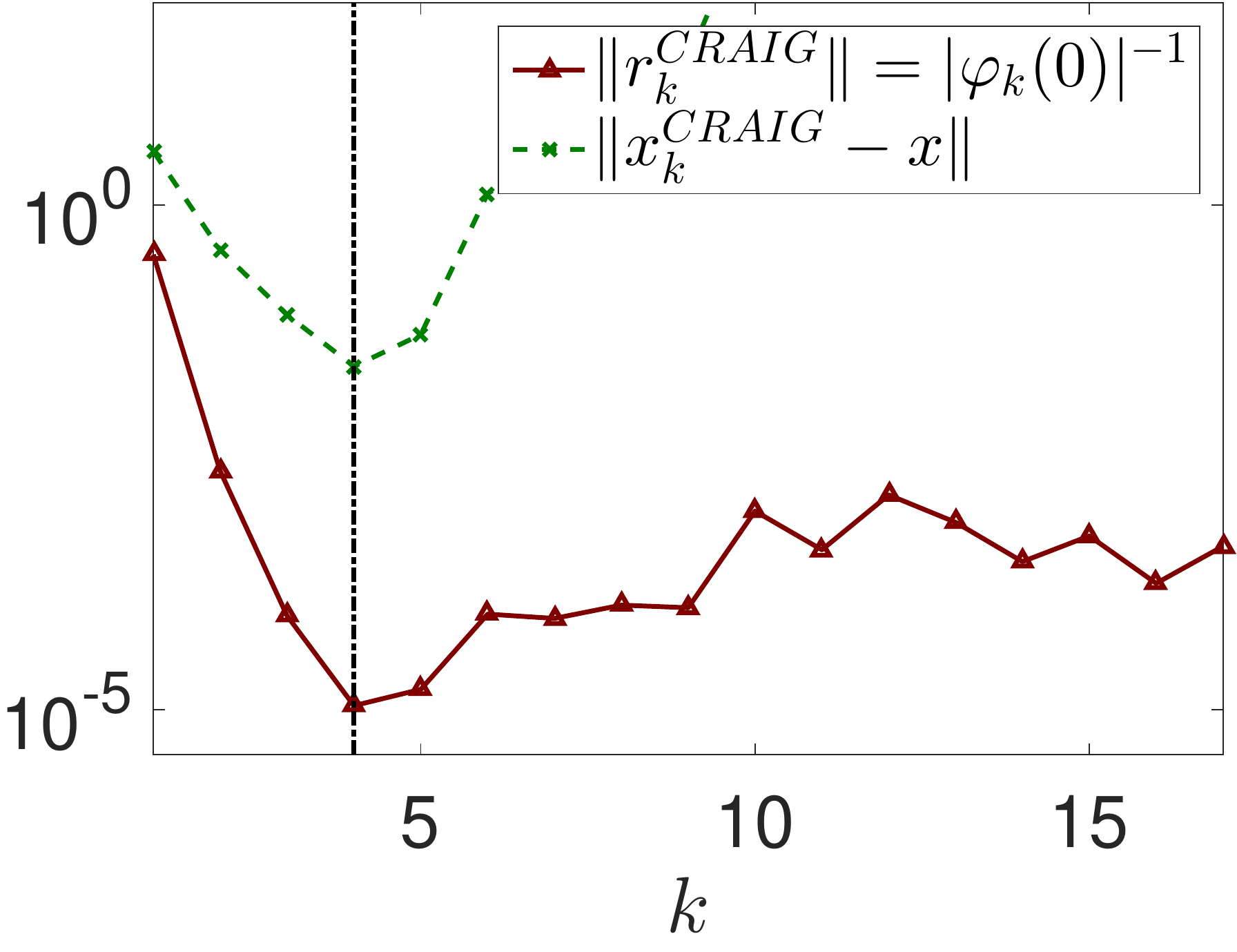}
                \caption{\texttt{foxgood(400)}, $\delta_\text{noise} = 10^{-6}$, white noise}
        \end{subfigure}
        \begin{subfigure}[b]{0.32\textwidth}
        \captionsetup{justification=centering}
                \includegraphics[width=\textwidth]{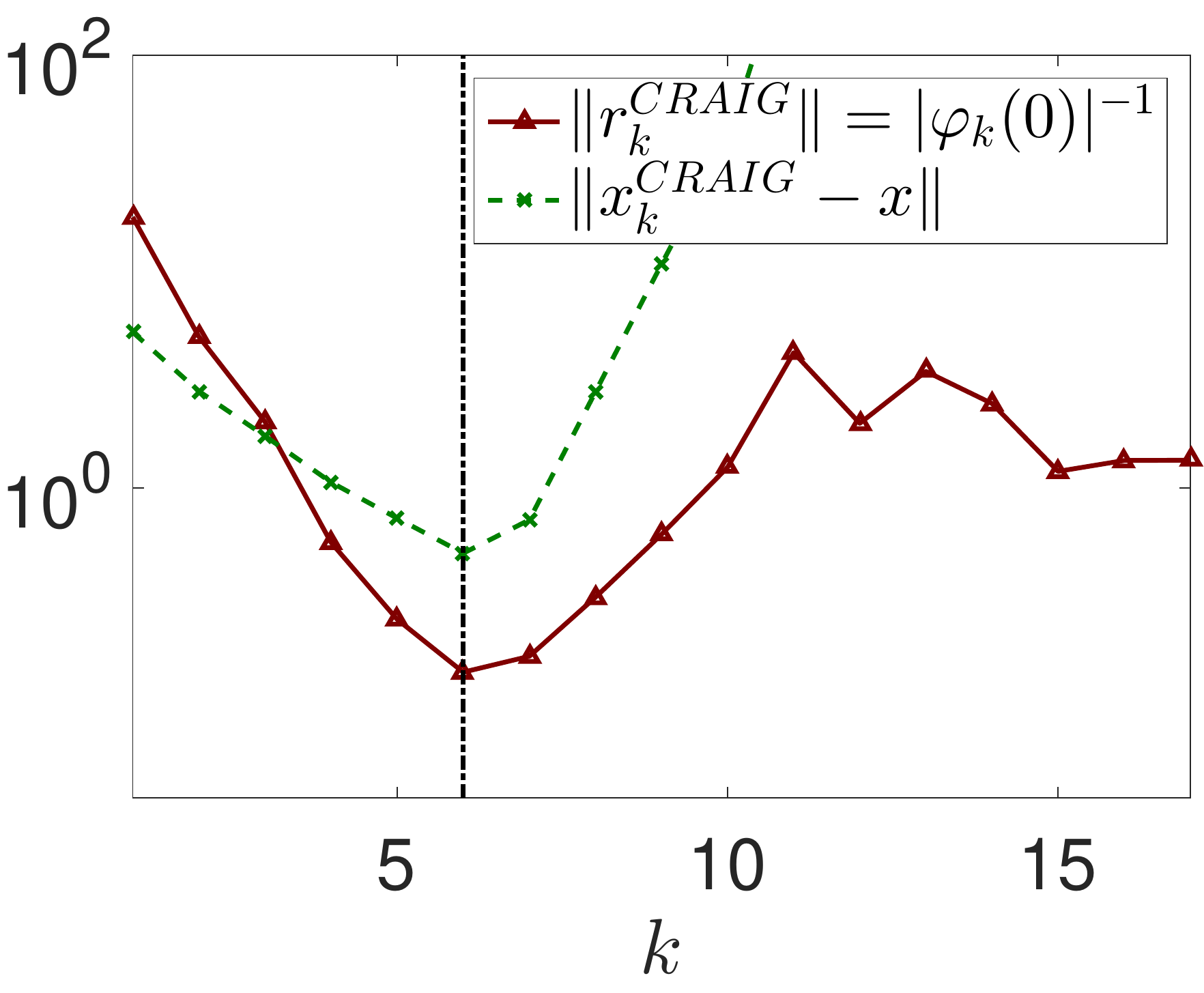}
                \caption{\texttt{gravity(400)},\\ $\delta_\text{noise} = 10^{-3}$, Poisson}
                \end{subfigure}
                 \begin{subfigure}[b]{0.32\textwidth}
               \captionsetup{justification=centering}
                \includegraphics[width=\textwidth]{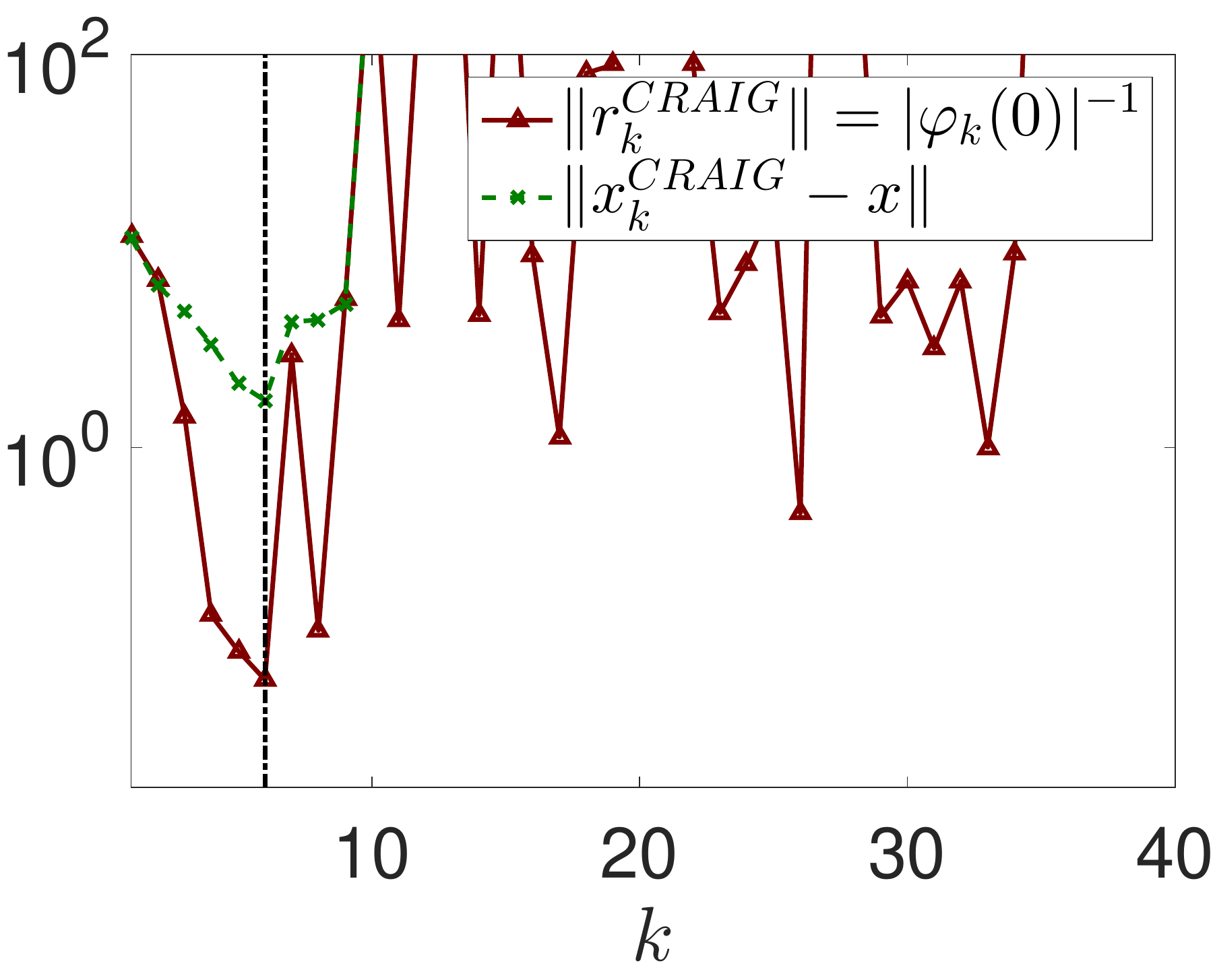}
                \caption{\texttt{shaw(400)},\\ $\delta_\text{noise} = 10^{-3}$, white noise}
        \end{subfigure}
        \begin{subfigure}[b]{0.32\textwidth}
        \captionsetup{justification=centering}
                \includegraphics[width=\textwidth]{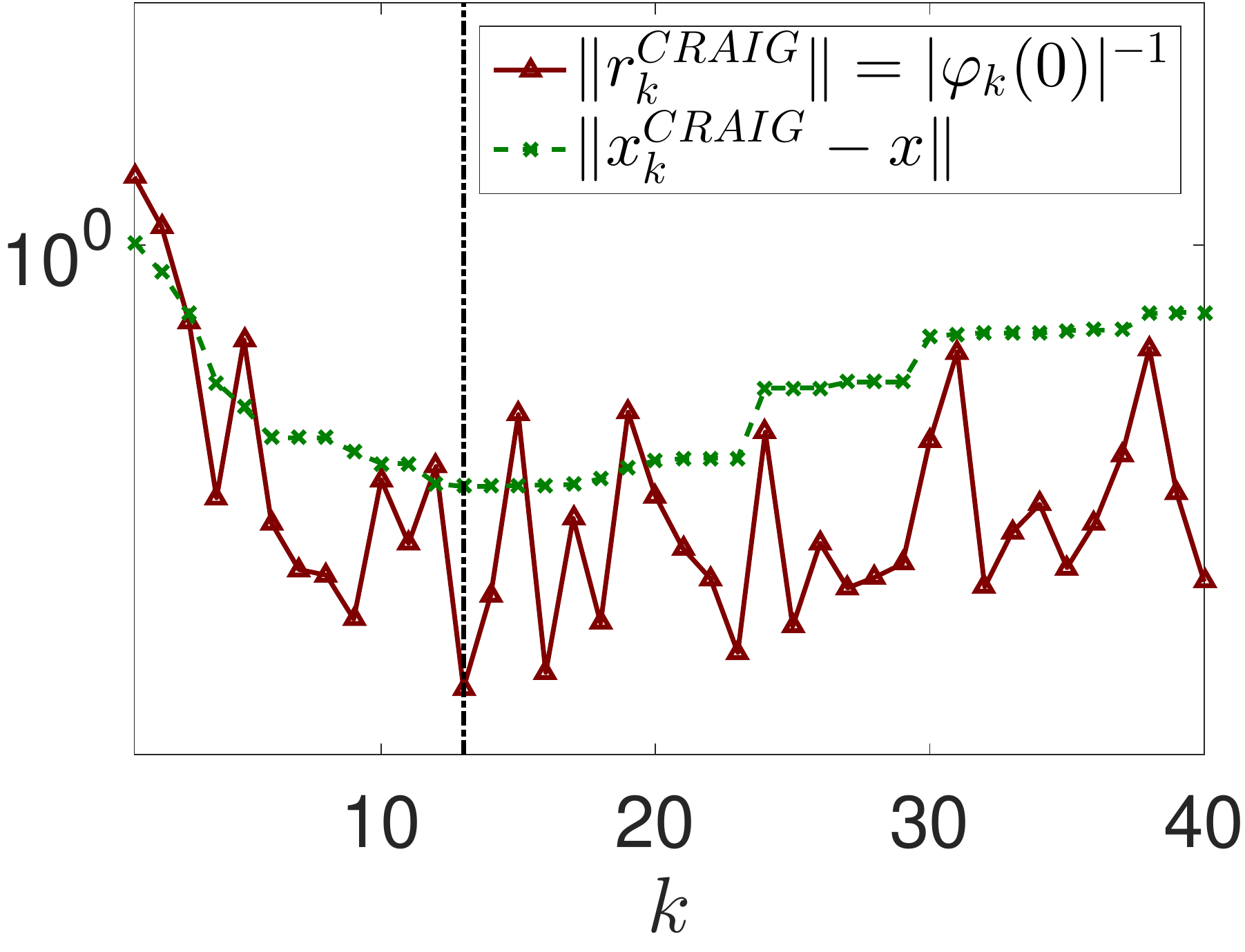}
                \caption{\texttt{phillips(400)},\\ $\delta_\text{noise} = 10^{-5}$, white noise}
        \end{subfigure}
        \begin{subfigure}[b]{0.32\textwidth}
        \captionsetup{justification=centering}
                \includegraphics[width=\textwidth]{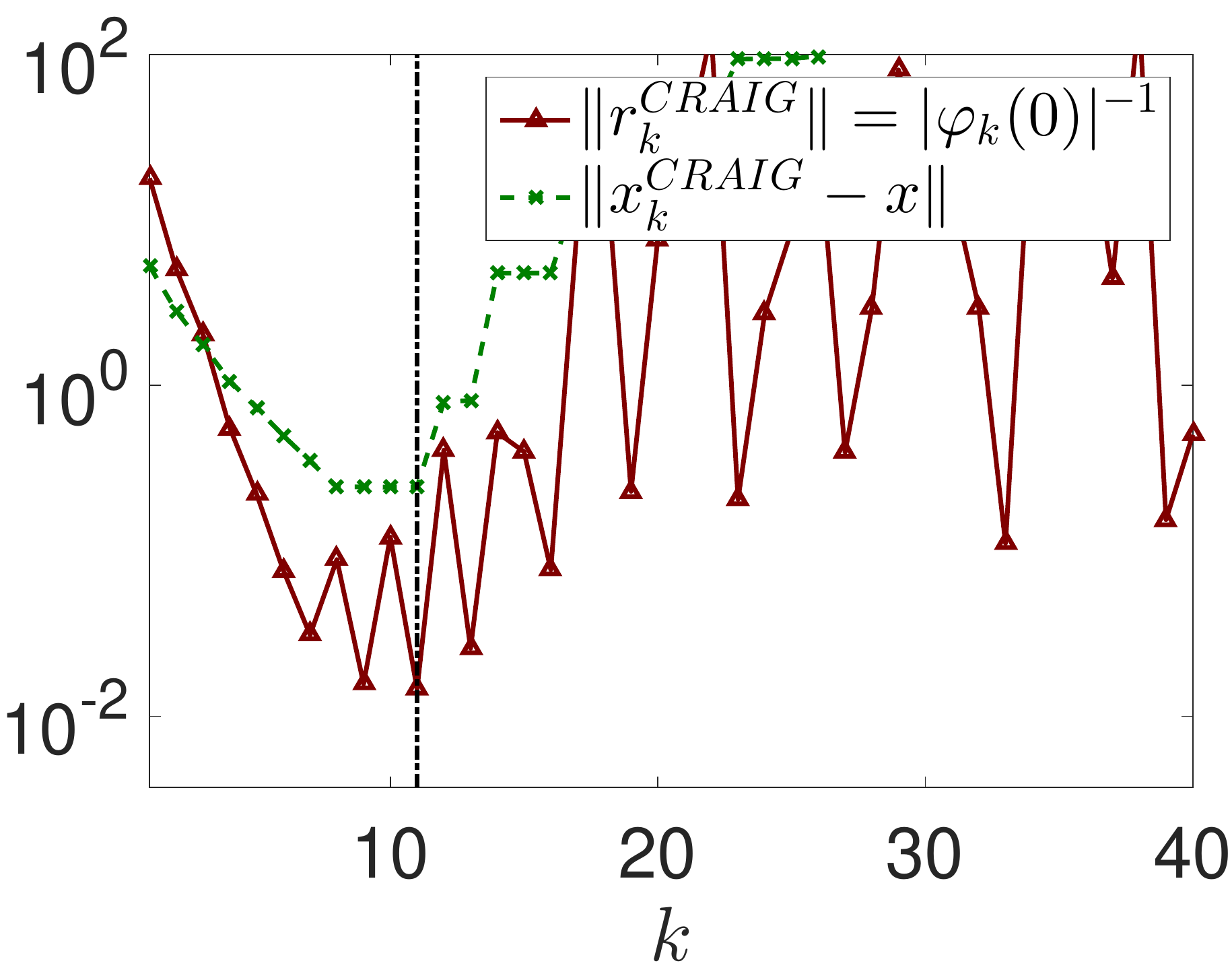}
                \caption{\texttt{gravity(400)},\\ $\delta_\text{noise} = 10^{-4}$, Poisson}
        \end{subfigure}
               \caption{Comparison of the size of the residual and the size of the error in CRAIG for various test problems with various noise characteristics. The minimal error is achieved approximately when the residual is minimized (vertical line). In Figures (g)-(i) without reorthogonalization.}
\label{fig:error_all}
\end{figure}

\subsection{LSQR residuals}\label{sec:residuals}
Whereas for CRAIG, the residual is just a~scaled left bidiagonalization vector, for LSQR  it is a~linear combination of all previously computed left bidiagonalization  vectors. Indeed,
\begin{equation}
r_k^\text{LSQR} = b - AW_ky_k^\text{LSQR} = S_{k+1}\left( \beta_1e_1 - L_{k+}y_k^\text{LSQR}\right),\label{eq:residual_LSQR}
\end{equation}
 see \eqref{eq:residual}. The entries of the residual of the projected problem
\begin{equation}
p_k^\text{LSQR}\equiv  \beta_1e_1 - L_{k+}y_k^\text{LSQR},
\label{eq:pk}
\end{equation}
see \eqref{eq:LSQR_projected},  represent the coefficients of the linear combination in \eqref{eq:residual_LSQR}. The following proposition shows the relation between the coefficients and the amplification factor $\varphi_k(0)$.
\begin{prop}\label{th:2}
Consider the first $k$ steps of the Golub-Kahan iterative bidiagonalization. Let $r_k^\text{LSQR} = b - Ax_k^\text{LSQR}$, where $x_k^\text{LSQR}$ is  the approximation defined in \eqref{eq:projected} and \eqref{eq:LSQR_projected}. Then
\begin{equation}
r_k^\text{LSQR} = \frac{1}{\sum_{l=0}^k\varphi_l(0)^{2}}\sum_{l=0}^k\varphi_l(0)s_{l+1}.\label{eq:residual_lsqr_sum}
\end{equation}
Consequently,
\begin{equation}
\|r_k^\text{LSQR}\| = \frac{1}{\sqrt{\sum_{l=0}^k\varphi_l(0)^{2}}}.
\end{equation}
\end{prop}
\begin{proof}
Since
\begin{equation}
y_k^\text{LSQR} = \underset{y}{\operatorname{argmin}}\|\beta_1e_1 - L_{k+}y\|,
\end{equation}
we get
\begin{equation}
L_{k+}^Tp_k^\text{LSQR} = 0.
\end{equation}
It follows from the structure of the matrix $L_{k+}$ that the entries of $p_k^\text{LSQR}$ satisfy
\begin{equation}
\alpha_le_l^Tp_k^\text{LSQR} + \beta_{l+1}e_{l+1}^Tp_k^\text{LSQR} = 0, \qquad \text{for} \qquad l = 1,\ldots,k.
\end{equation}
Thus
\begin{equation}
p_k^\text{LSQR} = c_k\left[
\begin{array}{c}
\varphi_0(0)\\
\varphi_1(0)\\
\vdots\\
\varphi_k(0)
\end{array}
\right],\label{eq:pk_k_with_ck}
\end{equation}
where $c_k$ is a~factor that changes with $k$. From \eqref{eq:res_norm_LSQR} and \eqref{eq:craig_res} it follows  that
\begin{equation}
\|p_k^\text{LSQR}\| = \|r_k^\text{LSQR}\| = \frac{1}{\sqrt{\sum_{l=0}^k\varphi_l(0)^{2}}}.\label{eq:pk_norm}
\end{equation}
By comparing \eqref{eq:pk_k_with_ck} and \eqref{eq:pk_norm},
we  get
\begin{equation}
c_k = \frac{1}{\sum_{l=0}^k\varphi_l(0)^{2}},\label{eq:ck}
\end{equation}
which together with \eqref{eq:residual_LSQR} and \eqref{eq:pk} yields \eqref{eq:residual_lsqr_sum}.\qed
\end{proof}

In other words, Proposition~\ref{th:2} says that the coefficients of the linear combination \eqref{eq:residual_LSQR} follow the behavior of the amplification factor  in the sense that representation of a~particular left bidiagonalization vector $s_{l+1}$ in the residual $r_k^\text{LSQR}$, $k\geq l$, is proportional to the amount of propagated non-smoothed part of noise $\eta$ in this vector.

\medskip
Relation \eqref{eq:residual_lsqr_sum} also suggests that the norm-minimizing process (LSQR) and the corresponding Galerkin process (CRAIG) provide similar solutions whenever
\begin{equation}
\frac{\varphi_k(0)^{2}}{\sum_{l=0}^{k}\varphi_l(0)^{2}}\approx 1,
\end{equation}
i.e., whenever the noise revealing in the last left bidiagonalization vector $s_{k+1}$ is much more significant than in all previous left bidiagonalization vectors $s_1,\ldots s_k$, i.e., typically before we reach the noise revealing iteration. This is confirmed numerically in Figure~\ref{fig:error_craig_lsqr}, comparing the semiconvergence curves of CRAIG and LSQR.
\begin{figure}
        \centering
        \begin{subfigure}[b]{0.32\textwidth}
        \captionsetup{justification=centering}
                \includegraphics[width=\textwidth]{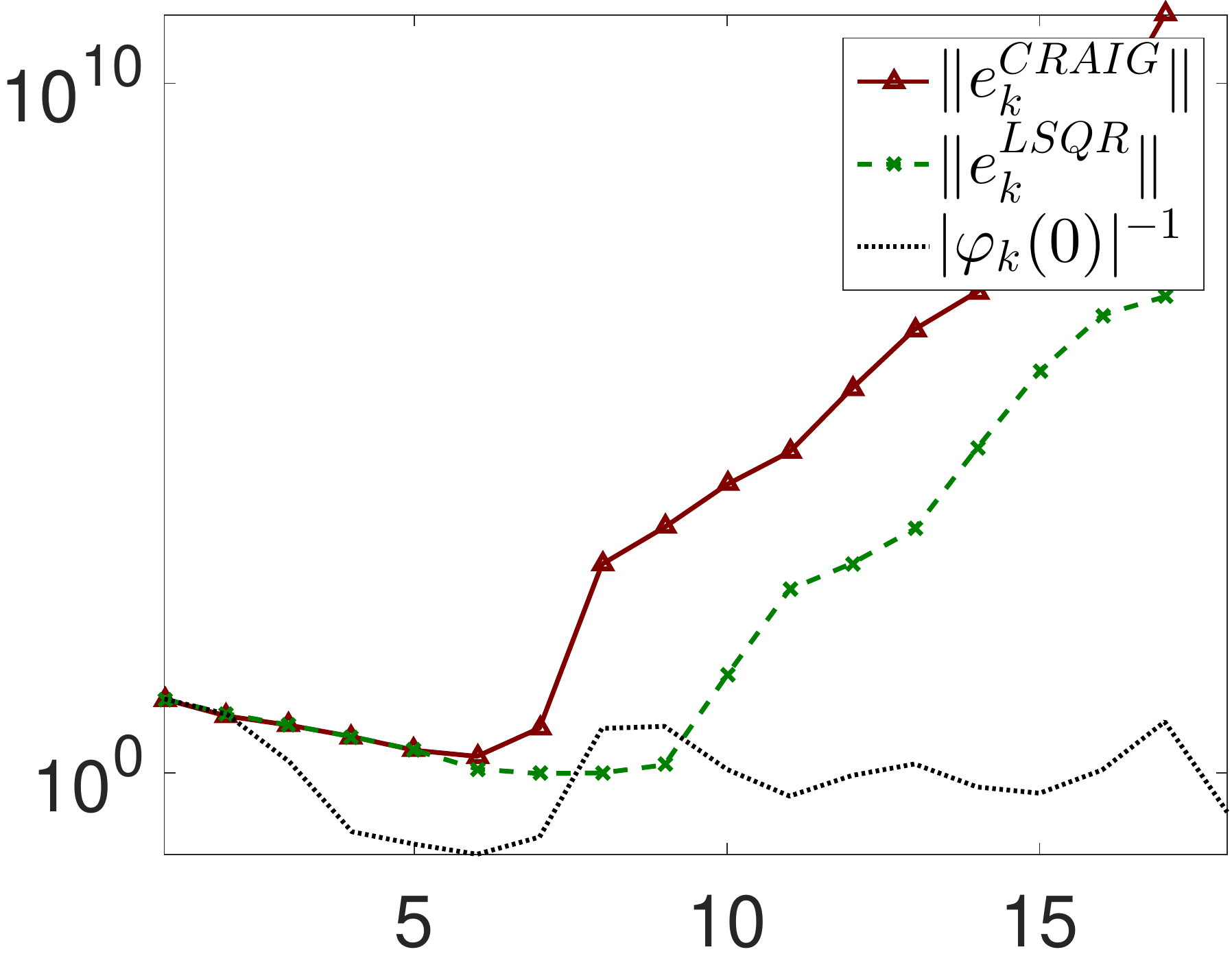}
                \caption{\texttt{shaw(400)},\\ $\delta_\text{noise} = 10^{-3}$, white noise}
        \end{subfigure}
        \begin{subfigure}[b]{0.32\textwidth}
        \captionsetup{justification=centering}
                \includegraphics[width=\textwidth]{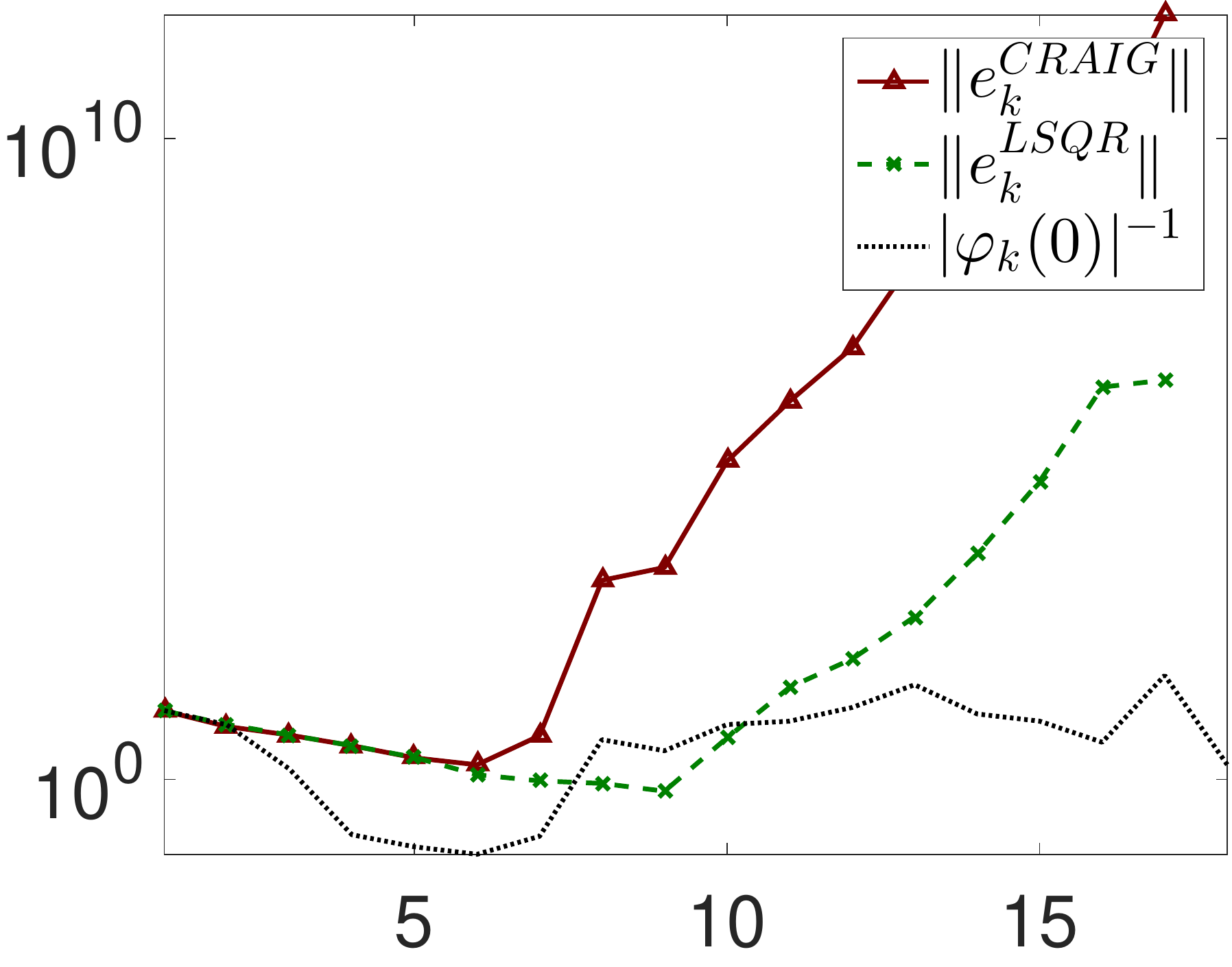}
                \caption{\texttt{shaw(400)}, $\delta_\text{noise} = 10^{-3}$, violet noise}
        \end{subfigure}
        \begin{subfigure}[b]{0.32\textwidth}
        \captionsetup{justification=centering}
                \includegraphics[width=\textwidth]{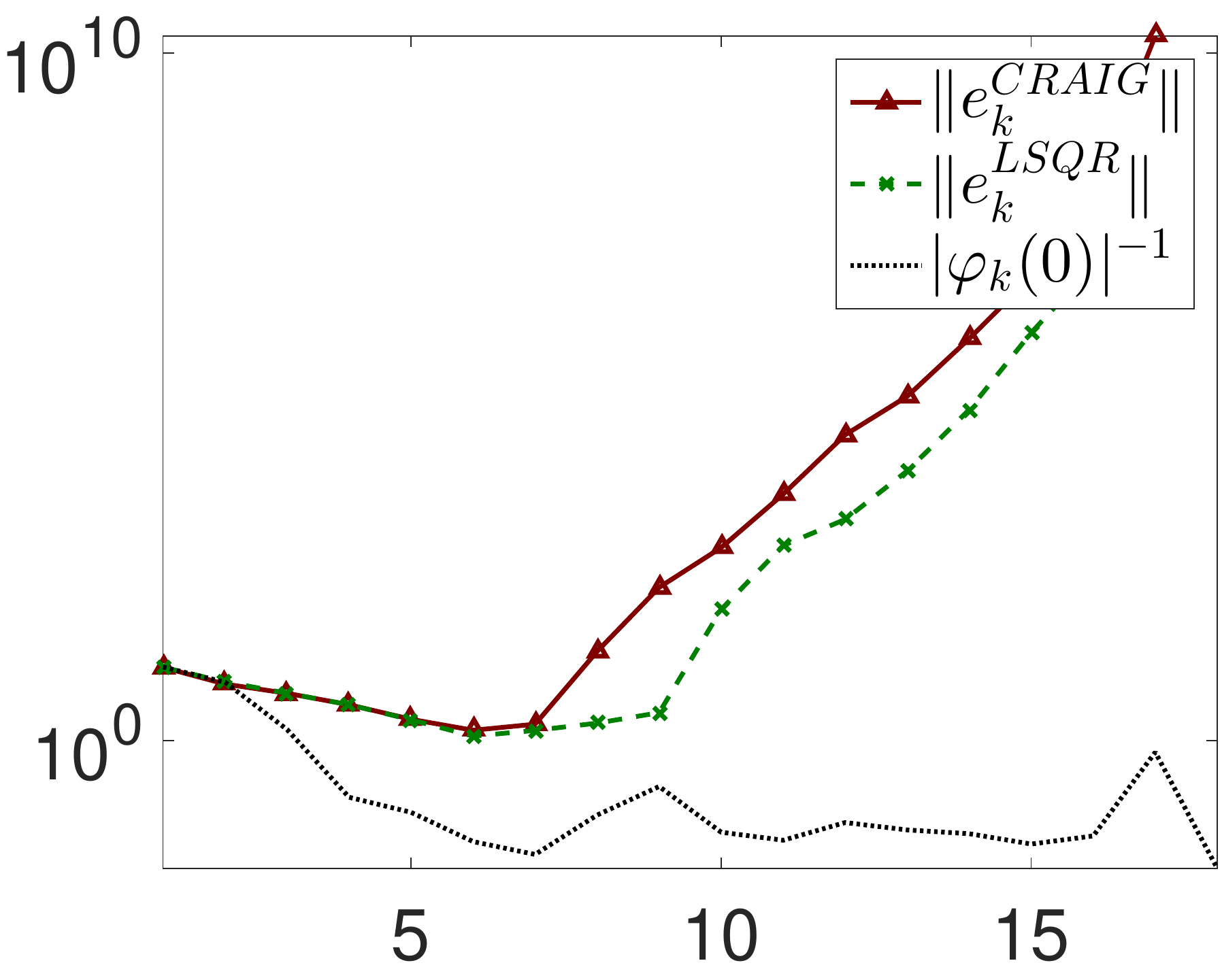}
                \caption{\texttt{shaw(400)}, $\delta_\text{noise} = 10^{-3}$, red noise}
        \end{subfigure}

        \vspace*{.3cm}

        \begin{subfigure}[b]{0.32\textwidth}
        \captionsetup{justification=centering}
                \includegraphics[width=\textwidth]{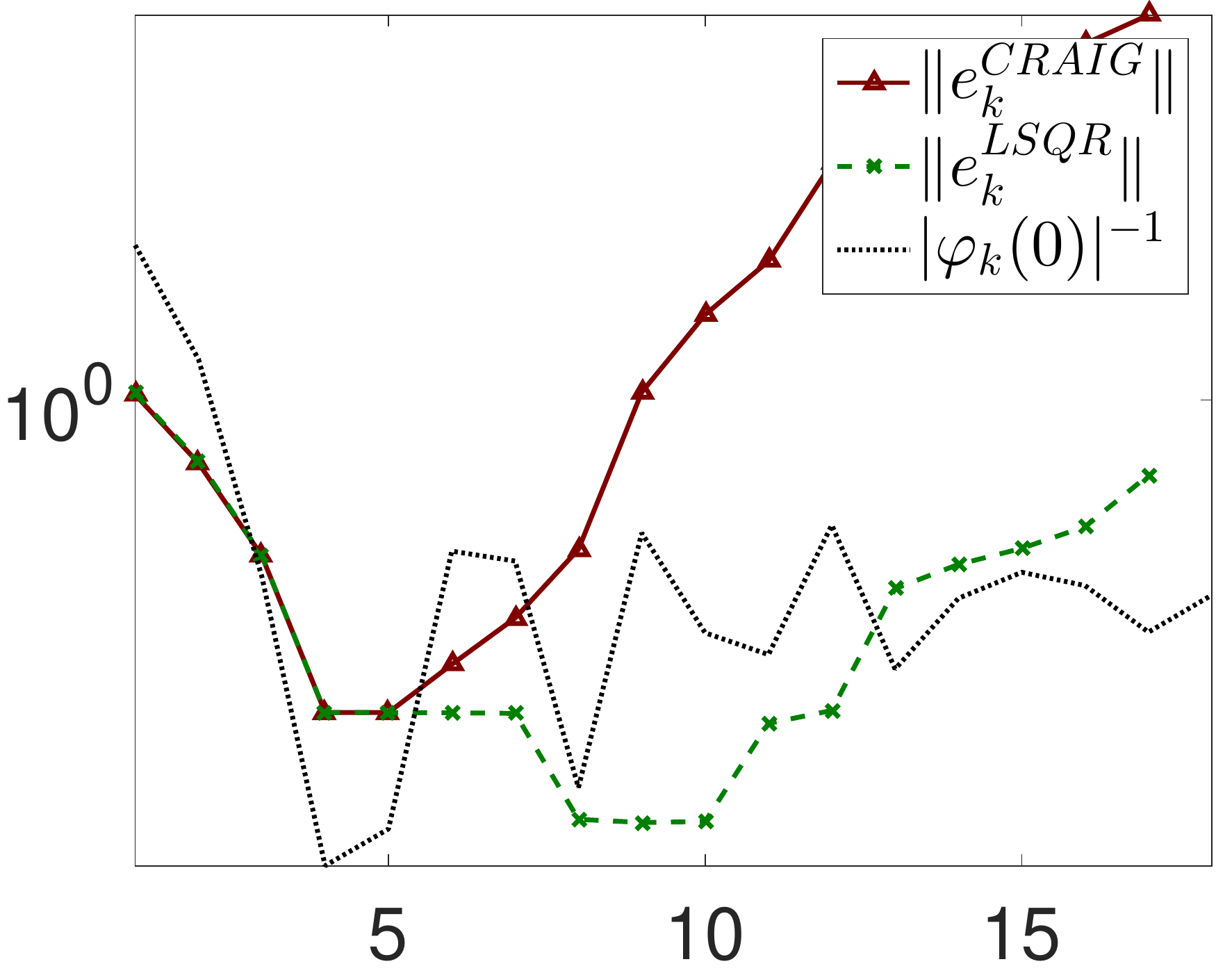}
                \caption{\texttt{phillips(400)}, $\delta_\text{noise} = 10^{-3}$, white noise}
        \end{subfigure}
        \begin{subfigure}[b]{0.32\textwidth}
        \captionsetup{justification=centering}
                \includegraphics[width=\textwidth]{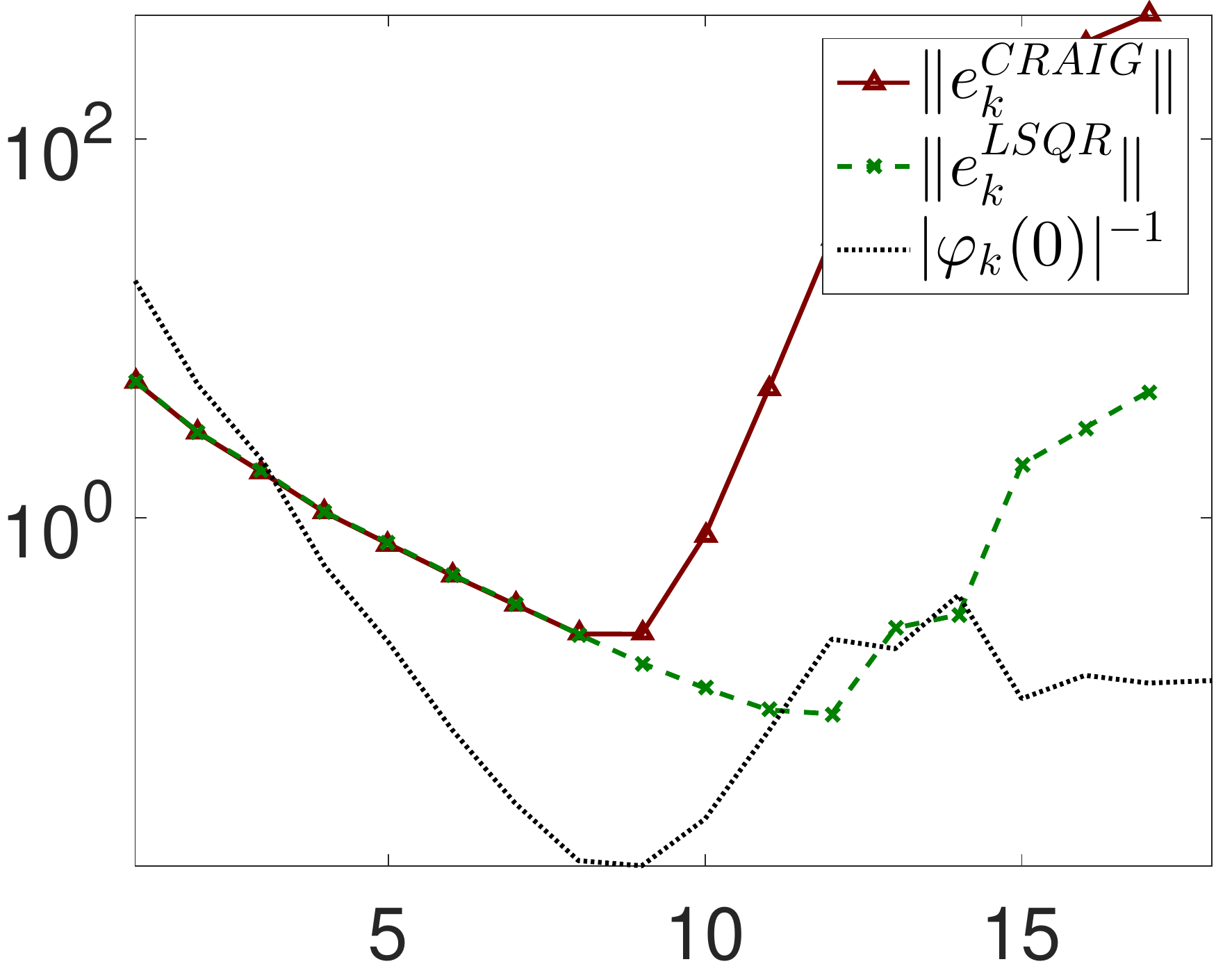}
                \caption{\texttt{gravity(400)}, $\delta_\text{noise} = 10^{-4}$, Poisson}
        \end{subfigure}
        \begin{subfigure}[b]{0.32\textwidth}
        \captionsetup{justification=centering}
                \includegraphics[width=\textwidth]{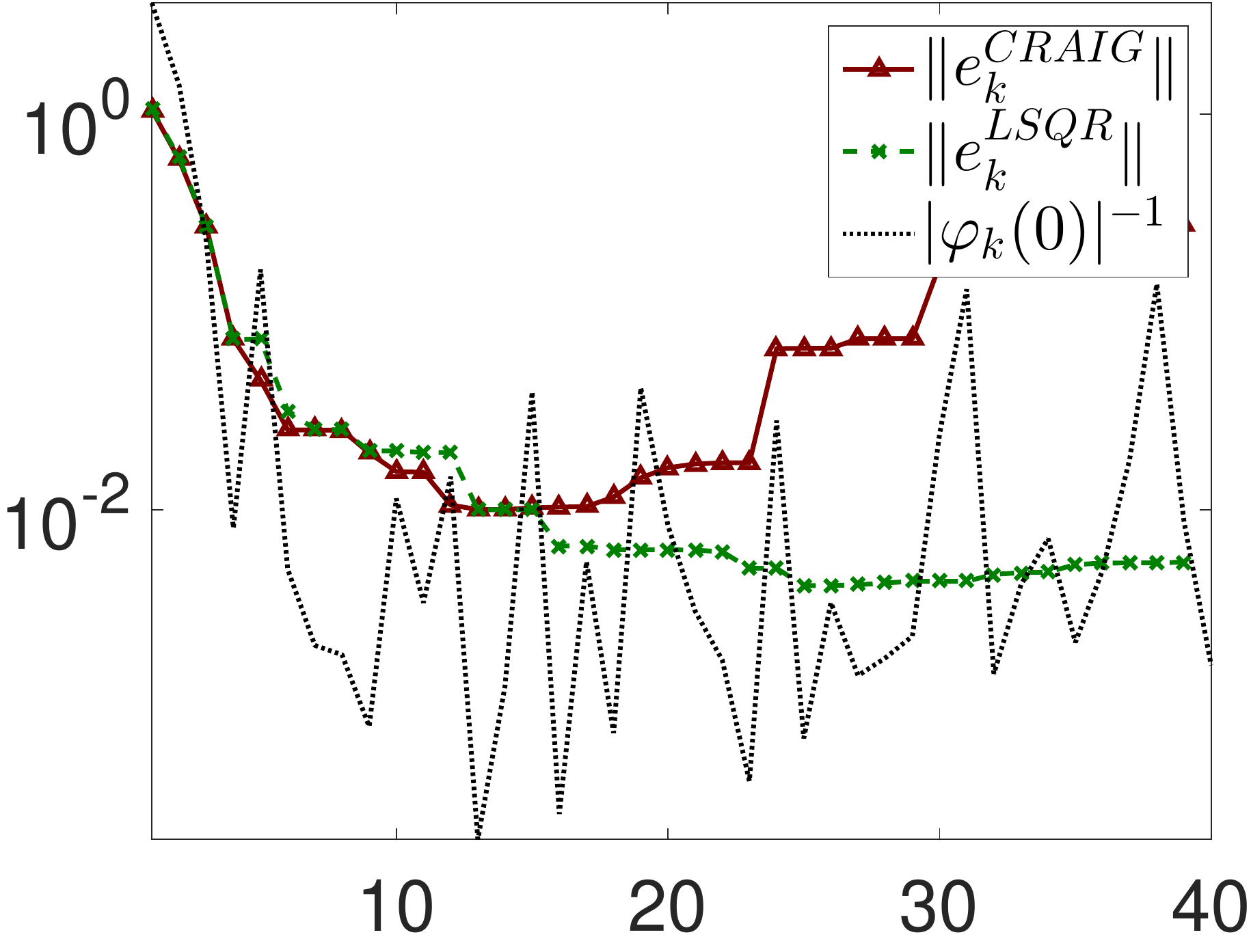}
                \caption{\texttt{phillips(400)},\\ $\delta_\text{noise} = 10^{-5}$, white noise}
                \end{subfigure}
               \caption{The size of the error of LSQR and CRAIG in comparison with the inverse of the amplification factor for various test problems with various noise characteristics. The semiconvergence curves exhibit similar behavior until the noise revealing iteration. In Figure (f) without reorthogonalization.}
\label{fig:error_craig_lsqr}
\end{figure}

\subsection{LSMR residuals}\label{ssec:LSMRres}
Before we investigate the residual of LSMR with respect to the basis $S_k$, we should understand how it is related to the residual of LSQR. It follows from Table~\ref{tab:methods_ne} that the relation between $A^Tr_k^\text{LSMR}$ and  $A^Tr_k^\text{LSQR}$ is analogous to the relation between $r_k^\text{CRAIG}$ and  $r_k^\text{LSQR}$. Using Proposition \ref{th:1} and \ref{th:2}, with $\varphi_k$ substituted by $\psi_k$ and $s_k$ substituted by $w_k$, we obtain
\begin{equation}
A^Tr_k^\text{LSQR} = \psi_k(0)^{-1}w_{k+1},
\end{equation}
and 
\begin{equation}
A^Tr_k^\text{LSMR} = \frac{1}{\sum_{l=0}^k\psi_l(0)^{2}}\sum_{l=0}^k\psi_l(0)w_{l+1}.
\end{equation}
Since
\begin{equation}
A^Tr_k^\text{LSMR} = W_{k+1}L_{k+1}^Tp_k^\text{LSMR},
\end{equation}
we obtain that
\begin{align}
\\
L_{k+1}^Tp_k^\text{LSMR} &= \frac{1}{\sum_{l=0}^k\psi_l(0)^{2}}\left[
\begin{array}{c}
\psi_0(0)\\
\psi_1(0)\\
\vdots\\
\psi_k(0)
\end{array}
\right].\label{eq:Lkpk}
\end{align}
This equality however does not provide the desired relationship between the residuals $r_k^\text{LSMR}$ themselves and the left bidiagonalization vectors $s_1,\ldots,s_{k+1}$. This is given in the following proposition.

\begin{prop}\label{th:3}
Consider the first $k$ steps of the Golub-Kahan iterative bidiagonalization. Let $r_k^\text{LSMR} = b - Ax_k^\text{LSMR}$, where $x_k^\text{LSMR}$ is  the approximation defined in \eqref{eq:projected} and \eqref{eq:LSMR_projected}. Then
\begin{equation}
r_k^\text{LSMR} = \frac{1}{\sum_{l=0}^k\psi_l(0)^2}
\sum_{l = 0}^k \left(\varphi_l(0)\sum_{j=l}^k\alpha_{j+1}^{-1}\varphi_j(0)^{-1}\psi_j(0)\right) s_{l+1}.
\label{eq:residual_lsmr_sum}
\end{equation}
\end{prop}
\begin{proof}
From \eqref{eq:Lkpk} it follows that
\begin{equation}
p_k^{\text{LSMR}}=\frac{1}{\sum_{l=0}^k\psi_l(0)^2}\,L_{k+1}^{-T}\left[\begin{array}{c}\psi_0(0)\\\psi_1(0)\\\vdots\\\psi_k(0)\end{array}\right],
\end{equation}
where $L_{k+1}^{-T}$ is an upper triangular matrix with entries
\begin{equation}
  e_i^TL_{k+1}^{-T}e_j
  = \left\{\begin{array}{cc}
    \frac{\displaystyle 1}{\displaystyle\alpha_j} & \text{(if $i=j$)} \\[2mm]
    (-1)^{i-j}\,\frac{\displaystyle\beta_{i+1}\cdots\beta_j}{\displaystyle\alpha_i\cdots\alpha_j} & \text{(if $i<j$)}
  \end{array}\right\}
  = \frac{\varphi_{i-1}(0)}{\alpha_j\,\varphi_{j-1}(0)}\,.
\end{equation}
Thus
\begin{equation}
  \begin{split}
  p_k^{\text{LSMR}}
  % &= \frac{1}{\sum_{l=0}^k\psi_l(0)^2}\,
  % \Phi_k(0)\,Z\,\Phi_k(0)^{-1}
  % \left[\begin{array}{c}
  %   \alpha_1^{-1}\psi_0(0)\\
  %   \alpha_2^{-1}\psi_1(0)\\
  %   \vdots\\
  %   \alpha_{k+1}^{-1}\psi_k(0)
  % \end{array}\right] \\
  = \frac{1}{\sum_{l=0}^k\psi_l(0)^2}\,
  \mathrm{triu}\left(
  \left[\begin{array}{c}
    \varphi_0(0)\\
    \varphi_1(0)\\
    \vdots\\
    \varphi_k(0)
  \end{array}\right]
  \left[\begin{array}{c}
    \varphi_0(0)^{-1}\\
    \varphi_1(0)^{-1}\\
    \vdots\\
    \varphi_k(0)^{-1}
  \end{array}\right]^T
  \right)
  \left[\begin{array}{c}
    \alpha_1^{-1}\psi_0(0)\\
    \alpha_2^{-1}\psi_1(0)\\
    \vdots\\
    \alpha_{k+1}^{-1}\psi_k(0)
  \end{array}\right], \\
  \end{split}
\end{equation}
where $\mathrm{triu}(\cdot)$ extracts the upper triangular part of the matrix. Multiplying out, we obtain
\begin{equation}
  p_k^{\text{LSMR}}  = \frac{1}{\sum_{l=0}^k\psi_l(0)^2}
  \left[\begin{array}{c}
    \varphi_0(0)\sum_{l=0}^k\alpha_{l+1}^{-1}\varphi_l(0)^{-1}\psi_l(0)\\
    \varphi_1(0)\sum_{l=1}^k\alpha_{l+1}^{-1}\varphi_l(0)^{-1}\psi_l(0)\\
    \vdots\\
    \varphi_{k-1}(0)\sum_{l=k-1}^k\alpha_{l+1}^{-1}\varphi_l(0)^{-1}\psi_l(0)\\
    \varphi_k(0)\qquad\;\alpha_{k+1}^{-1}\varphi_k(0)^{-1}\psi_k(0)
  \end{array}\right].
  \end{equation}\label{eq:pkcoef}
\qed
\end{proof}

Here the sizes of coefficients in $p_k^{\text{LSMR}}$ need careful discussion. From \eqref{eq:ampl_factor} and \eqref{eq:abspsi} it
follows that the absolute terms of the Lanczos polynomials $\varphi_l(0)$ and $\psi_l(0)$ have the same sign. Thus we have
\begin{equation}
\alpha_{l+1}^{-1}\varphi_l(0)^{-1}\psi_l(0) >0, \quad \forall  l=0, 1, \dots
\end{equation}
and therefore the sum
\begin{equation}
\sum_{l=j}^k\alpha_{l+1}^{-1}\varphi_l(0)^{-1}\psi_l(0) \label{eq:LSQRsum}
\end{equation}
decreases when $j$ increases. 
Furthermore,  it was shown in \cite[sec. 3.2]{Hnetynkova2009regularizing} that for $j<k_\text{rev}$
\begin{equation}
\alpha_l \approx \beta_l.
\end{equation}
Thus \eqref{eq:ampl_factor} yields
\[
\sum_{l=j}^k\alpha_{l+1}^{-1}\varphi_l(0)^{-1}\psi_l(0) \approx \sum_{l=j}^k \psi_l(0).
\]
However, since $|\varphi_j(0)|$ on average increases rapidly with $j$ (see Section~\ref{sec:behavior}), 
the sizes of the entries of $p_k^\text{LSMR}$ in \eqref{eq:pkcoef} generally increase with $l$ before $k_\text{rev}$.
After $j$ reaches the noise revealing iteration $k_\text{rev}$, $|\varphi_j(0)|$ decreases at least for one but typically for more subsequent iterations;
see Section~\ref{sec:behavior}. Multiplication by the decreasing \eqref{eq:LSQRsum} causes that the size of the entries in \eqref{eq:pkcoef} can be expected to decrease after $k_\text{rev}$.

From the previous we conclude that the behavior of the entries of $p_k^\text{LSMR}$ resembles the behavior of $\varphi_l(0)$, i.e., the size of a particular entry is proportional to the amount of propagated noise in the corresponding bidiagonalization vector, similarly as in the LSQR method. Figure~\ref{fig:pkLSQR_MR} compares the entries of $p_k^\text{LSMR}$ with appropriately re-scaled amplification factor $\varphi_k(0)$ on the problem \texttt{shaw} with white noise. We see that the difference is negligible an therefore the residuals for LSQR and LSMR resemble. In early iterations, the resemblance of the residuals indicates resemblance of the solutions since the remaining perturbation only contains low frequencies, which are not amplified by $A^{\dagger}$.
\begin{figure}
\centering
\includegraphics[width = .32\textwidth]{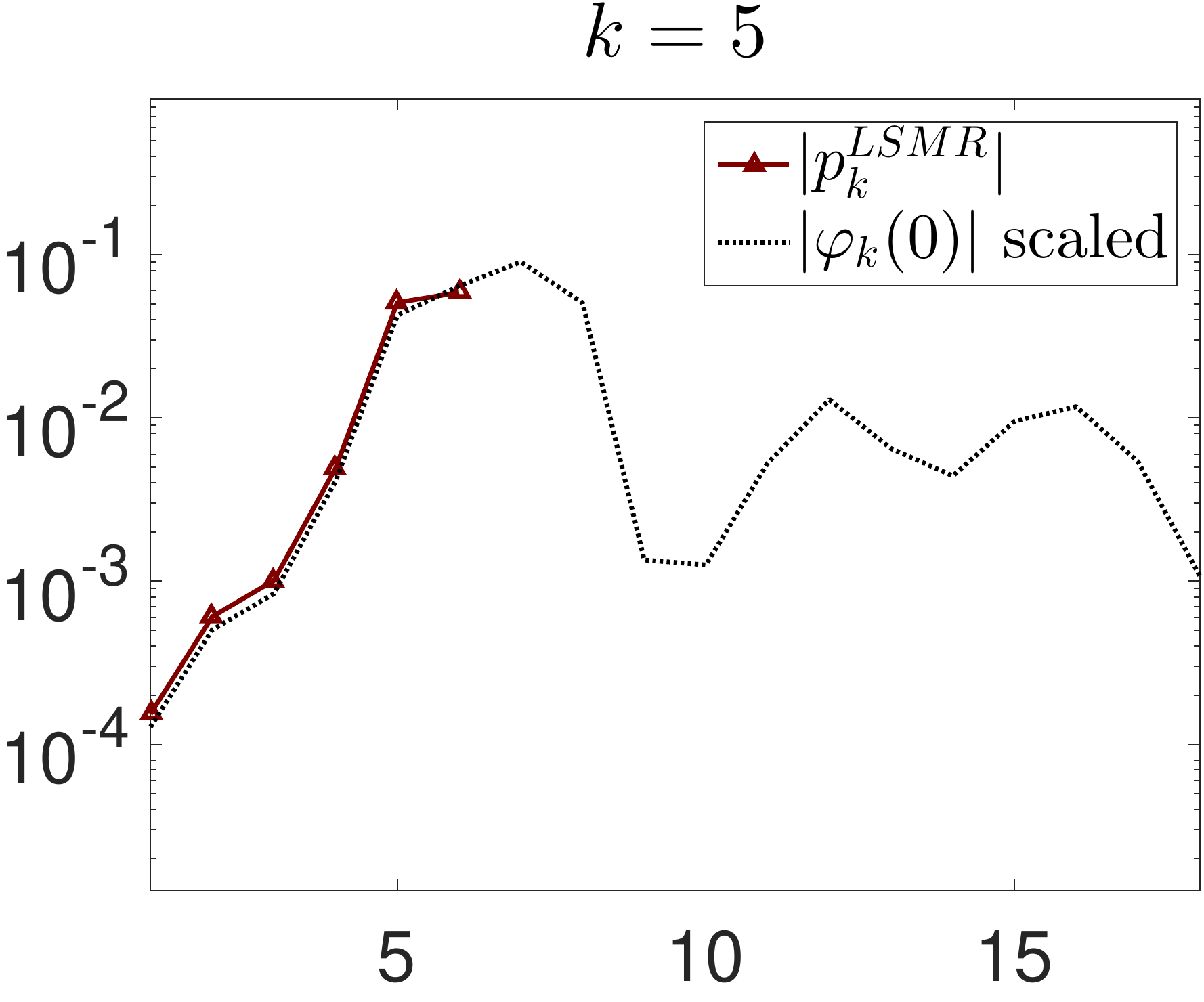}
\includegraphics[width = .32\textwidth]{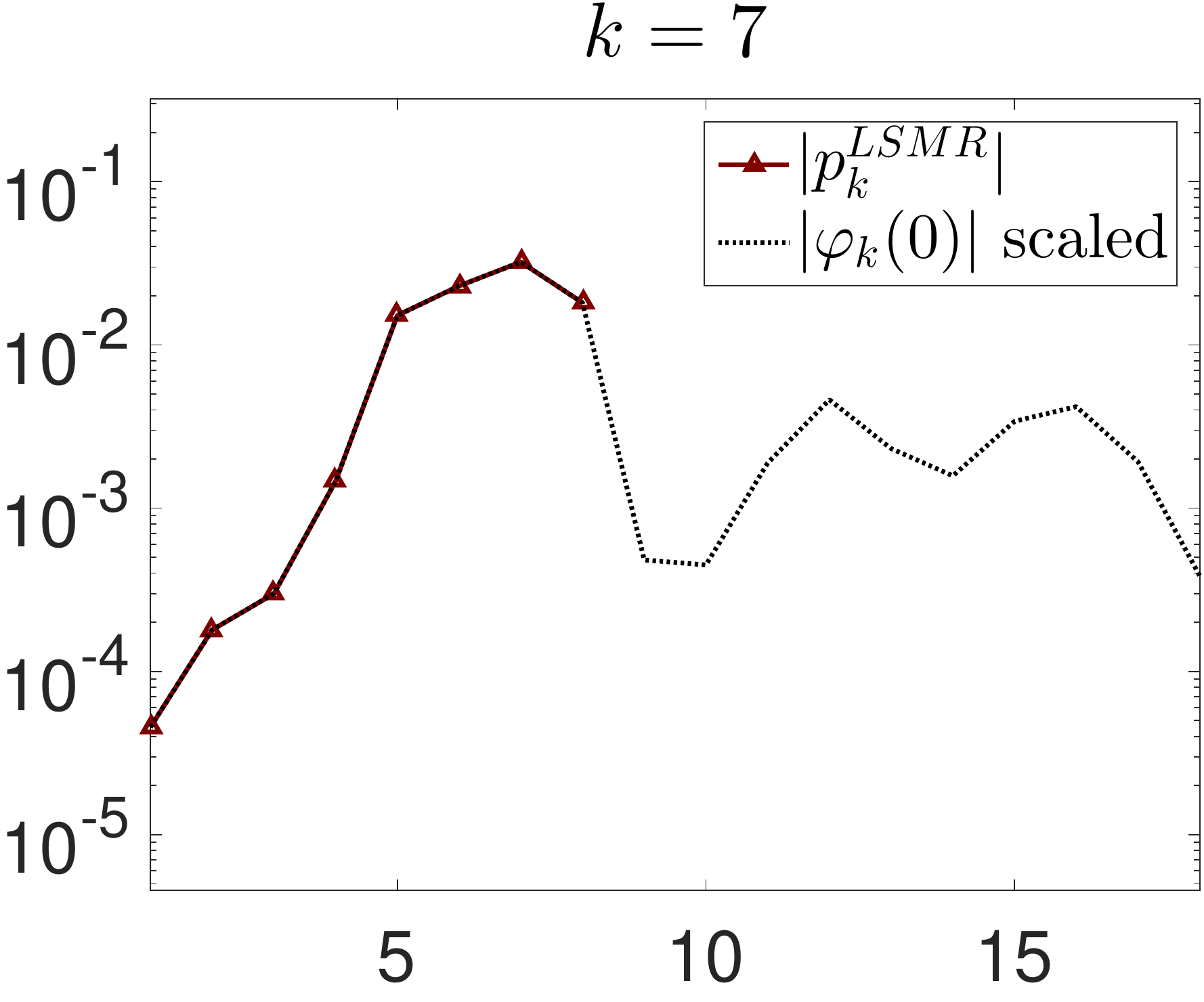}
\includegraphics[width = .32\textwidth]{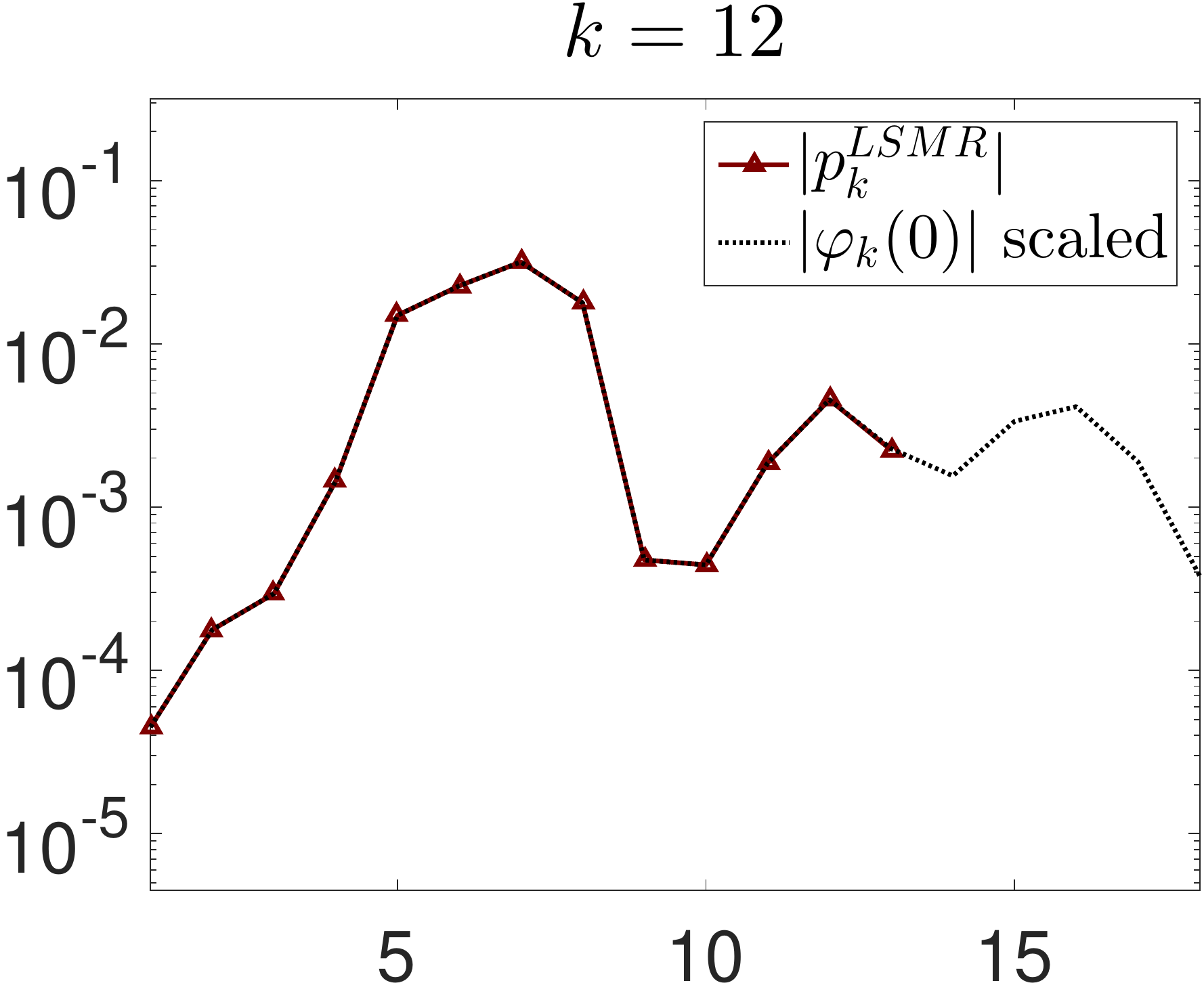}
\caption{The componets of $p_k^\text{LSQR}$ vs. the size of the amplification factor $\varphi_k(0)$ (after scaling) for several values of $k$ for the problem \texttt{shaw} with white noise, $\delta_\text{noise}=10^{-3}$. The differences are negligible.}\label{fig:pkLSQR_MR}
\end{figure}

Note also that since $\psi_k(0)$ grows rapidly on average, see Figure \ref{fig:psi0} in Section~\ref{sec:behavior}, we may expect
\begin{equation}
\frac{\psi_k(0)^{2}}{\sum_{l=0}^{k}\psi_l(0)^{2}}\approx 1.
\end{equation}
Therefore $A^Tr_k^\text{LSMR}$ resembles $A^Tr_k^\text{LSQR}$ giving another explanation why
LSMR and LSQR behave similarly for inverse problems with a smoothing operator $A$, see Figure~\ref{fig:error_lsqr_lsmr} for a comparison on several test problems. 
\begin{figure}
        \centering
        \begin{subfigure}[b]{0.32\textwidth}
        \captionsetup{justification=centering}
                \includegraphics[width=\textwidth]{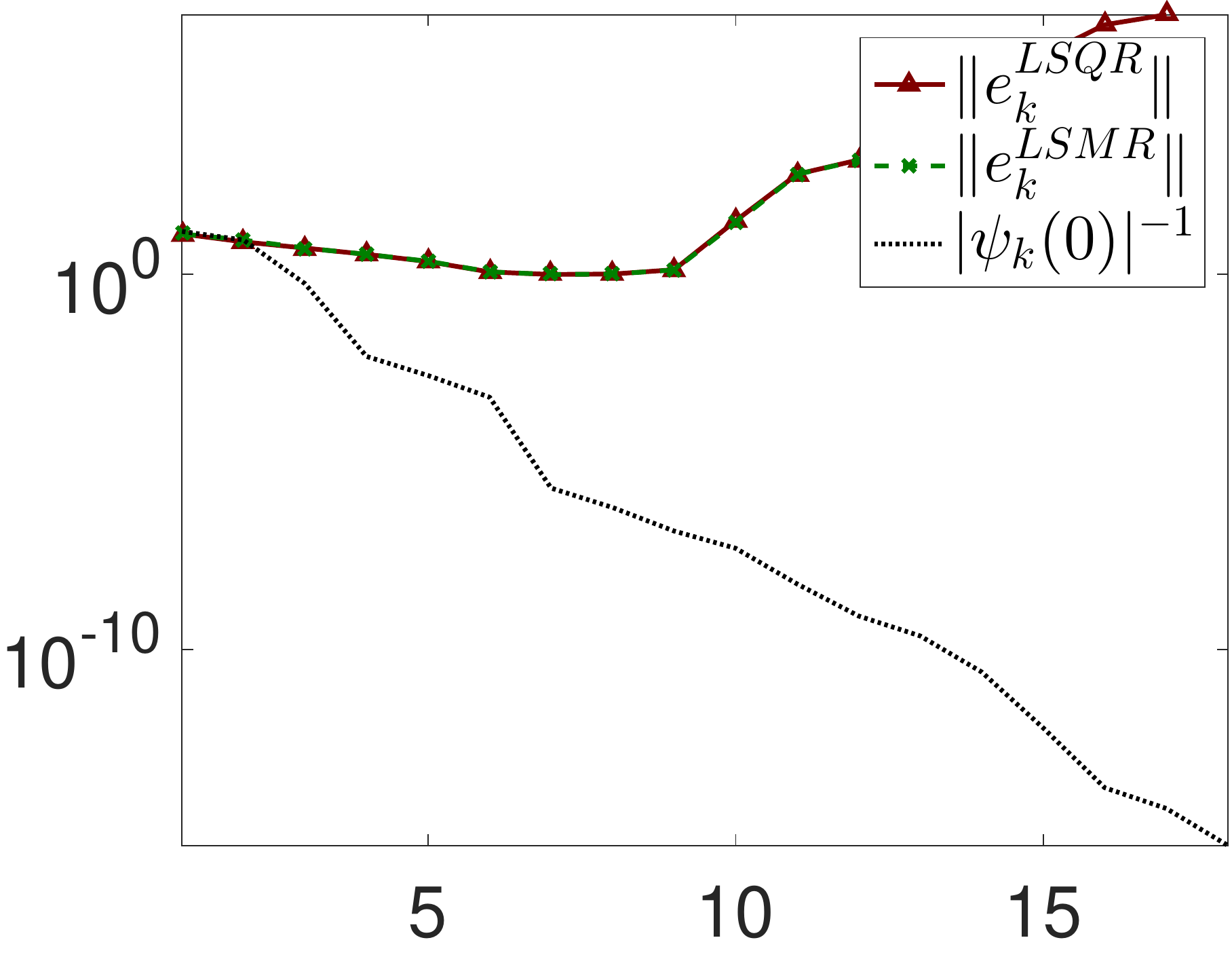}
                \caption{\texttt{shaw(400)},\\ $\delta_\text{noise} = 10^{-3}$, white noise}
        \end{subfigure}
        \begin{subfigure}[b]{0.32\textwidth}
        \captionsetup{justification=centering}
                \includegraphics[width=\textwidth]{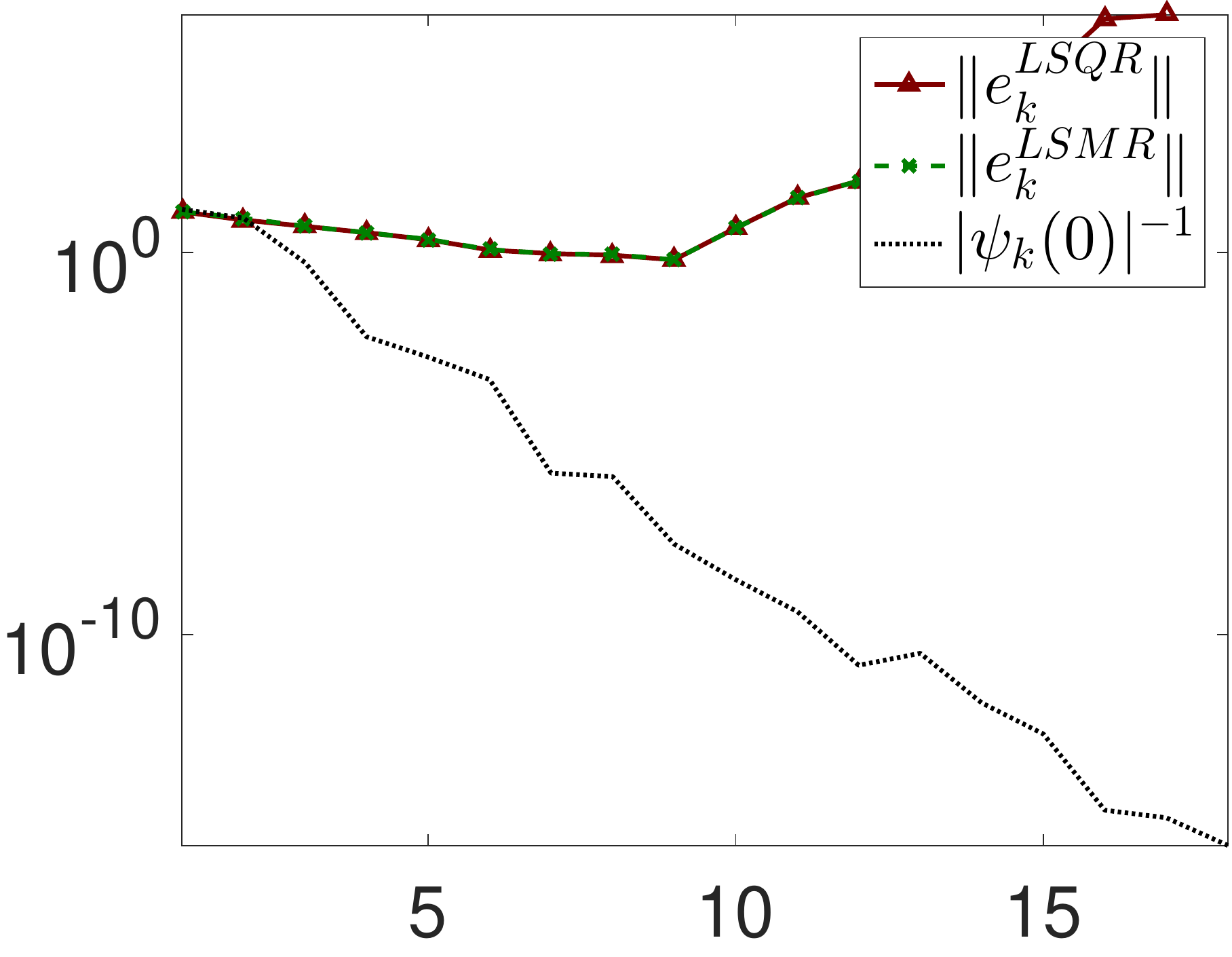}
                \caption{\texttt{shaw(400)}, $\delta_\text{noise} = 10^{-3}$, violet noise}
        \end{subfigure}
        \begin{subfigure}[b]{0.32\textwidth}
        \captionsetup{justification=centering}
                \includegraphics[width=\textwidth]{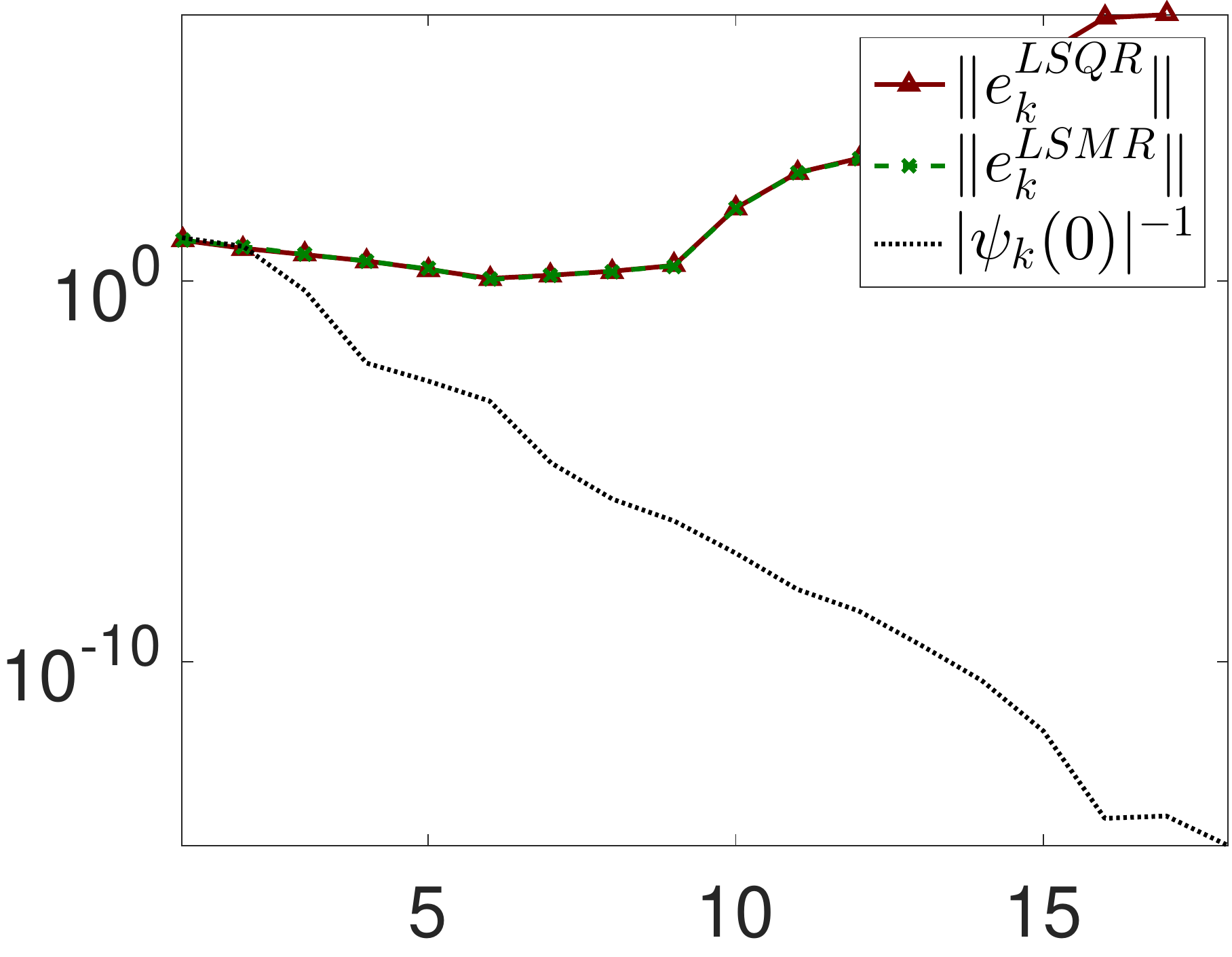}
                \caption{\texttt{shaw(400)}, $\delta_\text{noise} = 10^{-3}$, red noise}
        \end{subfigure}

        \vspace*{.3cm}

        \begin{subfigure}[b]{0.32\textwidth}
        \captionsetup{justification=centering}
                \includegraphics[width=\textwidth]{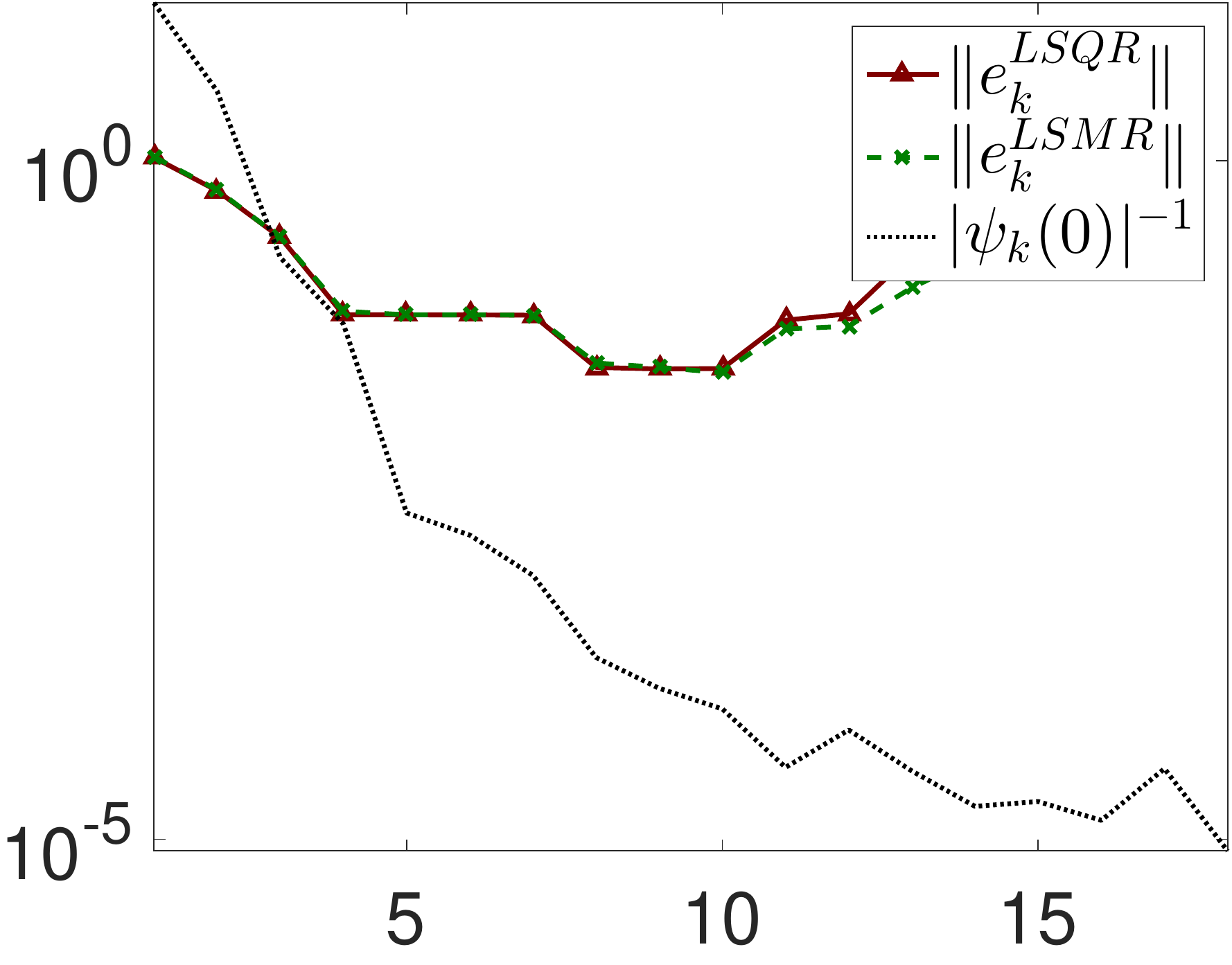}
                \caption{\texttt{phillips(400)}, $\delta_\text{noise} = 10^{-3}$, white noise}
        \end{subfigure}
        \begin{subfigure}[b]{0.32\textwidth}
        \captionsetup{justification=centering}
                \includegraphics[width=\textwidth]{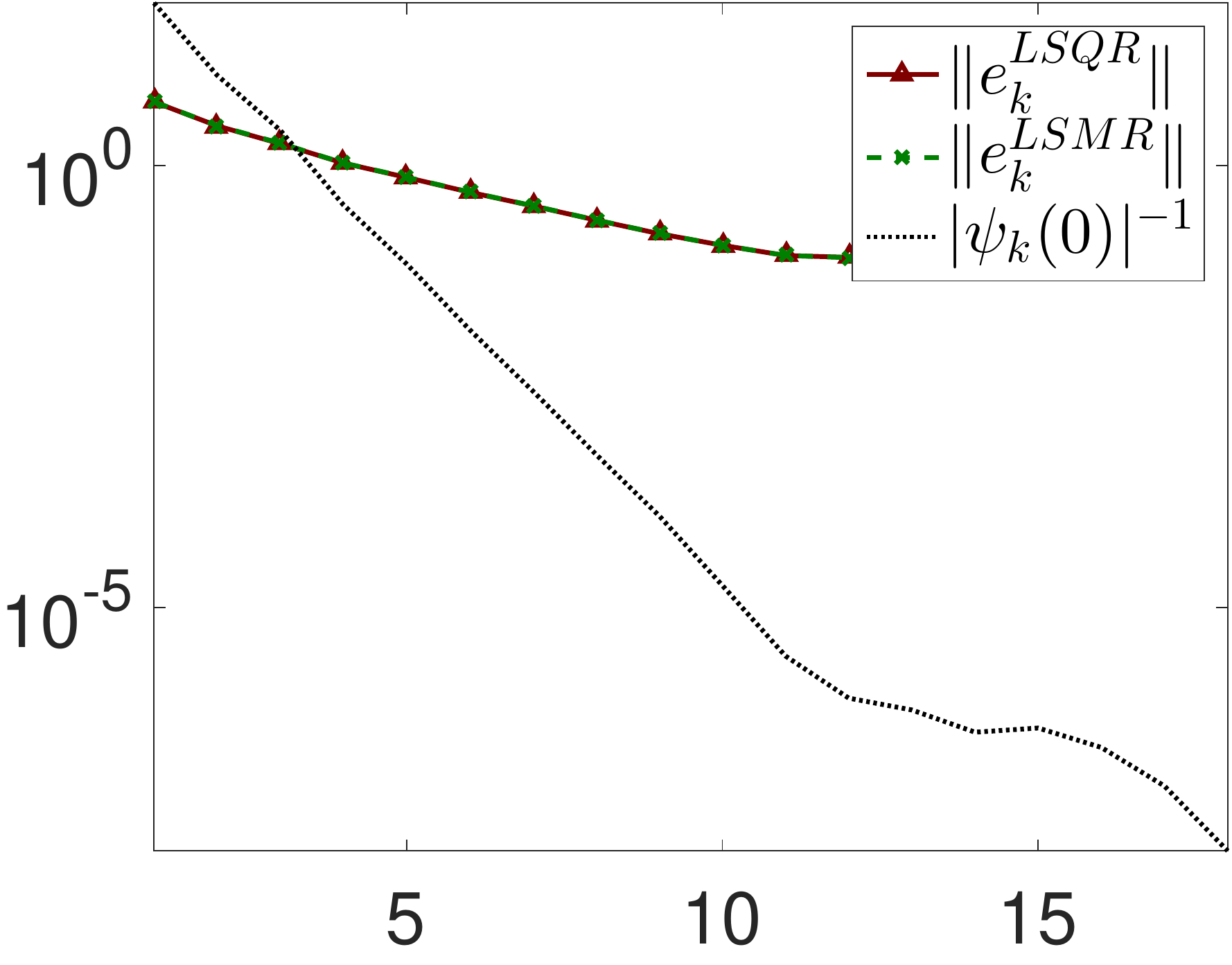}
                \caption{\texttt{gravity(400)}, $\delta_\text{noise} = 10^{-4}$, Poisson}
        \end{subfigure}
        \begin{subfigure}[b]{0.32\textwidth}
        \captionsetup{justification=centering}
                \includegraphics[width=\textwidth]{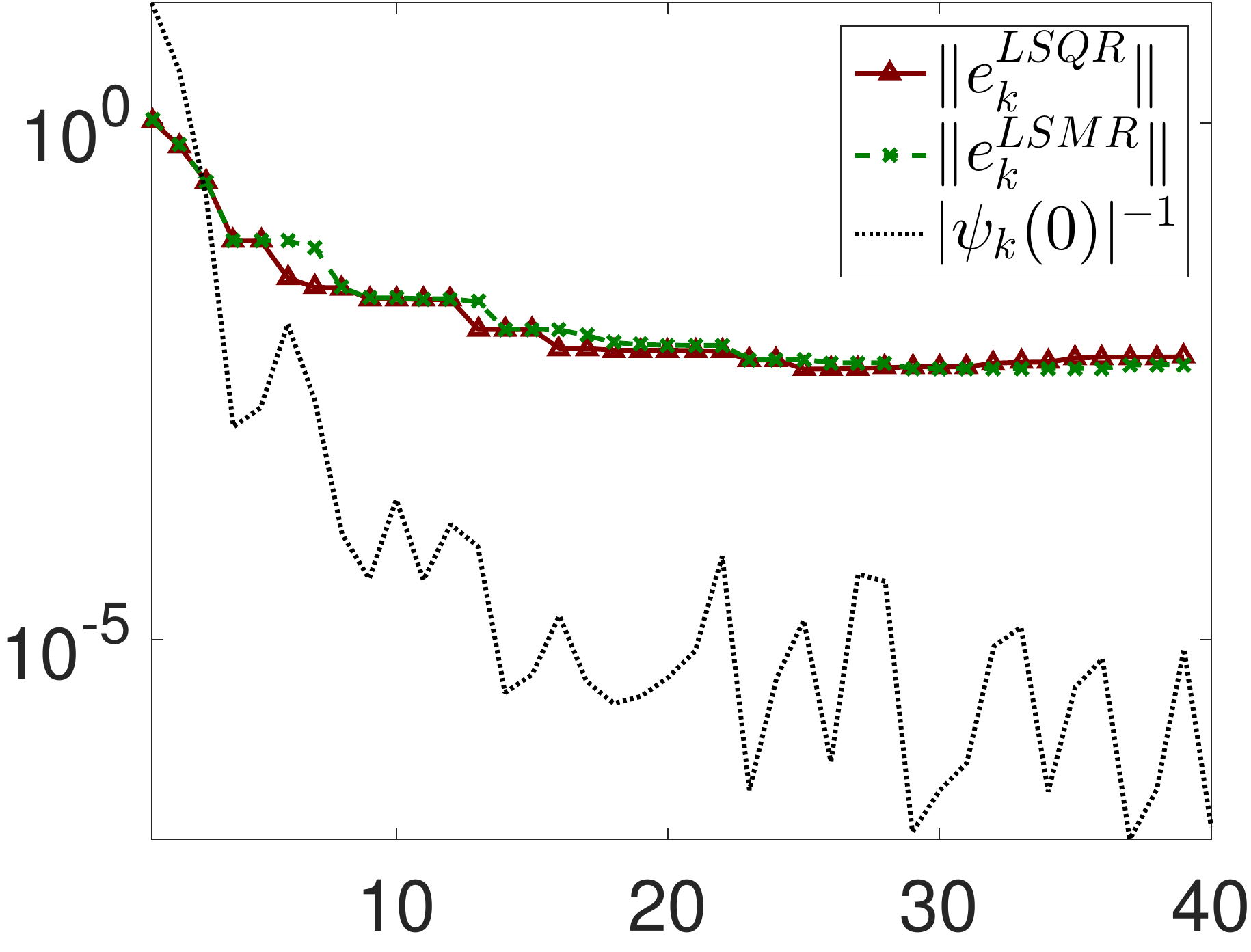}
                \caption{\texttt{phillips(400)},\\ $\delta_\text{noise} = 10^{-5}$, white noise}
                \end{subfigure}
               \caption{The size of the error of LSMR and LSQR in comparison with the inverse of the size of $\psi_k(0)$ for various test problems with various noise characteristics. Since $|\psi_k(0)|$ often grow on average till very late iterations, the semiconvergence curves exhibit similar behavior. In Figure~(f) without reorthogonalization.}
\label{fig:error_lsqr_lsmr}
\end{figure}

\medskip

Figure \ref{fig:remaining_perturbations} illustrates the match between the noise vector and residual of CRAIG, LSQR and LSMR method. We see that
while CRAIG residual resembles noise only in the noise revealing iteration, LSQR and LSMR are less sensitive to the particular number of iterations $k$ as the residuals are combinations of bidiagonalization vectors with appropriate coefficients. Moreover, the best match in LSQR and LSMR method overcomes the best match in CRAIG. This is caused by the fact that the remaining low-frequency part is efficiently suppressed by the linear combination.

\begin{figure}
\centering
\begin{subfigure}[b]{.8\textwidth}
\includegraphics[width = .95\textwidth]{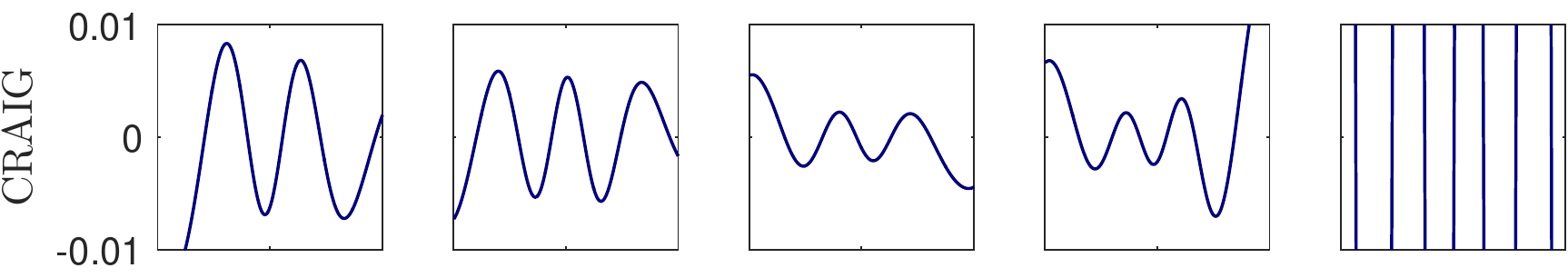}
\includegraphics[width = .95\textwidth]{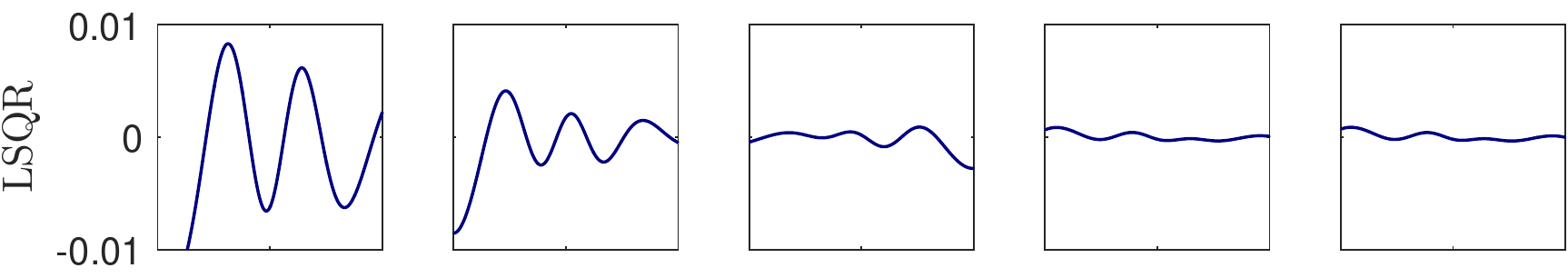}
\includegraphics[width = .95\textwidth]{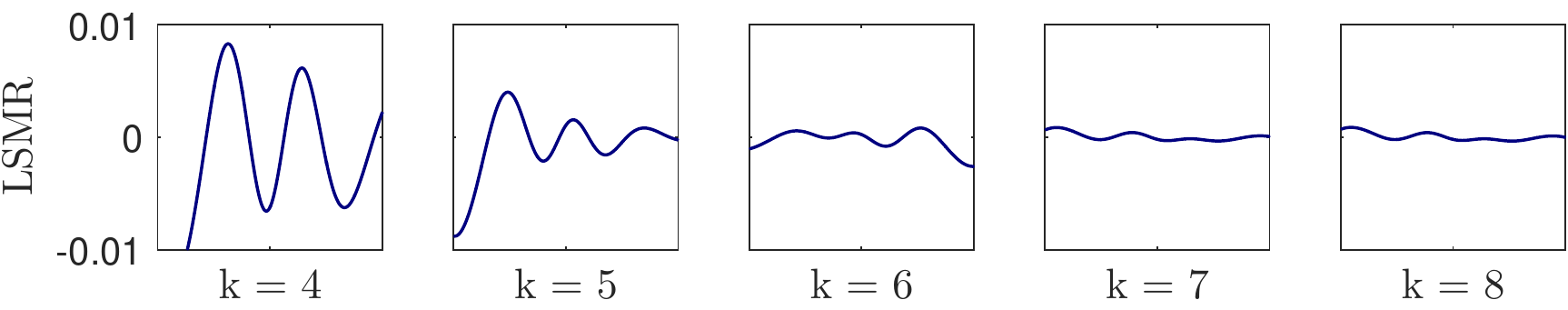}
\caption{$\eta - r_k$, white noise}
\end{subfigure}

\vspace{.3cm}

\begin{subfigure}[b]{.8\textwidth}
\includegraphics[width = .95\textwidth]{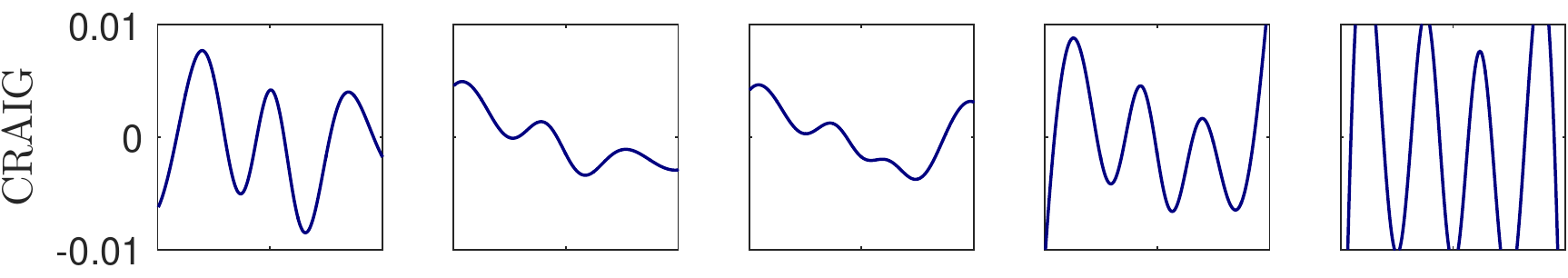}
\includegraphics[width = .95\textwidth]{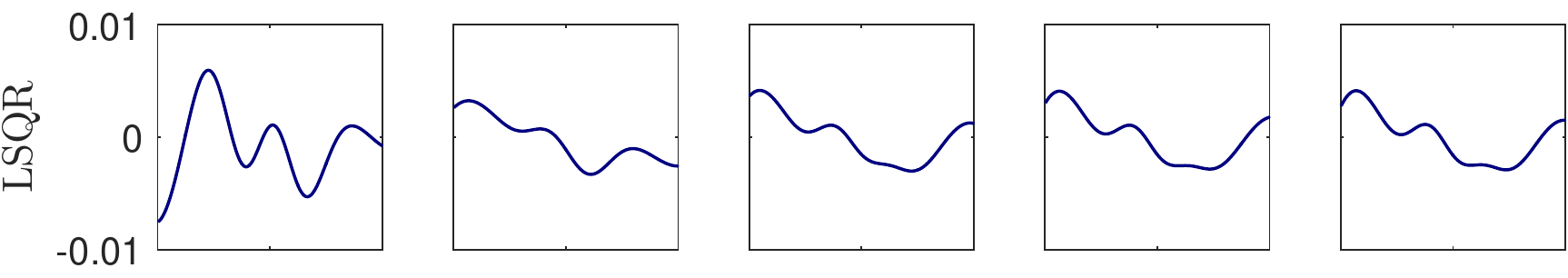}
\includegraphics[width = .95\textwidth]{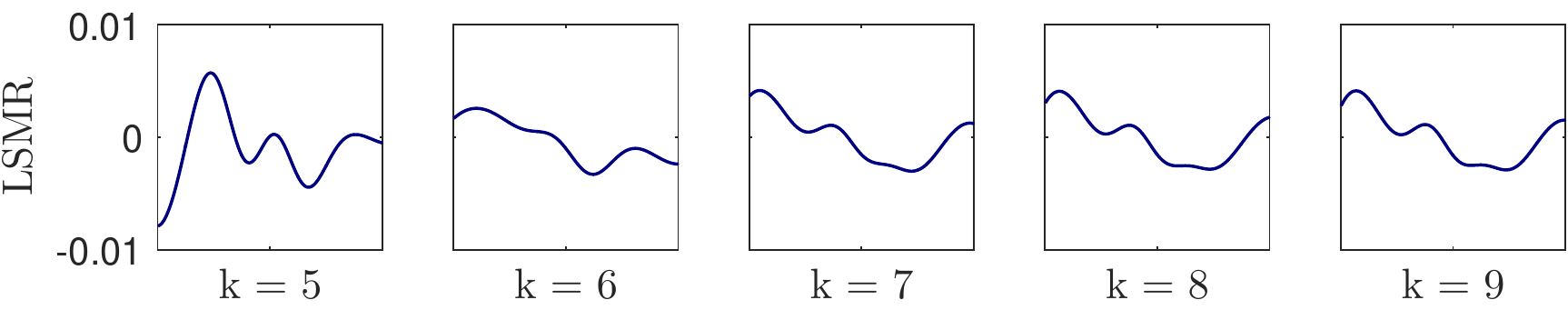}
\caption{$\eta - r_k$, red noise}
\end{subfigure}
\caption{Difference between the noise vector and the residual of considered iterative methods for the problem \texttt{shaw} with white and red noise noise, $\delta_\text{noise} = 10^{-3}$. 
Residuals of LSQR and LSMR have similar approximation properties with respect to the noise vector.}\label{fig:remaining_perturbations}
\end{figure}

\section{Numerical experiments for 2D problems}\label{sec:2Dproblems}

In this section we discuss validity of the conclusions made above for larger 2D inverse problems,
where the smoothing property of $A$ (revealing itself in the decay of singular values) is typically less significant. 
Consequently, noise propagation in the bidiagonalization process may be more complicated; see also \cite{HnetynkovaAlgoritmy}.
However, we illustrate that essential aspects of the behavior described in previous sections are still present. 
Note that all experiments in this section are computed {\em without} reorthogonalization. We consider the following 
2D benchmarks:

\begin{description}
\item{\bf Medical tomography problem} --- a simplified 2D model of X-ray medical tomography adopted
from \cite{Hansen2012AIR}, function \texttt{paralleltomo(256,0:179,362)}. The data is represented the $256$-by-$256$ discretization of the Shepp--Logan phantom projected in angles
$\theta = 0\degree,1\degree,\ldots,179\degree$
by $362$ parallel rays, resulting in a linear algebraic problem with $A\in\mathbb{R}^{65160\times 65536}$. We use Poisson-type additive noise $\eta$ generated as follows (see \cite[chap. 2.6]{Buzug} and \cite{AndersenJorgensen}) to simulate physically realistic noise:
\begin{verbatim}
A = paralleltomo(N,theta)/N;  % forward model
t = exp(-A*x);                % transmission probabilities
c = poissrnd(t*N0);           % photon counts
eta = -log(c/N0);             % noisy measurements
\end{verbatim}
where $N_0=10^5$ denotes the mean number of photons, resulting in the noise level $\delta_\text{noise} \approx 0.028$. We refer to this test problem as \texttt{paralleltomo}.

\item{\bf Seismic tomography problem} --- a simplified 2D model of seismic tomography
adopted from \cite{Hansen2012AIR}, function \texttt{seismictomo(100,100,200)}.
The data is represented by a $100$-by-$100$ discretization
of a vertical domain intersecting two tectonic plates with $100$ sources located on its right boundary and $200$
receivers (seismographs), resulting in a linear algebraic problem with $A\in\mathbb{R}^{20000\times 10000}$. The right-hand side is polluted with additive white noise with $\delta_\text{noise} = 0.01$. We refer to this test problem as \texttt{seismictomo}.

\item{\bf Image deblurring problem} --- an image deblurring problem with spatially
variant blur adopted from \cite{Berisha2013Iterative,Nagy2004Iterative},
data \texttt{VariantGaussianBlur1}. The data is represented by a monochrome microscopic $316$-by-$316$ image of a grain blurred by spatially variant Guassian blur (with $49$ different point-spread functions), resulting in a linear algebraic problem with $A\in\mathbb{R}^{99856\times 99856}$. The right-hand side is polluted with additive white noise with $\delta_\text{noise} = 0.01$. We refer to this test problem as
 \texttt{vargaussianblur}.
\end{description}

Figure~\ref{fig:2D_factors} shows the absolute terms of the Lanczos polynomials $\varphi_k$ and $\psi_k$. We can identify the two phases of the behavior of $\varphi_k(0)$ --  average growth and average decay. However, the transition does not take place 
in one particular (noise revealing) iteration, but rather in a few subsequent steps, which we refer to as the
{\em noise revealing phase} of the bidiagonalization process. 
The size of $\psi_k(0)$ grows on average till late iterations, however, we often observe here that the speed of this growth slows down after the noise 
revealing phase. In conclusion, both curves $|\varphi_k(0)|$ and $|\psi_k(0)|$ can be flatter than for 1D problem considered in previous sections. 
This can be further pronounced for problems with low noise levels.

\begin{figure}
        \begin{subfigure}[b]{0.32\textwidth}
                \includegraphics[width = .95\textwidth]{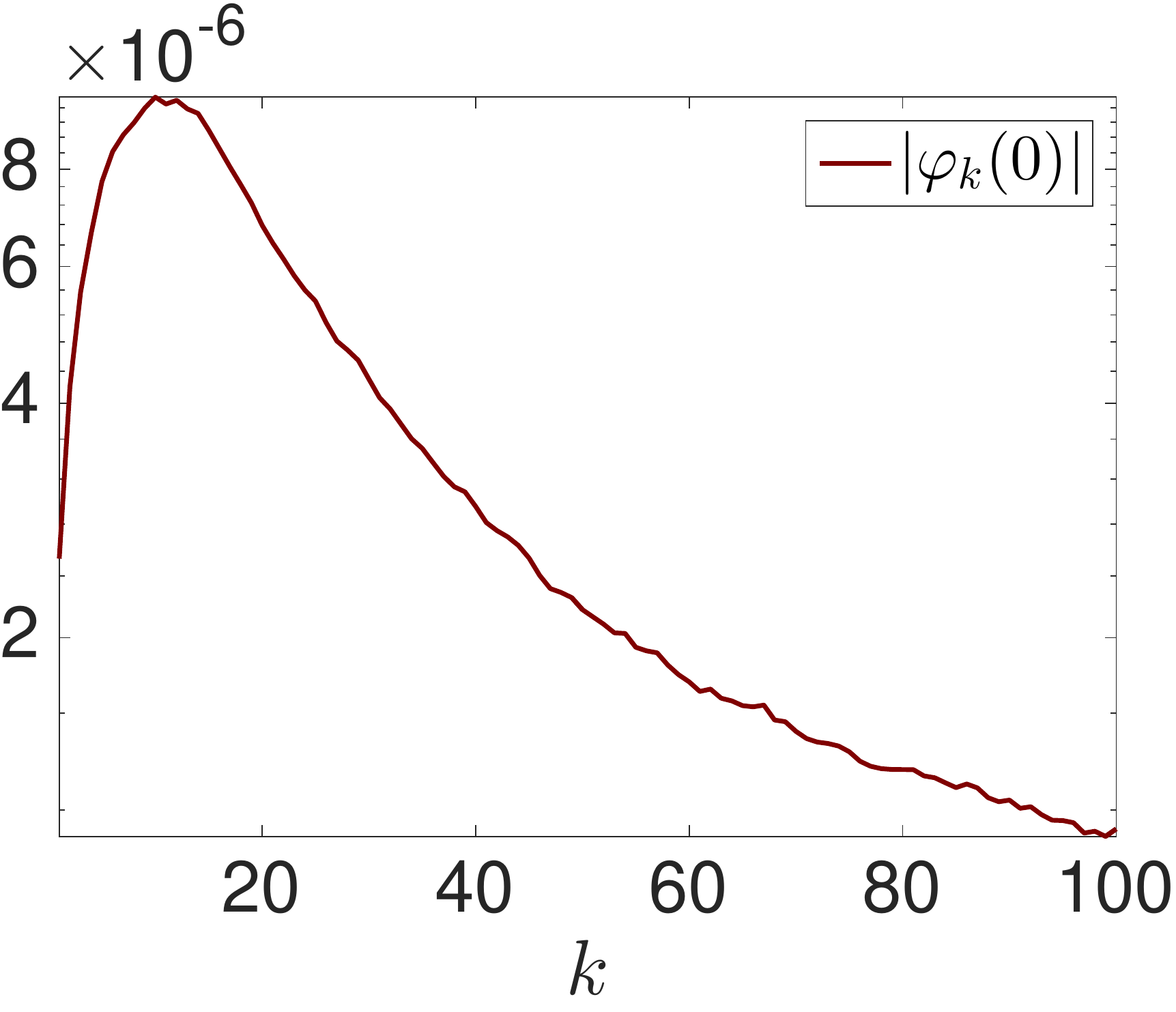}
        \end{subfigure}
        \begin{subfigure}[b]{0.32\textwidth}
                \includegraphics[width = .95\textwidth]{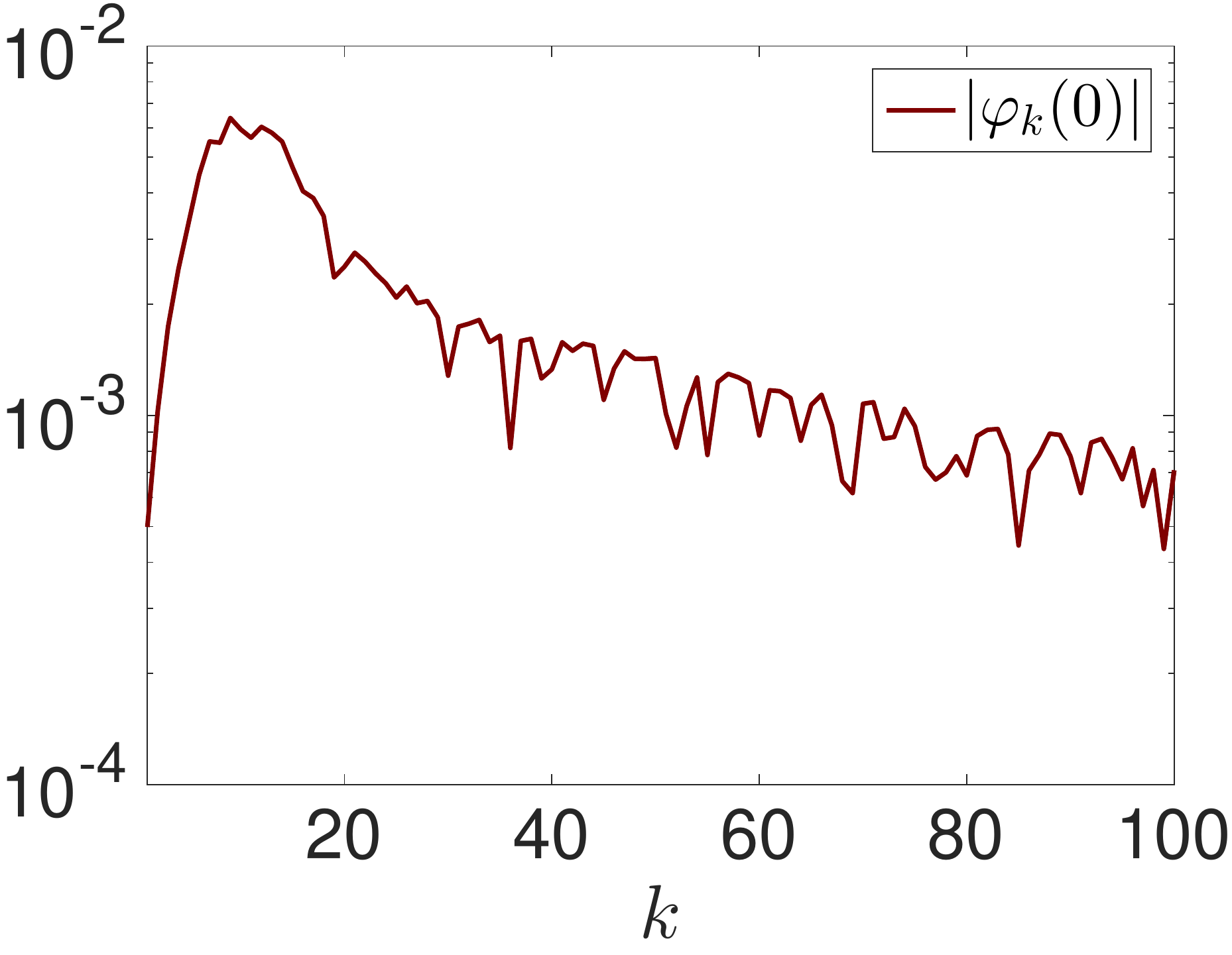}
        \end{subfigure}
        \begin{subfigure}[b]{0.32\textwidth}
                \includegraphics[width = .95\textwidth]{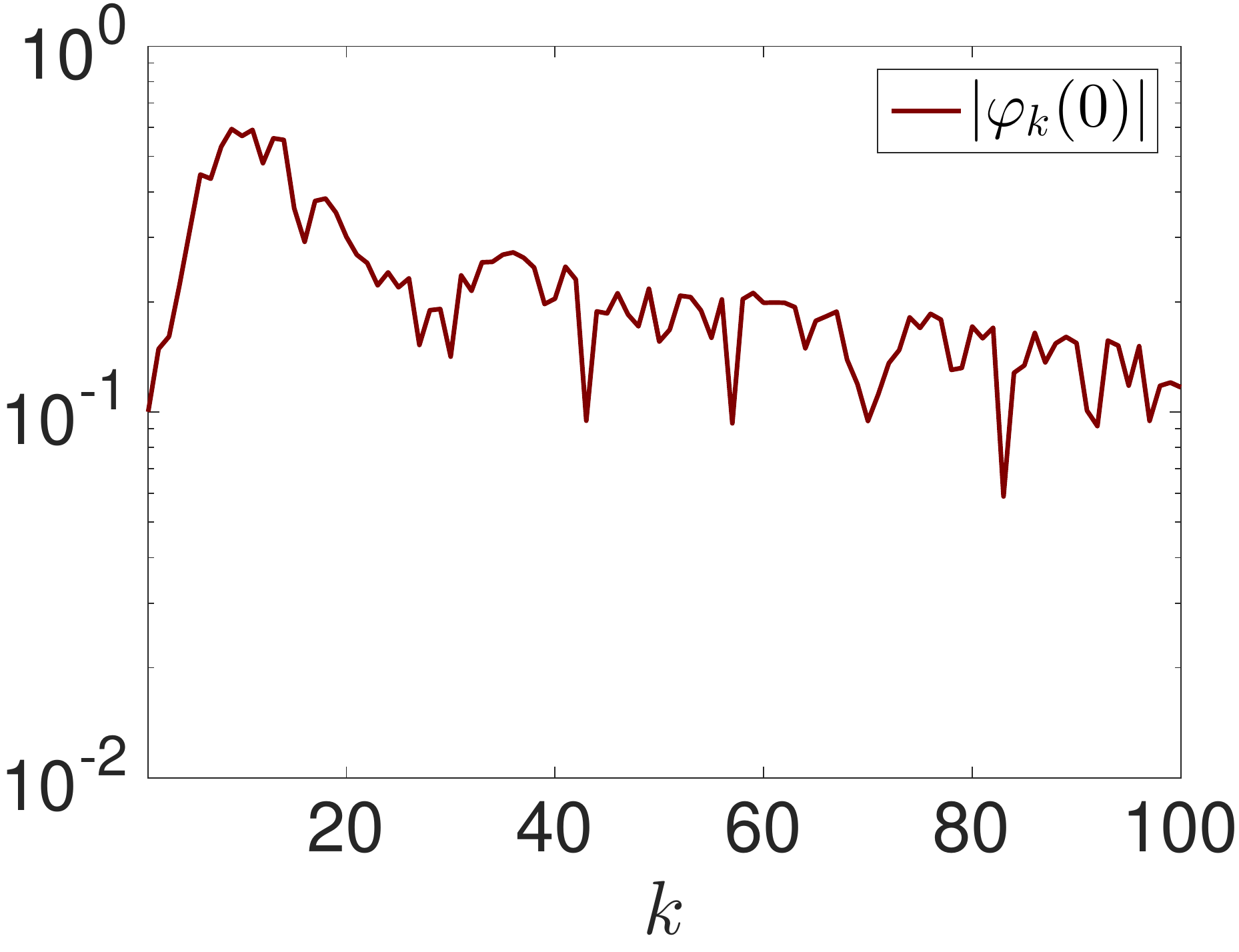}
        \end{subfigure}
        
        \begin{subfigure}[b]{0.32\textwidth}
                \includegraphics[width = .95\textwidth]{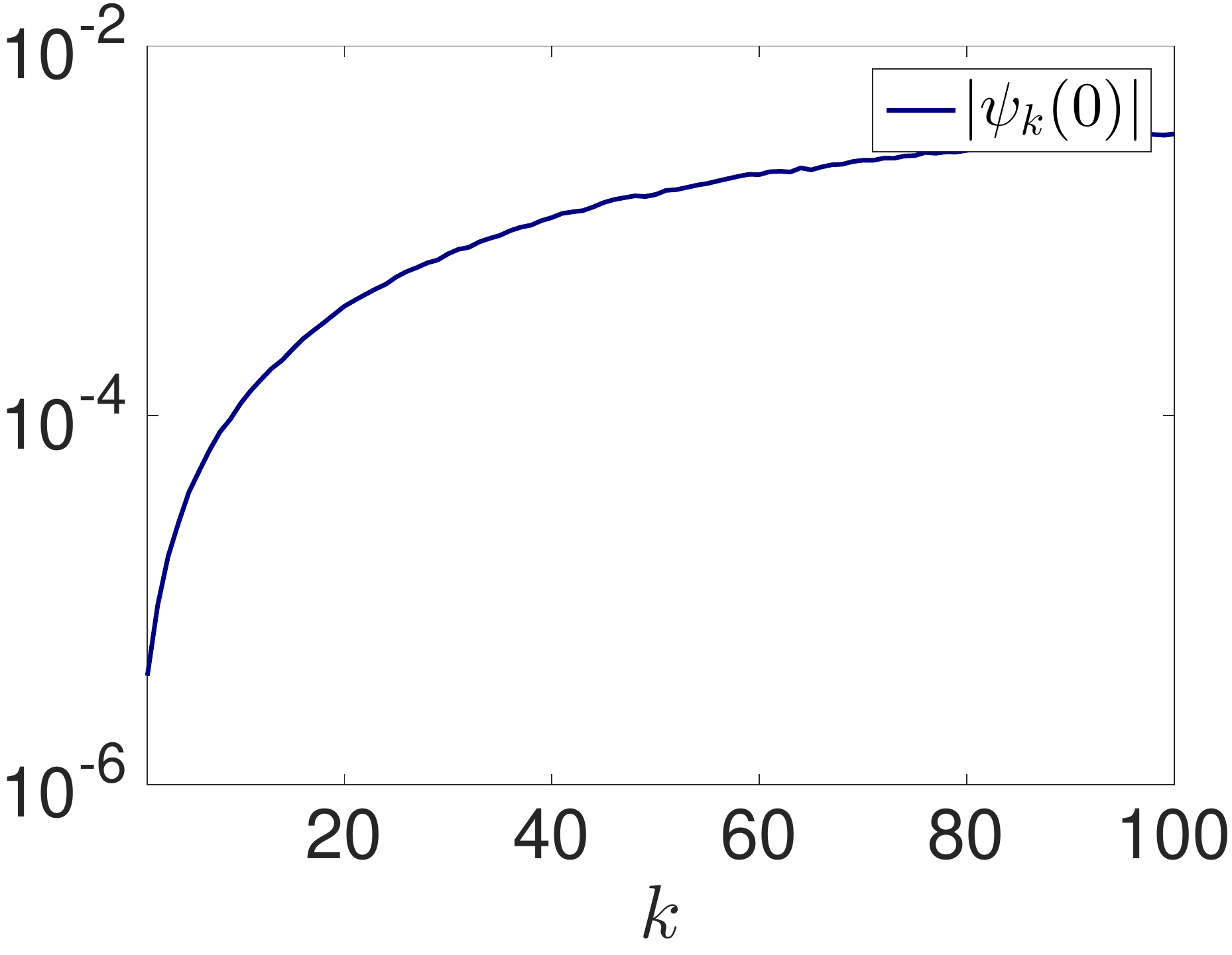}
                \caption{\texttt{vargaussianblur}}
        \end{subfigure}
        \begin{subfigure}[b]{0.32\textwidth}
                \includegraphics[width = .95\textwidth]{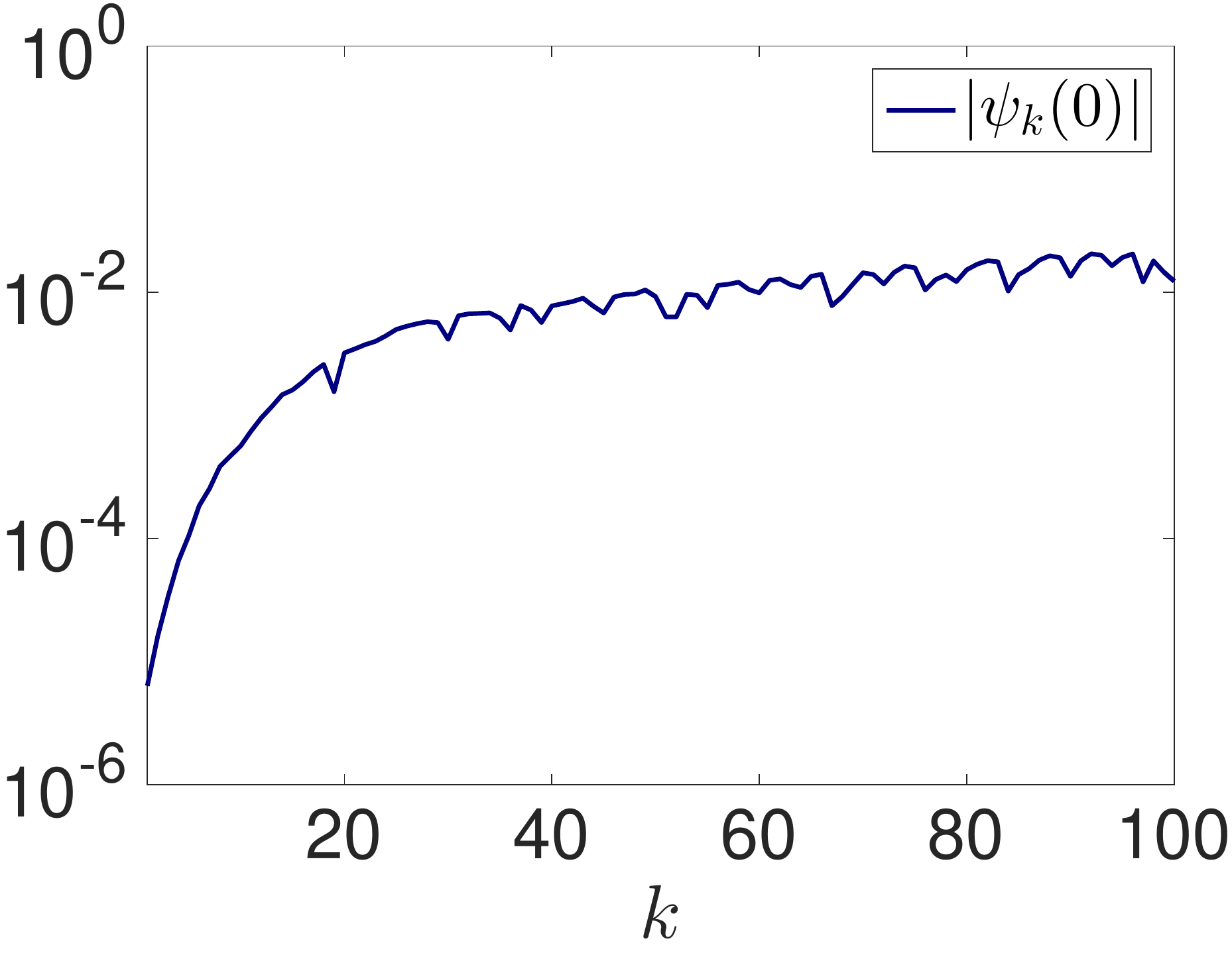}
                \caption{\texttt{seismictomo}}\label{fig:factor_seismic}
        \end{subfigure}
        \begin{subfigure}[b]{0.32\textwidth}
                \includegraphics[width = .95\textwidth]{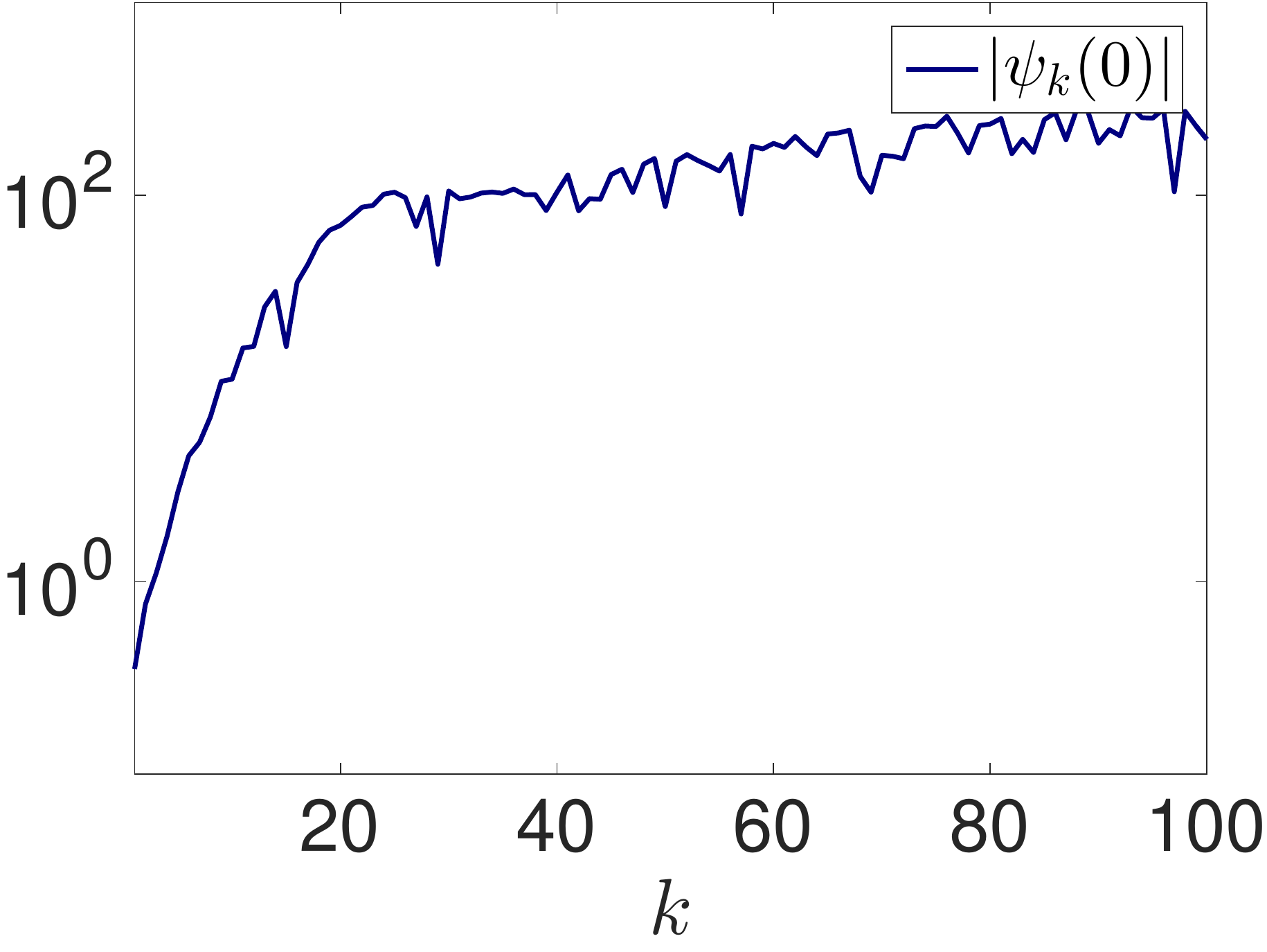}
                \caption{\texttt{paralleltomo} }
        \end{subfigure}
        
        \caption{The size of the absolute term of the Lanczos polynomials $\varphi_k$ and $\psi_k$ for selected 2D problems contaminated by noise as 
				  described in the text. For all problems $\delta_\text{noise} \approx 10^{-2}$. Computed without reorthogonalization.}\label{fig:2D_factors}
\end{figure}

Figure~\ref{fig:2D_vectors} shows several (appropriately reshaped) left bidiagonalization
vectors $s_k$ and their cumulative periodograms for the problem \texttt{seismictomo}. 
Even though it is hard to make clear conclusions based on the vectors $s_k$ themselves, we see that the periodogram 
for $k=10$ is flatter than the periodograms for smaller or larger values of $k$, meaning that $s_{10}$ resembles most white noise.
This corresponds to Figure~\ref{fig:factor_seismic} showing that $s_{10}$ belongs to the noise revealing phase of the bidiagonalization process.
Note that similar flatter periodograms can be obtained for other few vectors belonging to this phase.

\begin{figure}
\centering
        \begin{subfigure}[b]{\textwidth}
                \includegraphics[width = .95\textwidth]{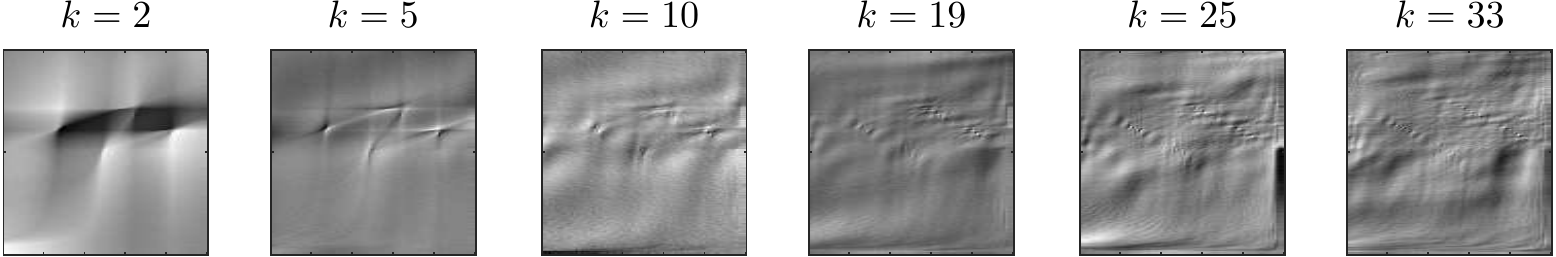}
                \caption{left bidiagonalization vectors $s_k$ (reshaped)}\label{fig:2D_bidvect}
        \end{subfigure}
        \begin{subfigure}[b]{\textwidth}
                \includegraphics[width = .95\textwidth]{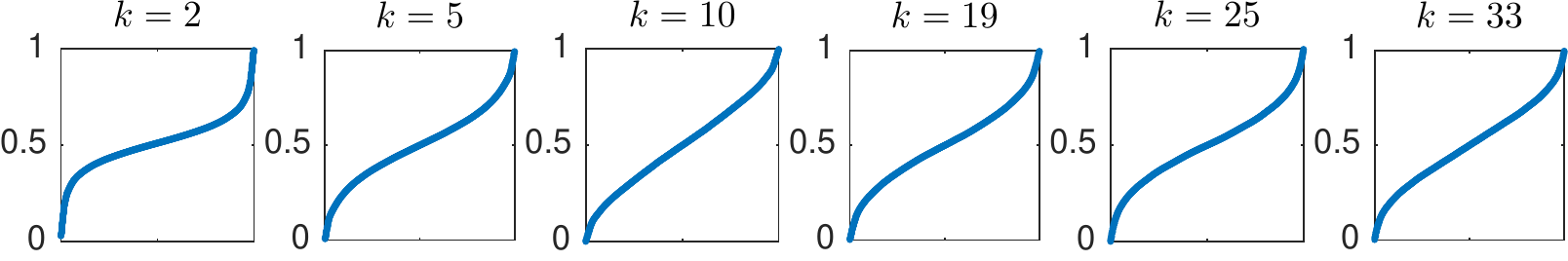}
                \caption{cumulative periodograms of $s_k$}\label{fig:2D_periodograms}
        \end{subfigure}
\caption{Left bidiagonalization vectors $s_k$ for the problem \texttt{seismictomo} and their cumulative periodograms. The periodogram of the vector 
$s_{10}$ belonging to the noise revealing phase of the bidiagonalization process is flatter. Computed without reorthogonalization.}\label{fig:2D_vectors}
\end{figure}

The absence of one particular noise revealing vector makes the direct comparison between $s_k$ and the exact noise vector $\eta$ irrelevant here.
However, Propositions~\ref{th:1}--\ref{th:3} remain valid and the overall behavior of the terms $|\varphi_k(0)|$ and $|\psi_k(0)|$ is as expected,
allowing comparing the bidiagonalization-based methods. 
Figure~\ref{fig:2D_errors} gives comparisons of CRAIG, LSQR and LSMR for all considered 2D test problems,
analogous to Figure~\ref{fig:error_all}, \ref{fig:error_craig_lsqr}, and \ref{fig:error_lsqr_lsmr}.
The first row of Figure~\ref{fig:2D_errors} shows that the CRAIG error is minimized approximately in 
the noise revealing phase, i.e., when the residual is minimal, see Section~\ref{sec:relation_CRAIG}.
The minimum is emphasized by the vertical line.

The second row of Figure~\ref{fig:2D_errors} compares the errors of CRAIG and LSQR. According to the derivations
in Section~\ref{sec:residuals}, the curves are similar before the noise revealing phase, after which
they separate with CRAIG diverging more quickly. Note that the size of the inverted amplification factor $\varphi_k(0)$
is included to illustrate the noise revealing phase and has different scaling (specified on the right).

The third row of Figure~\ref{fig:2D_errors} shows the errors of LSQR and LSMR with the underlying size of the inverted
factor $\psi_k(0)$ (scaling specified on the right). The errors behave similarly as long as $|\psi_k(0)|^{-1}$ decays
rapidly, see Section~\ref{ssec:LSMRres}. The LSMR solution is slightly less sensitive to the particular choice of the number of bidiagonalization
iterations $k$, which is a well know property \cite{Fong2011LSMR}.

\begin{figure}
        \begin{subfigure}[b]{0.32\textwidth}
                \includegraphics[width = .95\textwidth]{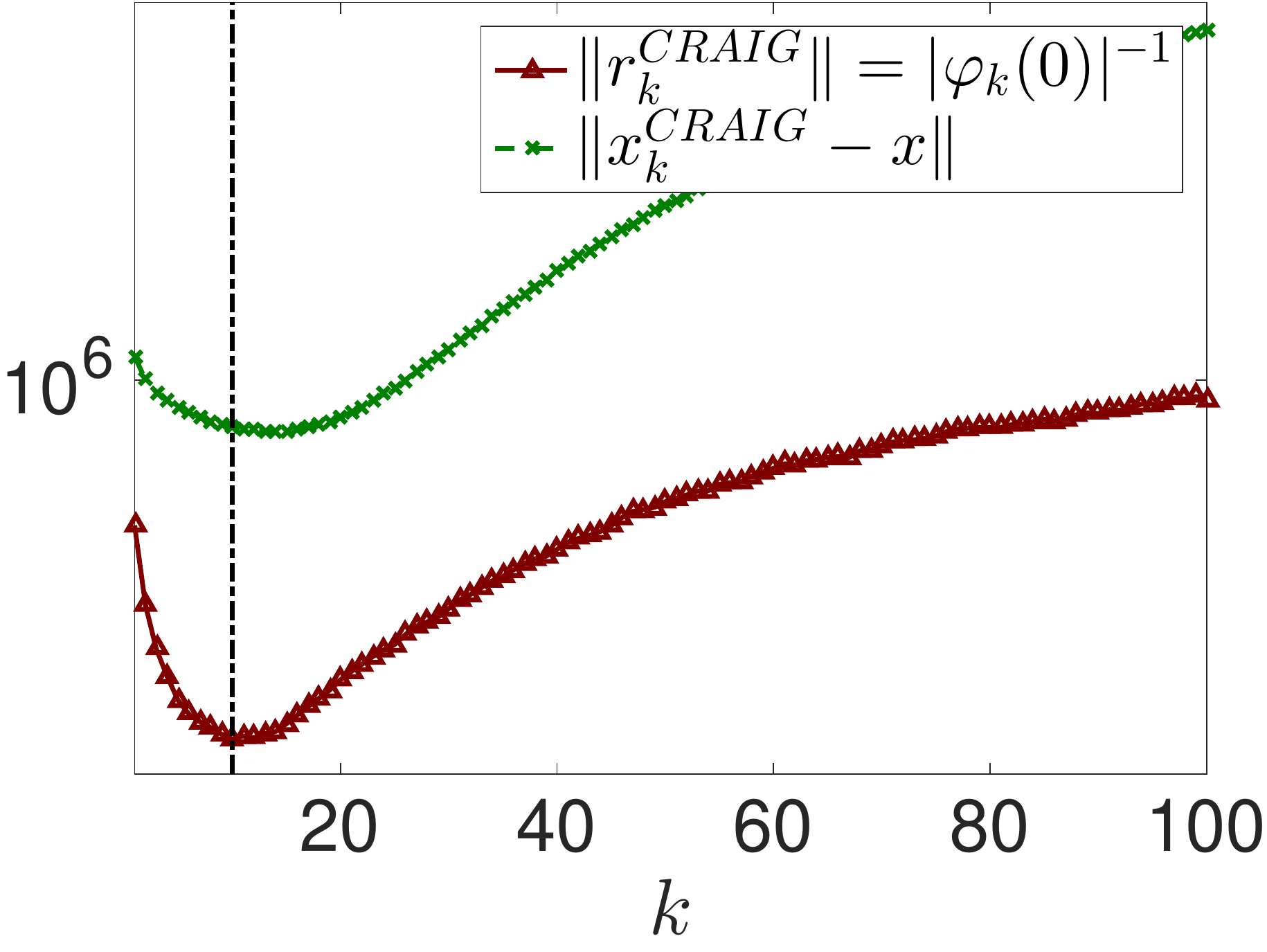}
        \end{subfigure}
        \begin{subfigure}[b]{0.32\textwidth}
                \includegraphics[width = .95\textwidth]{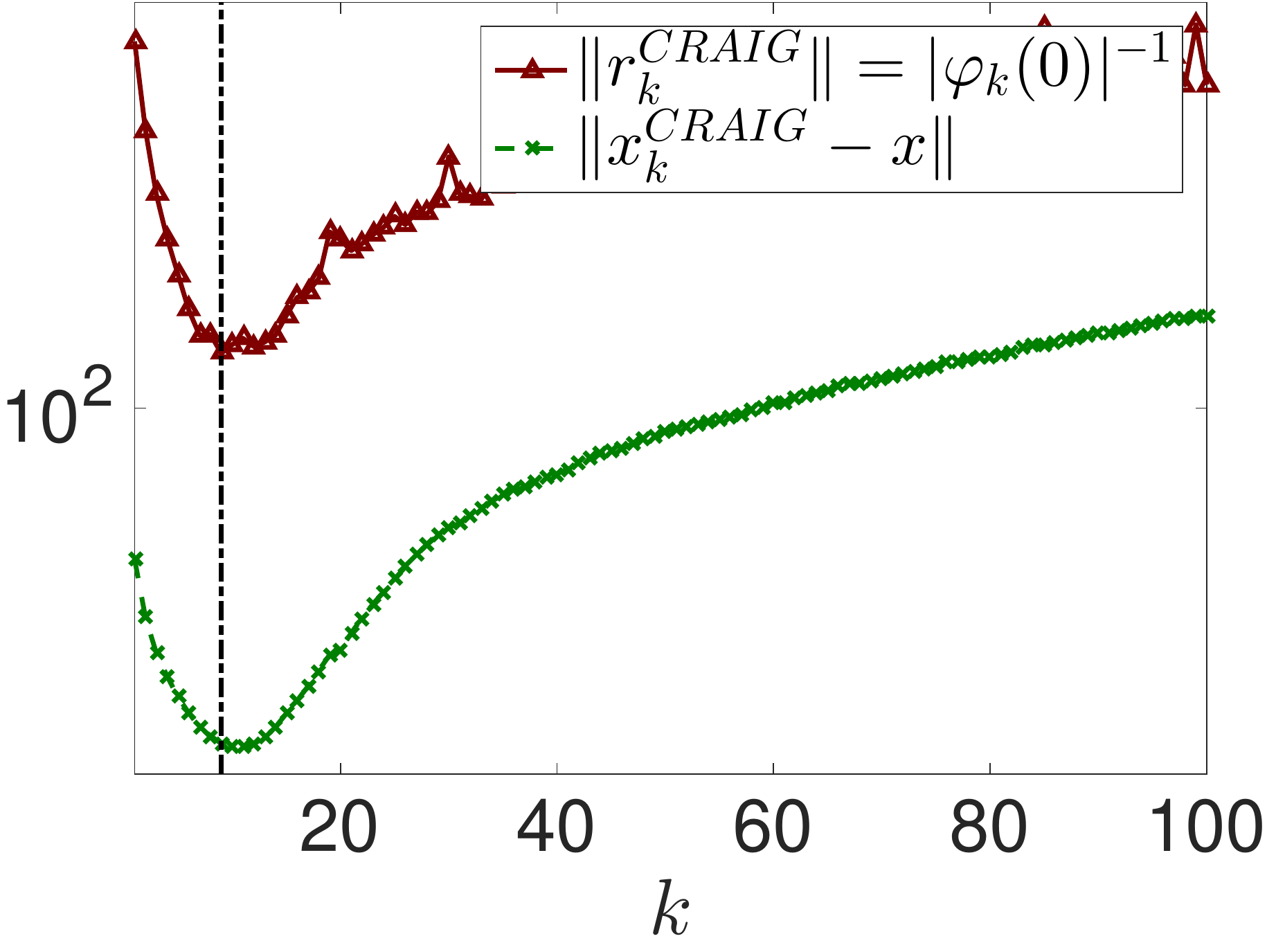}
        \end{subfigure}
        \begin{subfigure}[b]{0.32\textwidth}
                \includegraphics[width = .95\textwidth]{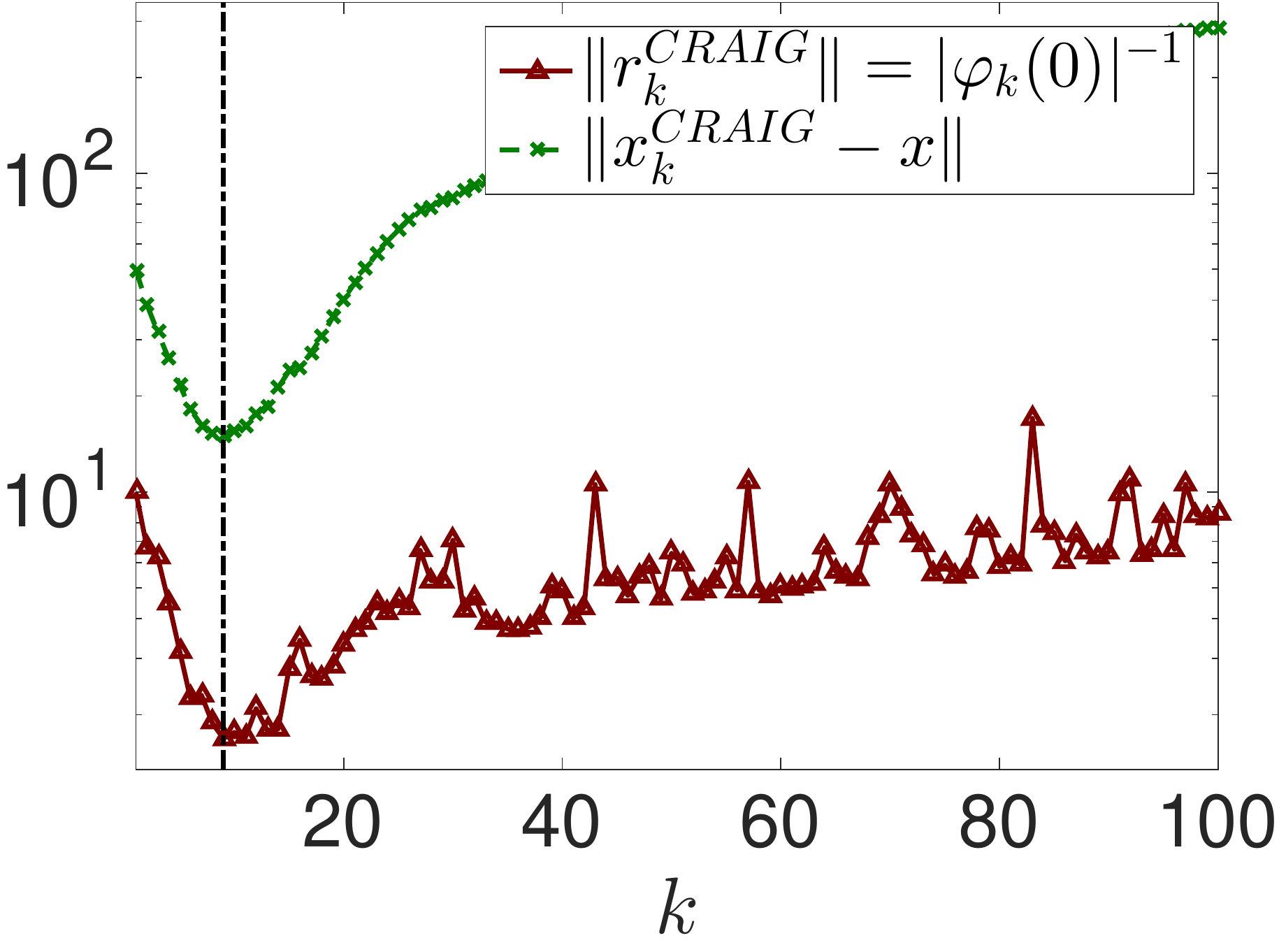}
        \end{subfigure}
        
        \begin{subfigure}[b]{0.32\textwidth}
                \includegraphics[width = .95\textwidth,trim={1.4cm 5cm .4cm 7cm},clip]{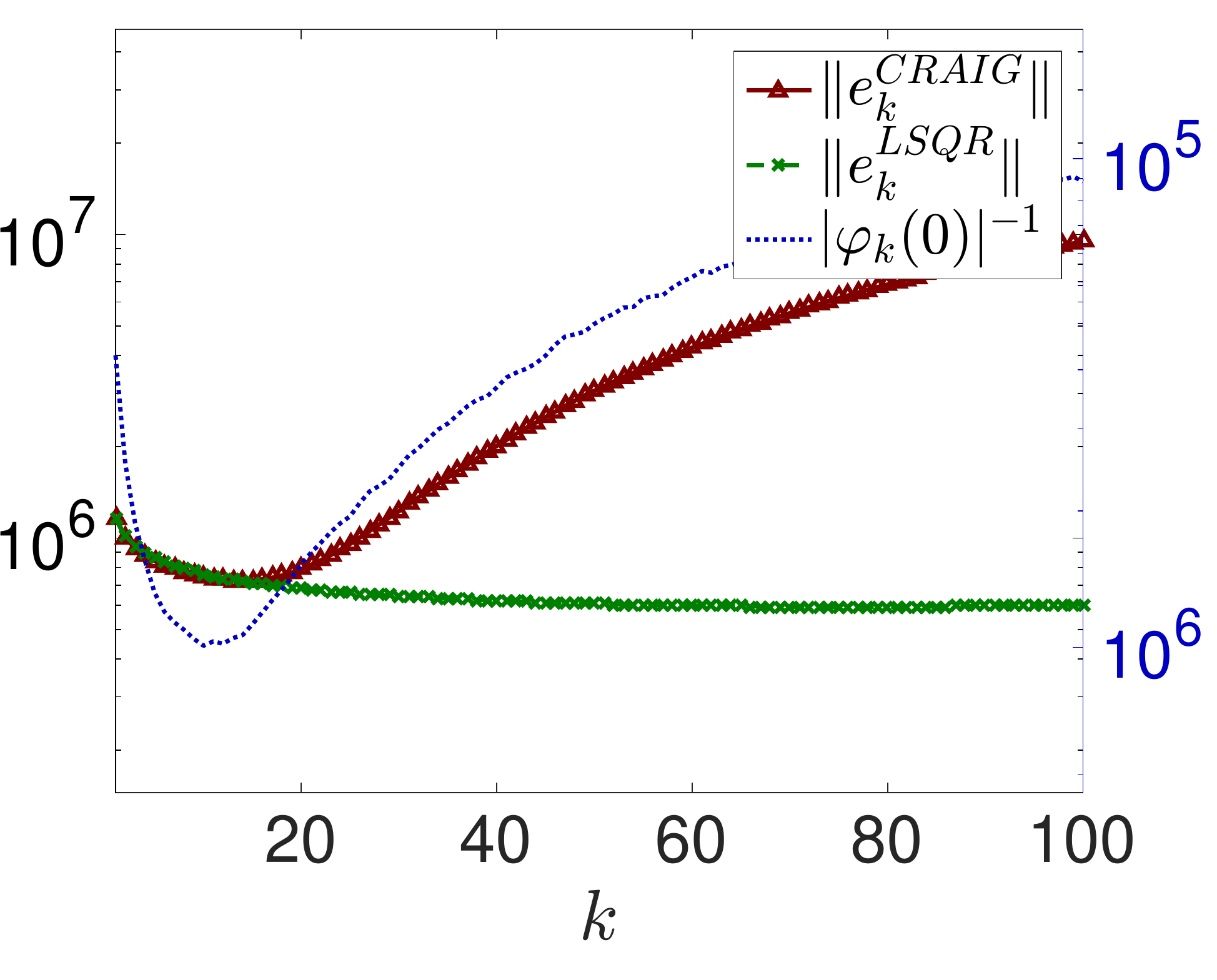}
        \end{subfigure}
        \begin{subfigure}[b]{0.32\textwidth}
                \includegraphics[width = .95\textwidth,trim={1.4cm 5cm .4cm 7cm},clip]{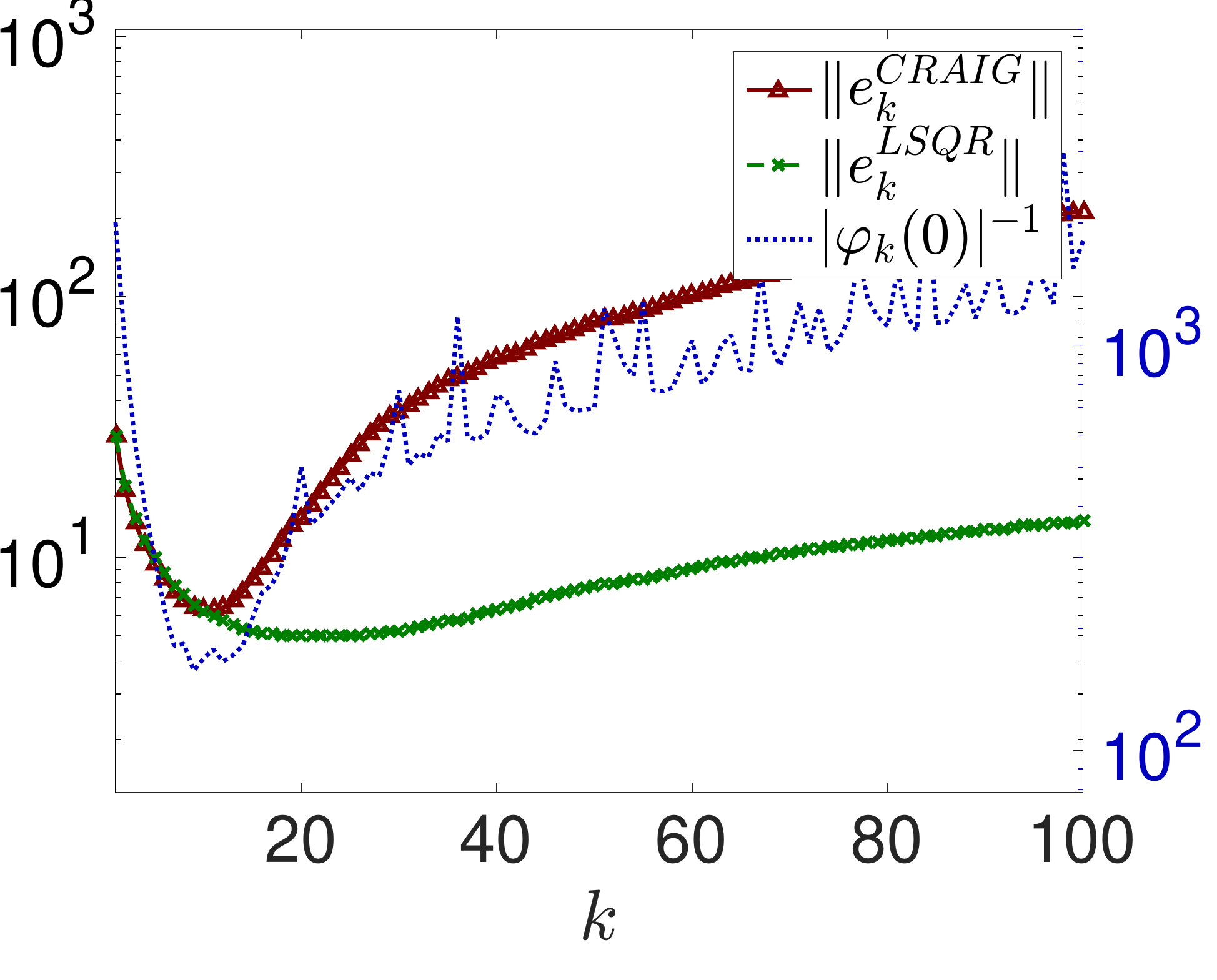}
        \end{subfigure}
        \begin{subfigure}[b]{0.32\textwidth}
                \includegraphics[width = .95\textwidth,trim={1.4cm 5cm .4cm 7cm},clip]{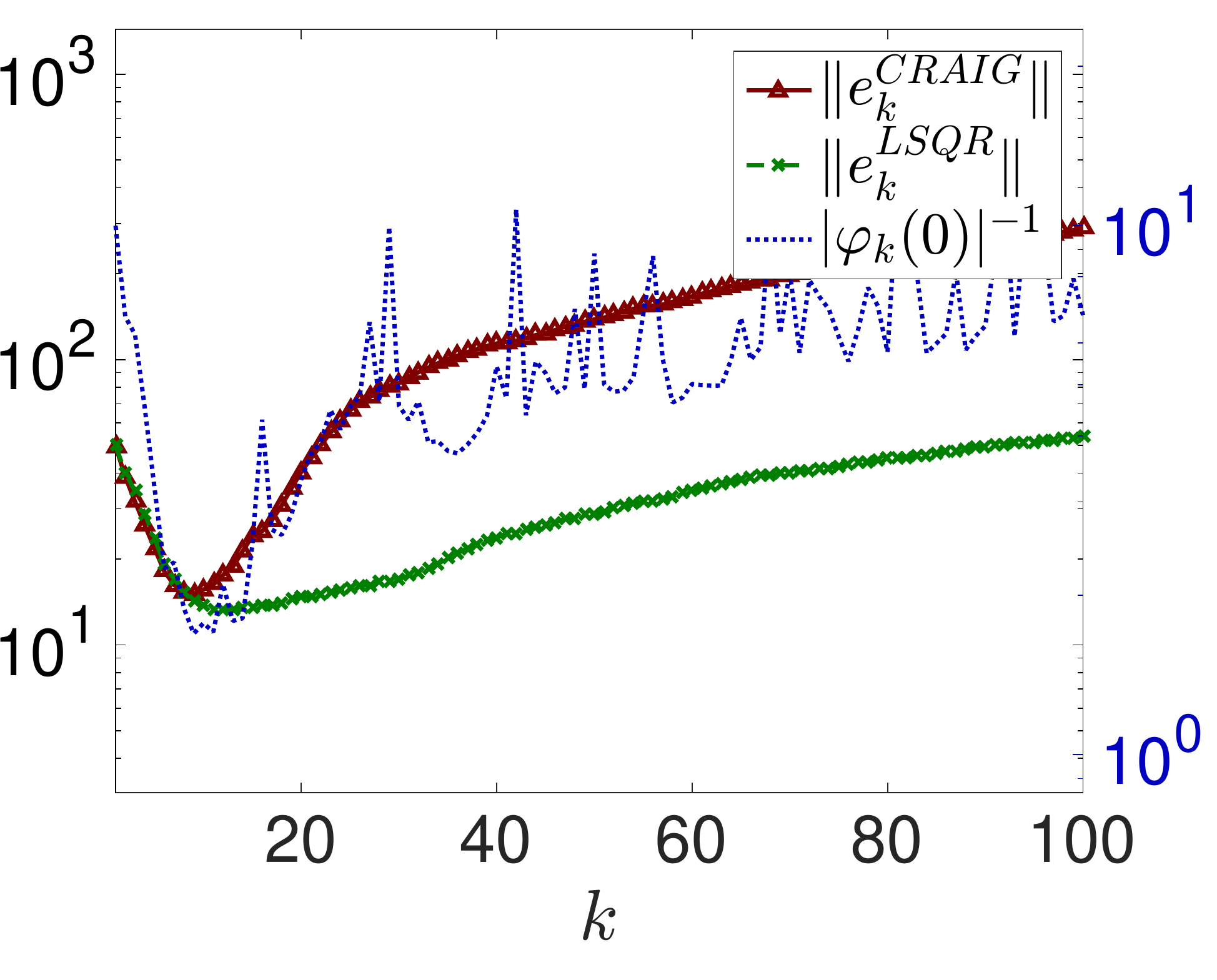}
        \end{subfigure}        
        
        \begin{subfigure}[b]{0.32\textwidth}
                \includegraphics[width = .95\textwidth,trim={1.4cm 5cm .4cm 7cm},clip] {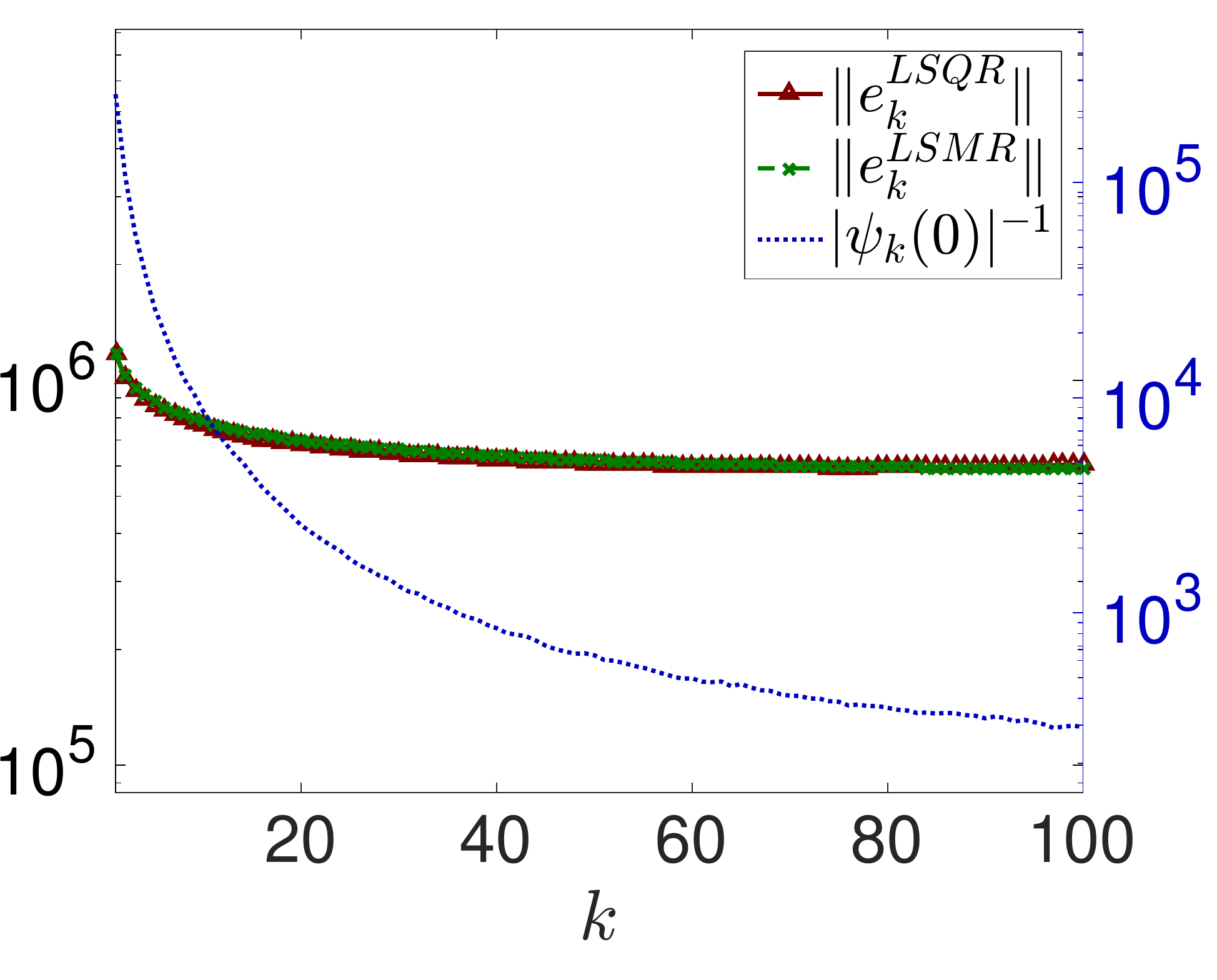}
                \caption{\texttt{vargaussianblur}}
        \end{subfigure}
        \begin{subfigure}[b]{0.32\textwidth}
                \includegraphics[width = .95\textwidth,trim={1.4cm 5cm .4cm 7cm},clip]{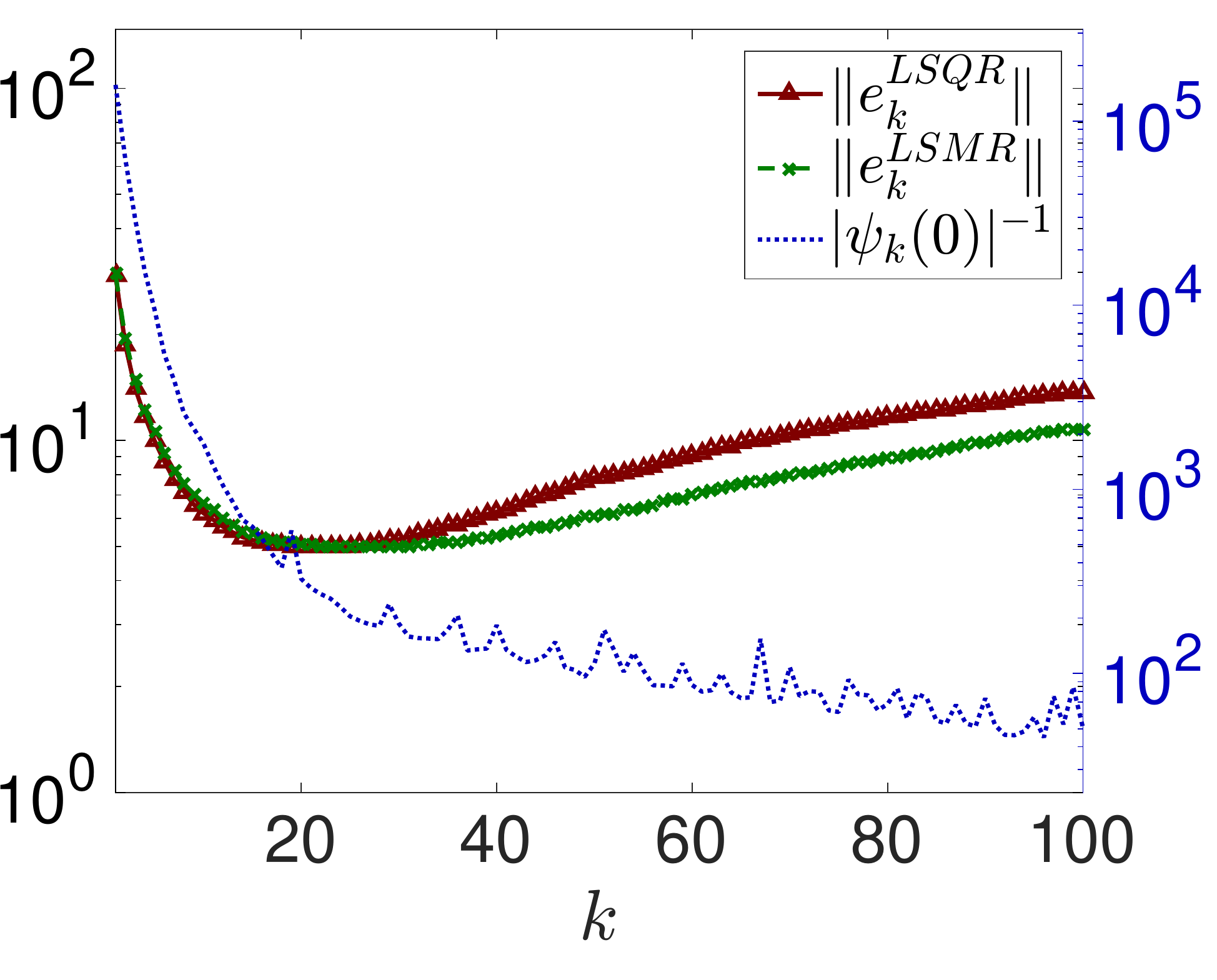}
                \caption{\texttt{seismictomo}}
        \end{subfigure}
        \begin{subfigure}[b]{0.32\textwidth}
                \includegraphics[width = .95\textwidth,trim={1.4cm 5cm .4cm 7cm},clip]{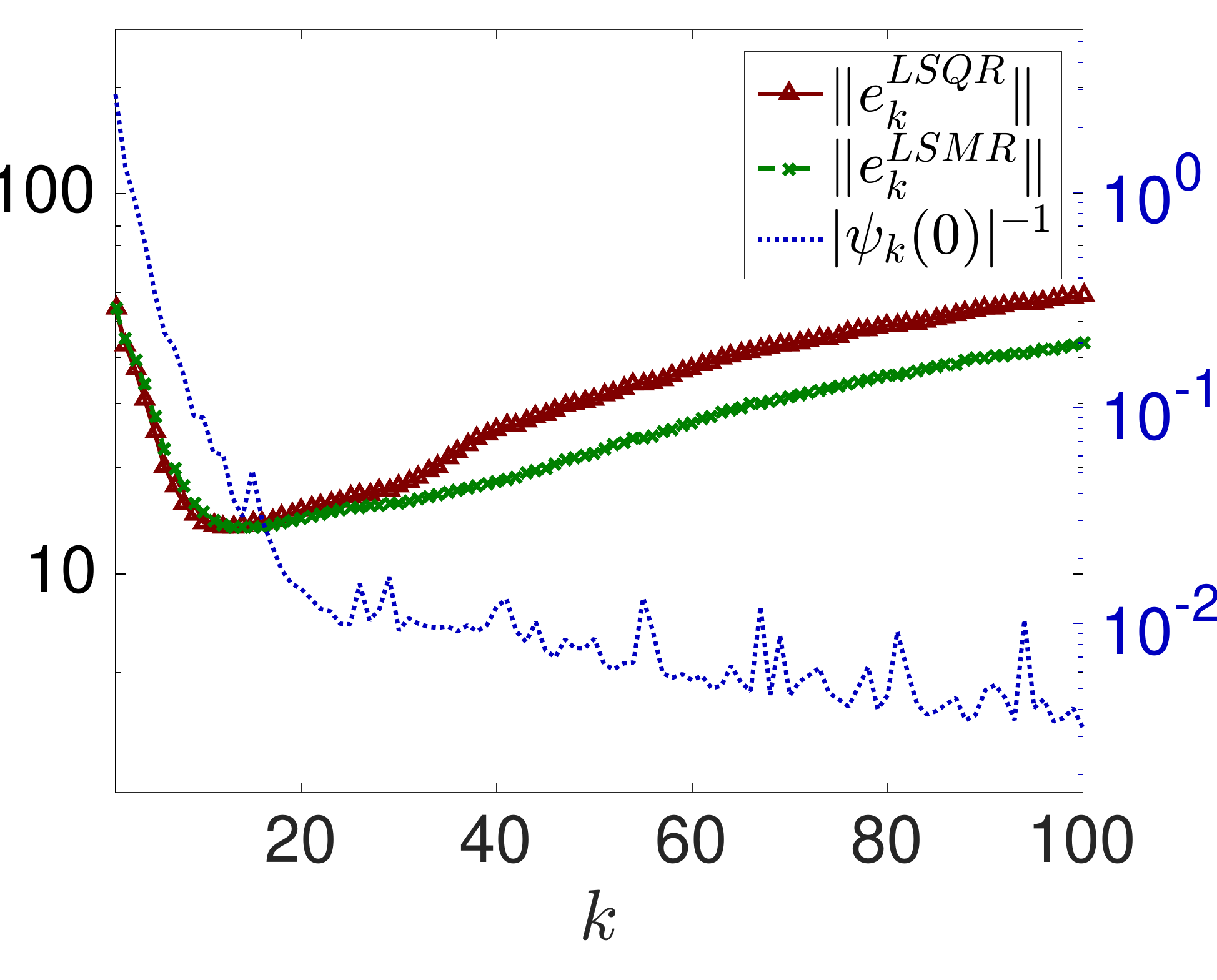}
                \caption{\texttt{paralleltomo}}
        \end{subfigure}
        
        \caption{First row: The size of the residual and the size of the error in CRAIG. Vertical line illustrates the minimum. 
        Second row: The size of the error of CRAIG and LSQR, together with the rescaled inverse of the amplification factor $\varphi_k(0)$
        (vertical scale on the right). Third row: The size of the error of LSQR and CRAIG, together with the rescaled inverse of the factor $\psi_k(0)$ (vertical scale on the right). Computed without reorthogonalization.}\label{fig:2D_errors}
\end{figure}

\section{Conclusion}\label{sec:conclusions} 
We proved that approximating the solution of an inverse problem by the $k$th iterate of CRAIG is mathematically equivalent to solving consistent linear algebraic problem with the same matrix and a right-hand side, where a particular (typically high-frequency) part of noise is removed. Using the analysis of noise propagation, we showed that the size of the CRAIG residual is given by the inverted noise amplification factor, which explains why optimal regularization properties are often obtained when the minimal residual is reached. For LSQR and LSMR, the residual is a linear combination of the left bidiagonalization vectors. The representation of these vectors in the residuals is determined by the amplification factor, in particular, left bidiagonalization vectors with larger amount of propagated noise are on average represented with a larger coefficient in both methods. These results were used in 1D problems to compare the methods in terms of matching between the residuals and the unknown noise vector. For large 2D (or 3D) problems 
the direct comparison of the vectors may not be possible, since noise reveals itself in a few subsequent bidiagonalization vectors (noise revealing phase of bidiagonalization) instead of in one particular iteration. However, the conclusions on the methods themselves remain generally valid.
Presented results contribute to understanding of the behavior of the methods when solving noise-contaminated inverse problems.

\section*{Acknowledgment}
Research supported in part by the Grant Agency of the Czech Republic under the grant 17-04150J.
Work of the first and the second author supported in part by Charles University, project GAUK~196216.
The authors are grateful to the anonymous referee for useful suggestions and comments that improved the presentation of the paper.

\end{document}